\documentclass{amsart}

\usepackage{amssymb}
\usepackage{amsthm}
\usepackage{amsmath}

\usepackage{enumitem}
\usepackage[normalem]{ulem}

\usepackage[usenames]{xcolor}

\usepackage[hyperindex,breaklinks]{hyperref}

\usepackage[capitalise]{cleveref}
\usepackage{graphicx}

\theoremstyle{plain}
\newtheorem{proposition}{Proposition}[section]
\newtheorem{theorem}[proposition]{Theorem}
\newtheorem{lemma}[proposition]{Lemma}
\newtheorem{corollary}[proposition]{Corollary}
\theoremstyle{definition}
\newtheorem{example}[proposition]{Example}
\newtheorem{definition}[proposition]{Definition}
\newtheorem{observation}[proposition]{Observation}
\theoremstyle{remark}
\newtheorem{remark}[proposition]{Remark}

\DeclareMathOperator{\Aut}{Aut}

\DeclareMathOperator{\diam}{diam}

\DeclareMathOperator{\SL}{SL}
\DeclareMathOperator{\GL}{GL}

\DeclareMathOperator{\PGL}{PGL}

\DeclareMathOperator{\End}{End}

\DeclareMathOperator{\Spanset}{Span} 
\DeclareMathOperator{\Gr}{Gr} 
\DeclareMathOperator{\id}{id} 
\DeclareMathOperator{\Haus}{Haus} 
\DeclareMathOperator{\CAT}{CAT} 
\DeclareMathOperator{\Isom}{Isom}

\DeclareMathOperator{\Stab}{Stab}

\DeclareMathOperator{\relint}{relint}
\DeclareMathOperator{\Span}{Span}
\DeclareMathOperator{\Min}{Min}
\DeclareMathOperator{\ConvHull}{ConvHull}

\DeclareMathOperator{\dist}{d}
\DeclareMathOperator{\partiali}{\partial_i}
\DeclareMathOperator{\partialn}{\partial_n}

\DeclareMathOperator{\ev}{\lambda}
\DeclareMathOperator{\hil}{d_{\Omega}}

\DeclareMathOperator{\Cc}{\mathcal{C}}

\DeclareMathOperator{\Ec}{\mathcal{E}}
\DeclareMathOperator{\Fc}{\mathcal{F}}
\DeclareMathOperator{\Gc}{\mathcal{G}}

\DeclareMathOperator{\Lc}{\mathcal{L}}
\DeclareMathOperator{\Nc}{\mathcal{N}}

\DeclareMathOperator{\Sc}{\mathcal{S}}
\DeclareMathOperator{\Tc}{\mathcal{T}}

\DeclareMathOperator{\Ab}{\mathbb{A}}
\DeclareMathOperator{\Cb}{\mathbb{C}}
\DeclareMathOperator{\Db}{\mathbb{D}}

\DeclareMathOperator{\Hb}{\mathbb{H}}

\DeclareMathOperator{\Nb}{\mathbb{N}}
\DeclareMathOperator{\Pb}{\mathbb{P}}
\DeclareMathOperator{\Rb}{\mathbb{R}}
\DeclareMathOperator{\Sb}{\mathbb{S}}

\DeclareMathOperator{\Zb}{\mathbb{Z}}

\newcommand{\abs}[1]{\left|#1\right|}

\newcommand{\wt}[1]{\widetilde{#1}}
\newcommand{\wh}[1]{\widehat{#1}}
\newcommand{\ip}[1]{\left\langle #1\right\rangle}

\hypersetup{
    colorlinks=true,
    linkcolor=blue,
    filecolor=blue,      
    urlcolor=blue,
    citecolor=blue,
    }

\begin{document}

\title{Convex co-compact representations of 3-manifold groups}
\author{Mitul Islam}
\address{\emph{Current Address:} Max Planck Institute for Mathematics in the Sciences, 04103 Leipzig}
\address{Department of Mathematics, University of Michigan, Ann Arbor, MI 48109.}
\email{mitul.islam@mis.mpg.de}
\author{Andrew Zimmer}\address{Department of Mathematics, University of Wisconsin-Madison, Madison, WI 53706.}
\email{amzimmer2@wisc.edu}
\date{\today}
\subjclass[2010]{}

\begin{abstract} A representation of a finitely generated group into the projective general  linear group is called convex co-compact if it has finite kernel and its image acts convex co-compactly on a properly convex domain in real projective space. We prove that the fundamental group of a closed irreducible orientable 3-manifold can admit such a representation only when the manifold is geometric (with Euclidean, Hyperbolic, or Euclidean $\times$ Hyperbolic geometry) or when every component in the geometric decomposition is hyperbolic.  In each case, we describe the structure of such examples. 
\end{abstract}

\maketitle

\tableofcontents

\section{Introduction}

This paper is motivated by the theory of Anosov representations. These are representations of word hyperbolic groups into semisimple Lie groups with discrete image and very nice geometric properties. We are particularly interested in understanding how the definition of Anosov representations can be relaxed to admit representations of non-word hyperbolic groups while still preserving nice geometric properties. 

Recently, Kapovich--Leeb~\cite{KL2018} have introduced various notions of relative Anosov representations of relatively hyperbolic groups. Later, Zhu~\cite{Zhu2021} proposed another notion using the framework in~\cite{BPS2019}. The basic example of a group admitting these types of representations is a non-uniform lattice in $\Isom(\Hb^d)$.  Another class of representations into $\PGL_d(\Rb)$, called convex co-compact, was introduced by Danciger--Gu{\'e}ritaud--Kassel~\cite{DGK2017}. A convex co-compact representation of a word hyperbolic group is an Anosov representation, but non-word hyperbolic groups can also admit convex co-compact representations (see for instance \cite[Section 4]{B2006}, \cite[Section 4.4]{LM2022}, and \cite{BDL2018}).

In earlier work~\cite{IZ2019b}, we studied relatively hyperbolic groups admitting convex co-compact representations and showed that their images have similar structure to groups acting properly discontinuously and co-compactly on a $\CAT(0)$ space with isolated flats. In particular, convex co-compact representations of relatively hyperbolic groups are very different from relatively Anosov representations (in the sense of Kapovich--Leeb or Zhu).

In this paper we consider convex co-compact representations of 3-manifold groups. We will show that the fundamental group of a closed irreducible orientable 3-manifold can admit such a representation only when the manifold is either geometric (with $\Rb^3$, $\Hb^3$, or $\Rb \times \Hb^2$ geometry) or when every component in the geometric decomposition is hyperbolic. Hence in the non-geometric case, the fundamental group is relatively hyperbolic with respect to a collection of rank two Abelian subgroups. In this case we will show that the convex co-compact representation, like an Anosov representation, induces an equivariant embedding of the boundary of the group (in this case the Bowditch boundary). However, unlike an Anosov representation, the image is not into a flag manifold but instead into a naturally defined quotient. 

We now state the results of this paper more precisely. We begin with the definition of a convex co-compact representation.

A \emph{properly convex domain} is an open subset $\Omega$ of $\Pb(\Rb^d)$ which is a bounded convex subset of some affine chart.  The \emph{automorphism group} of a properly convex domain $\Omega\subset\Pb(\Rb^d)$ is defined to be
\begin{align*}
\Aut(\Omega) : = \{ g \in \PGL_d(\Rb) : g \Omega = \Omega\}.
\end{align*}
One can associate a natural (possibly empty) convex subset to any subgroup of $\Aut(\Omega)$. In particular, for a subgroup $\Gamma \leq \Aut(\Omega)$, the  \emph{full orbital limit set of $\Gamma$ in $\Omega$}, denoted by $\Lc_{\Omega}(\Gamma)$, is the set of all $x \in \partial \Omega$ where there exist $p \in \Omega$ and a sequence $(g_n)_{n \geq 1}$ in  $\Gamma$ such that $g_n(p) \rightarrow x$. Then let $\Cc_\Omega(\Gamma)$ denote the convex hull of $\Lc_\Omega(\Gamma)$ in $\Omega$. The \emph{ideal boundary} of $\Cc_\Omega(\Gamma)$ is the set $\partiali \Cc_\Omega(\Gamma) = \overline{\Cc_\Omega(\Gamma)} \cap \partial \Omega$.

\begin{definition}[{Danciger--Gu{\'e}ritaud--Kassel~\cite[Definition 1.10]{DGK2017}}] Suppose $\Omega \subset \Pb(\Rb^d)$ is a properly convex domain. An infinite discrete subgroup $\Gamma \leq \Aut(\Omega)$ is called \emph{convex co-compact} if $\Cc_\Omega(\Gamma)$ is non-empty and $\Gamma$ acts co-compactly on $\Cc_\Omega(\Gamma)$. 
\end{definition}

A convex co-compact representation is then defined as follows.

\begin{definition} Suppose $\Gamma$ is a finitely generated infinite group. A representation $\rho : \Gamma \rightarrow \PGL_d(\Rb)$ is \emph{convex co-compact} if $\ker \rho$ is finite and there exists a properly convex domain $\Omega \subset \Pb(\Rb^d)$ where $\rho(\Gamma) \leq \Aut(\Omega)$ is convex co-compact. 
\end{definition} 

When $\Gamma$ is word hyperbolic there is a close connection between this class of discrete groups in $\PGL_d(\Rb)$ and Anosov representations, see~\cite{DGK2017} for details and~\cite{DGK2018,Z2017} for related results.

Our main result describes the geometric decomposition of 3-manifolds whose fundamental groups can admit convex co-compact representations. Recall, that the geometric decomposition theorem says that given a closed irreducible orientable 3-manifold, one can remove a (possibly empty) collection of embedded tori and Klein bottles (that are unique up to isotopy), such that each of the resulting connected components supports one of the eight Thurston geometries, for more details see \Cref{sec:three_manifold_theory}. In terms of this decomposition, we will prove the following.

\begin{theorem}\label{thm:main}(see Section~\ref{sec:pf_main_thm} below) Suppose $M$ is a closed irreducible orientable 3-manifold. If $\rho : \pi_1(M) \rightarrow \PGL_d(\Rb)$ is a convex co-compact representation, then either
\begin{enumerate}
\item $M$ is geometric with geometry $\Rb^3$, $\Rb \times \Hb^2$, or $\Hb^3$,
\item $M$ is non-geometric and every component in the geometric decomposition is hyperbolic (i.e. has $\Hb^3$ geometry).
\end{enumerate}
\end{theorem}

Theorem~\ref{thm:main} provides a higher dimensional generalization of Benoist's~\cite{B2006} well known description of properly convex divisible domains in $\Pb(\Rb^4)$. Recall, a properly convex domain $\Omega \subset \Pb(\Rb^4)$ is called \emph{divisible} if there exists a discrete group $\Gamma \leq \Aut(\Omega)$ which acts properly discontinuously, freely, and co-compactly on $\Omega$. In the case when $d=4$, $M := \Gamma \backslash \Omega$ is a closed 3-manifold and the natural isomorphism $\pi_1(M) \rightarrow \Gamma$ is a convex co-compact representation into $\PGL_4(\Rb)$. Then Benoist's results can be deduced from Theorem~\ref{thm:main} and the structural results described in Section~\ref{sec:non_geom_structure} below.

Benoist's arguments in the $d=4$ case relies on the low dimensionality (see~\cite[Sections 2.3, 2.5]{B2006}) and also work of Morgan-Shalen~\cite{MS1988} about 3-manifold groups acting on $\Rb$-trees. In contrast, the main tool in our proof is the following structure theorem for centralizers in convex co-compact groups. 

\begin{theorem}\label{thm:center_intro}(see Theorem~\ref{thm:center} below) Suppose $\Omega \subset \Pb(\Rb^d)$ is a properly convex domain, $\Gamma \leq \Aut(\Omega)$ is convex co-compact, $A \leq \Gamma$ is an infinite Abelian subgroup, and $Z_{\Gamma}(A)$ is the centralizer of $A$ in $\Gamma$. If 
\begin{align*}
V := \Span \left\{ v \in \Rb^d\backslash \{0\} :  [v] \in \overline{\Cc_\Omega(\Gamma)} \text{ and } a [v] = [v] \text{ for all } a \in A \right\},
\end{align*}
then $\Omega \cap \Pb(V)$ is a non-empty $Z_{\Gamma}(A)$-invariant properly convex domain in $\Pb(V)$ and the quotient $Z_{\Gamma}(A) \backslash \Omega \cap \Pb(V)$ is compact. 
\end{theorem}

\begin{remark} Theorem~\ref{thm:center_intro} builds upon earlier work in~\cite{IZ2019} which considered centralizers of Abelian subgroups of so-called naive convex co-compact subgroups. This is a larger class of groups and the results in~\cite{IZ2019} had weaker conclusions. In Section~\ref{sec:failure of center theorem for naive convex co-compact groups} we provides examples of naive convex co-compact groups where the stronger conclusions of Theorem~\ref{thm:center_intro} fail to hold. \end{remark}

Theorem~\ref{thm:center_intro} implies that the centralizer of an infinite Abelian subgroup of a convex co-compact subgroup is virtually the fundamental group of a closed aspherical $k$-manifold where $k =\dim \Pb(V)$. This greatly restricts the groups that can admit convex co-compact representations.

Next we strengthen our classification result (Theorem \ref{thm:main}) by describing the structure of each of the four types of examples which can appear in Theorem~\ref{thm:main}. 
 In these results, properly embedded simplices feature prominently. These are defined precisely in Section~\ref{sec:cones and simplices}, but informally a set $S \subset \Pb(\Rb^d)$ is a $k$-dimensional simplex if it is the interior of the convex hull of $(k+1)$ linearly independent points in $\Pb(\Rb^d)$. A simplex $S \subset \Omega$ (more generally, any set $X \subset \Omega$) is properly embedded if the set  inclusion map is proper.

\subsection{Euclidean and Euclidean $\times$ hyperbolic manifolds} 

In these cases the representations are particularly simple. Delaying the definition of a properly convex cone until Section~\ref{sec:cones and simplices}, we will prove the following. 

\begin{proposition}\label{prop:boring_examples}(see Proposition~\ref{prop:geom_case} below) Suppose $M$ is a closed 3-manifold with $\Rb^3$ or $\Rb \times \Hb^2$ geometry. If $\rho : \pi_1(M) \rightarrow \PGL_d(\Rb)$ is a convex co-compact representation and $\Omega \subset \Pb(\Rb^d)$ is a properly convex domain where $\Gamma:=\rho(\pi_1(M)) \leq \Aut(\Omega)$ is convex co-compact, then there exists a four dimensional linear subspace $V \subset \Rb^d$ such that 
\begin{align*}
\Cc_\Omega(\Gamma) = \Omega \cap \Pb(V).
\end{align*}
Moreover, 
\begin{enumerate}
\item If $M$ has $\Rb^3$ geometry, then $\Cc_\Omega(\Gamma)$ is a properly embedded simplex in $\Omega$,
\item If $M$ has $\Rb \times \Hb^2$ geometry, then $\Cc_\Omega(\Gamma)$ is a properly embedded cone in $\Omega$ with strictly convex base. 
\end{enumerate}
\end{proposition}

\subsection{Hyperbolic manifolds} In the case when the manifold $M$ has $\Hb^3$ geometry we will show that the class of convex co-compact representations of $\pi_1(M)$ coincides with the class of projective Anosov representations. 

 We will give a precise definition of Anosov representations (into $\PGL_d(\Rb)$) in Section~\ref{sec:Anosov}, but informally: if $\Gamma$ is a word hyperbolic group with Gromov boundary $\partial_\infty \Gamma$, $G$ is a semisimple Lie group, and $P \leq G$ is a parabolic subgroup, then  a representation $\rho: \Gamma \rightarrow G$ is called $P$-Anosov if there exists an injective, continuous, $\rho$-equivariant map $\xi: \partial_\infty \Gamma \rightarrow G/P$ satisfying certain dynamical properties.

When $G=\PGL_d(\Rb)$ and $P_1$ is the stabilizer of a line the quotient $\PGL_d(\Rb)/P_1$ can be identified with the $(d-1)$-dimensional real projective space $\Pb(\Rb^d)$ and $P_1$-Anosov representations are often called \emph{projective Anosov}. 

For many word hyperbolic groups, including fundamental groups of closed real hyperbolic 3-manifolds, the class of convex co-compact and projective Anosov representations coincide. 

\begin{theorem}\label{thm:anosov_intro}(see Section~\ref{sec:Anosov} below) Suppose $\Gamma$ is a one-ended word hyperbolic group which is not commensurable to a surface group. If $\rho : \Gamma \rightarrow \PGL_d(\Rb)$ is a representation, then the following are equivalent:
\begin{enumerate}
\item  $\rho$ is convex co-compact 
\item $\rho$ is projective Anosov.
\end{enumerate}
In this case, if $\Omega \subset \Pb(\Rb^d)$ is a properly convex domain such that $\rho(\Gamma) \leq \Aut(\Omega)$ is convex co-compact, $\Cc:=\Cc_\Omega(\rho(\Gamma))$, and $\xi^{(1)}:\partial_\infty \Gamma \rightarrow \Pb(\Rb^d)$ is the Anosov boundary map, then
\begin{enumerate}[label={(\alph*)}]
\item $\xi^{(1)} \left( \partial_\infty \Gamma\right) = \partiali \Cc$,
\item $\partiali\Cc$ contains no non-trivial line segments, and
\item every point in $\partiali\Cc$ is a $\Cc^1$-smooth point of $\partial\Omega$ (i.e. admits a unique supporting hyperplane).
\end{enumerate}
\end{theorem} 

\begin{remark} The equivalence in Theorem~\ref{thm:anosov_intro} fails for surface groups (and hence free groups), for instance Danciger--Gu\'{e}ritaud--Kassel~\cite{DGK2017} observed that a Hitchin representation into $\PGL_{2d}(\Rb)$ never preserves a properly convex domain, but is projective Anosov by results of Labourie~\cite{L2006}. \end{remark} 

In the special case when $\rho$ is irreducible, Theorem~\ref{thm:anosov_intro} was established by the second author~\cite{Z2017} (using different terminology). In full generality, the $(1) \Rightarrow(2)$ direction and the ``in this case'' assertions are a consequence of a result of Danciger--Gu\'{e}ritaud--Kassel~\cite[Theorem 1.15]{DGK2017}. 

In Section~\ref{sec:Anosov}, we will explain how the argument in~\cite{Z2017} and a result from~\cite{DGK2017} can be used to prove that $(2) \Rightarrow(1)$ in the reducible case. Canary~\cite{canary-notes}, in his lecture notes on Anosov representations, also provided a proof of Theorem~\ref{thm:anosov_intro} along similar lines. 

In the context of this section, we should also mention recent work of Canary--Tsouvalas  and Tsouvalas.  Canary--Tsouvalas~\cite{CK2020} proved that any torsion-free word hyperbolic group that admits a projective Anosov representation into $\SL_4(\Rb)$ is isomorphic to the fundamental group of a compact (not necessarily closed) $\Hb^3$-manifold. Recent work of Tsouvalas~\cite{T2020} further explores the connection between convex co-compact and Anosov representations. 

\subsection{Non-geometric manifolds}\label{sec:non_geom_structure} 

We now describe the case when $M$ is a non-geometric closed irreducible orientable 3-manifold.

For the rest of this section, suppose $M$ is a non-geometric closed irreducible orientable 3-manifold and $\rho : \pi_1(M) \rightarrow \PGL_d(\Rb)$ is a convex co-compact representation. 

Then $\Gamma:=\rho(\pi_1(M))$ is a relatively hyperbolic group with respect to $\{ \rho(\pi_1(T)): T \in \Tc\}$, where $\Tc$ is a  collection of embedded tori  and Klein bottle in the geometric decomposition (see \Cref{thm:geom-decomp}) of $M$. Moreover, each $\pi_1(T)$ is virtually isomorphic to $\Zb^2$. This  follows from \Cref{thm:main} part (2) and Dahmani's~\cite{D2003} combination theorem, see \Cref{prop:non_geom_eg_rel_hyp_fund_gp} below for details.

Next, we will describe the structure of $\Cc_{\Omega}(\Gamma)$ as well as define a boundary map from the Bowditch boundary of $\pi_1(M)$ to a quotient of $\partiali \Cc_{\Omega}(\Gamma)$. To ease notation in the discussion that follows, let $\Cc:=\Cc_\Omega(\Gamma)$.

\subsubsection{The structure of the domain}\label{sec:structure_of_the_domain_intro} In previous work~\cite{IZ2019b}, we studied the structure of convex co-compact subgroups which are relatively hyperbolic with respect to a collection of subgroups which are virtually free Abelian groups of rank at least two. In this subsection we briefly describe some of the consequences of these results. For more detail, see Section~\ref{subsec:struct-C} below. 
 
Let $\Sc$ be the collection of \textbf{all} properly embedded simplices in $\Cc$ of dimension at least two. By Theorems 1.7 and 1.8 in~\cite{IZ2019b}, $\Sc$ has the following properties:
\begin{enumerate}
\item $(\Cc,\hil)$ is relatively hyperbolic with respect to $\Sc$.
\item $\Sc$ is closed and discrete in the local Hausdorff topology.
\item Every line segment in $\partiali\Cc$ is contained in the boundary of a simplex in $\Sc$.
\item If $x \in \partiali \Cc$ is not a $\Cc^1$-smooth point of $\partial \Omega$, then there exists  $S \in \Sc$ with $x \in \partial S$.
\end{enumerate}
Properties (3) and (4) should be compared to Properties (b) and (c) in Theorem~\ref{thm:anosov_intro}. 

Further, Theorem 1.7 in~\cite{IZ2019b} implies the following correspondence between simplices in $\Sc$ and Abelian subgroups of $\Gamma$:
\begin{itemize}
\item If $S \in \Sc$, then $S$ is two dimensional, $\Stab_{\Gamma}(S)$ acts co-compactly on $S$, and $\Stab_{\Gamma}(S)$ is virtually isomorphic to $\Zb^2$. 
\item If $A \leq \Gamma$ is an Abelian subgroup with rank at least two, then $A$ is virtually isomorphic to $\Zb^2$ and there exists a unique $S \in \Sc$ such that $A \leq \Stab_{\Gamma}(S)$.  
\end{itemize}

\subsubsection{Equivariant boundary maps} We will also establish an analogue of Property (a) in Theorem~\ref{thm:anosov_intro}. By a result of Leeb, we can assume that $M$ is a non-positively curved Riemannian manifold~\cite{L1995}. Let $\wt{M}$ be the universal cover of $M$ endowed with the Riemannian metric making the covering map $\wt{M} \rightarrow M$ a local isometry. Then let $\wt{M}(\infty)$ be the $\CAT(0)$-boundary of $\wt{M}$. By a result of Hruska--Kleiner~\cite{HK2005} this boundary is a group invariant of $\pi_1(M)$, more precisely: if $\pi_1(M)$ acts  properly discontinuously and co-compactly on a $\CAT(0)$ space $X$, then there exists an equivariant homeomorphism $\wt{M}(\infty) \rightarrow X(\infty)$. 

Based on the existence of boundary maps in the hyperbolic case, it seems reasonable to ask if there exists a $\rho$-equivariant homeomorphism between $\wt{M}(\infty)$ and $\partiali\Cc$. However there is an obvious obstruction: if $A \leq \pi_1(M)$ is isomorphic to $\Zb^2$, then $A$ stabilizes an isometrically embedded flat $F \subset \wt{M}$ and $\rho(A)$ stabilizes a properly embedded 2-simplex $S \subset \Cc$. Any  $\rho$-equivariant homeomorphism of $\wt{M}(\infty)$ and $\partiali\Cc$ would map $F(\infty)$ to $\partial S$. This is impossible since the action of $A$ on $F(\infty)$ is trivial while the action of $\rho(A)$ on $\partial S$ is not. Thus a $\rho$-equivariant homeomorphism cannot exist between $\wt{M}(\infty)$ and $\partiali \Cc$.

To overcome this obstruction we introduce the following quotients. Let $\wt{M}(\infty) / {\sim}$ denote the quotient of $\wt{M}(\infty)$ obtained by identifying points which are in the geodesic boundary of the same flat and let $\partiali\Cc / {\sim}$ denote the quotient of $\partiali\Cc$ obtained by identifying points which are in the boundary of the same simplex  in $\Sc$. 

A general result of Tran~\cite{T2013} says that the quotient $ \wt{M}(\infty) / {\sim}$ naturally identifies with the Bowditch boundary of $\Gamma$. We will then establish the following analogue of the Anosov representation boundary maps.

\begin{theorem}\label{thm:bd_maps}(see Theorem~\ref{thm:bd_extensions} below) There exists a $\rho$-equivariant homeomorphism 
\begin{align*}
 \wt{M}(\infty) / {\sim} \longrightarrow \partiali\Cc / {\sim}.
\end{align*}
\end{theorem}

We remark that the above theorem, in fact, holds more generally: whenever  $M$ is a compact non-positively curved Riemannian manifold with isolated flats and $\rho:\pi_1(M) \to \PGL_d(\Rb)$ is a convex co-compact representation, see \cref{rem:generalizes_to_nonpos_curv}.

\subsubsection{Dynamics} We will use Theorem~\ref{thm:bd_maps} to study two dynamical systems associated to a convex co-compact representation. The first is the action of the group on the ideal boundary.

\begin{theorem}\label{thm:minimal_intro}(see Theorem~\ref{thm:minimal} below) The action of $\Gamma$ on $\partiali\Cc$ is minimal.\end{theorem}

We also study a natural geodesic flow associated to a convex co-compact representation. Let $T^1 \Omega$ be the unit tangent bundle of $\Omega$ relative to the Hilbert infinitesimal metric \cite[Section 3.2]{B2004}. Given $v \in T^1 \Omega$, let $\gamma_v : \Rb \rightarrow \Omega$ be the projective line geodesic with $\gamma_v^\prime(0)=v$. The \emph{(projective line) geodesic flow} on $T^1 \Omega$ is then defined by
\begin{align*}
\phi_t &: T^1 \Omega \rightarrow T^1 \Omega\\
\phi_t &(v) = \gamma_v^\prime(t).
\end{align*}  
Associated to $\Gamma$ is a natural flow invariant subset of $T^1 \Omega$ defined by 
\begin{align*}
\Gc_\Omega(\Gamma) =\left\{ v \in T^1 \Omega : \pi_{\pm}(v) \in \partiali\Cc\right\}
\end{align*}
where 
\begin{align*}
\pi_{\pm}(v) = \lim_{t \rightarrow \pm \infty} \gamma_v(t) \in \partial\Omega.
\end{align*}
The geodesic flow descends to a flow on the compact quotient $\Gamma \backslash \Gc_\Omega(\Gamma)$ which we also denote by $\phi_t$. 

\begin{remark} If one applies this construction to a convex co-compact representation of a word hyperbolic group, then one obtains a model of Gromov's geodesic flow space (see~\cite[Section 3]{ZZ2019} for details). 
\end{remark}

We will prove the following. 

\begin{theorem}\label{thm:transitive_intro}(see Theorem~\ref{thm:top-trans-equiv-minimal} below) The geodesic flow on $\Gamma \backslash \Gc_\Omega(\Gamma)$ is topologically transitive. \end{theorem}

\begin{remark} Theorem~\ref{thm:transitive_intro} is a consequence of Theorem~\ref{thm:minimal_intro} and Theorem~\ref{thm:top-trans-equiv-minimal} below, which states that for a general convex co-compact subgroup, the action of group on the ideal boundary is minimal if and only if the geodesic flow on the quotient is topologically transitive. 
\end{remark} 

\begin{remark}
In the special case when $\Gamma$ is a strongly irreducible subgroup of $\PGL_d(\Rb)$, Theorem~\ref{thm:minimal_intro} combined with a recent result of Blayac~\cite{PB2020} shows that the geodesic flow is topologically mixing. 
\end{remark}

\subsection{Recent developments} In the time between when this paper first appeared on the arXiv and when it was accepted for publication, a number of related results have appeared. For the reader's convenience we briefly mention some of these developments.

\begin{enumerate}
\item \emph{Boundary maps:} A version of Theorem~\ref{thm:bd_maps} is now known to be true for every relatively hyperbolic (naive) convex co-compact group. In particular, very shortly after this paper appeared on the arXiv, Weisman~\cite{W2020} posted a preprint to the arXiv showing that: if $\Gamma \leq \Aut(\Omega)$ is a convex co-compact subgroup which is relatively hyperbolic and whose peripheral subgroups were also convex co-compact, then there is a natural equivariant homeomorphism between the Bowditch boundary and a quotient of $\partiali \Cc_\Omega(\Gamma)$. In recent work~\cite{IZ2022}, we prove that in the above setting, the peripheral subgroups are always convex co-compact. So Weisman's boundary extension result holds for any convex co-compact group which is relatively hyperbolic.  In ~\cite{IZ2022}, we also prove a similar boundary extension result for any naive convex co-compact group which is relatively hyperbolic (recall that naive convex co-compact groups are a strictly larger class than convex co-compact groups).

\item \emph{Dynamics of the geodesic flow:} There have been tremendous advances by Blayac~\cite{Bpl2021} and Blayac--Zhu~\cite{BZ2021} in understanding the ergodic theory of the geodesic flow on convex real projective manifolds. In the context of Theorem~\ref{thm:transitive_intro}, these results, when combined with Theorem~\ref{thm:minimal_intro}, imply that the geodesic flow on $\Gamma \backslash \Gc_\Omega(\Gamma)$ is mixing with respect to a natural Bowen-Margulis measure.

\item \emph{A new class of representations:} In very recent work, Weisman~\cite{W2022} defines a class of representations of relatively hyperbolic groups which contains both relatively Anosov representations and convex co-compact representations.

\end{enumerate}

\subsection*{Acknowledgements} The authors thank the referees for their careful reading of the paper and helpful suggestions. The authors also thank Richard Canary, Jeffrey Danciger, and Ralf Spatzier for helpful conversations. M. Islam also thanks Louisiana State University for hospitality during a visit where work on this project started. 

M. Islam was partially supported by the National Science Foundation under grant DMS-1607260 and A. Zimmer was partially supported by the National Science Foundation under grants DMS-1904099, DMS-2104381, and DMS-2105580.

\section{Preliminaries}

\subsection{Notation}\label{sec:notation in preliminary} In this section we fix some notations.

\begin{itemize}
\item If $V \subset \Rb^d$ is a non-zero linear subspace, we will let $\Pb(V) \subset \Pb(\Rb^d)$ denote its projectivization. In most other cases, we will use $[o]$ to denote the projective equivalence class of an object $o$, for instance: 
\begin{enumerate}
\item if $v \in \Rb^{d} \setminus \{0\}$, then $[v]$ denotes the image of $v$ in $\Pb(\Rb^{d})$, 
\item if $\phi \in \GL_{d}(\Rb)$, then $[\phi]$ denotes the image of $\phi$ in $\PGL_{d}(\Rb)$, and 
\item if $T \in \End(\Rb^{d}) \setminus\{0\}$, then $[T]$ denotes the image of $T$ in $\Pb(\End(\Rb^{d}))$. 
\end{enumerate}
We also identify $\Pb(\Rb^d) = \Gr_1(\Rb^d)$, so for instance: if $x \in \Pb(\Rb^d)$ and $V \subset \Rb^d$ is a linear subspace, then $x \in \Pb(V)$ if and only if $x \subset V$. 

\item Given a choice, we will always prefer a linear subspace over the projectivization of a linear subspace. For instance,
\begin{enumerate}
\item Given a non-empty subset $X \subset \Pb(\Rb^d)$ we will let $\Span X \subset \Rb^d$ denote the smallest vector subspace whose projectivization contains $X$, that is
$$
\Span X :=\Span \{ v \in \Rb^d \setminus \{0\} : [v] \in X \}.
$$
\item Given an element $T \in \Pb(\End(\Rb^d))$ we will view its image and kernel as subspaces of $\Rb^d$.  When we refer to the respective projectivized subspaces, we will write them as $\Pb({\rm image} ~T)$ and $\Pb(\ker T)$.
\end{enumerate}
\item If $g \in \PGL_d(\Rb)$, we will let 
\begin{align*}
\lambda_1(g) \geq \lambda_2(g) \geq \dots \geq \lambda_d(g)
\end{align*}
denote the absolute values of the eigenvalues of some (hence any) lift of $g$ to $\SL_d^{\pm}(\Rb):=\{ h \in \GL_d(\Rb) : \det h = \pm 1\}$. 
\item If $(X,\dist)$ is a metric space, $A \subset X$, and $r>0$, we will use the notation \[ \Nc(A;r):=\{ x \in X : \dist(x,a)<r \text{ for some } a \in A\}.\] Also, given non-empty subsets $A,B \subset X$ the \emph{Hausdorff pseudo-distance} between $A$ and $B$ is
\begin{align*}
\dist^{\Haus}(A,B) := \inf \left\{ r > 0 : B \subset \Nc(A;r) \text{ and } A \subset \Nc(B;r) \right\}.
\end{align*}
\end{itemize}

\subsection{Convexity}

A subset $C \subset \Pb(\Rb^d)$ is \emph{convex} (respectively \emph{properly convex}) if there exists an affine chart $\mathbb{A}$ of $\Pb(\Rb^d)$ where $C \subset \mathbb{A}$ is a convex subset (respectively a bounded convex subset). Notice that if $C \subset \Pb(\Rb^d)$ is convex, then $C$ is a convex subset of every affine chart that contains it. When $C$ is a properly convex set which is open in $\Pb(\Rb^d)$ we say that $C$ is a \emph{properly convex domain}.

We also make the following topological definitions.

\begin{definition}\label{defn:topology} Suppose $C \subset \Pb(\Rb^d)$ is a properly convex set. The \emph{relative interior of $C$}, denoted by $\relint(C)$, is  the interior of $C$ in $\Pb(\Spanset C)$. In the case that $C = \relint(C)$, then $C$ is  \emph{open in its span}. The \emph{boundary of $C$} is $\partial C : = \overline{C} \setminus \relint(C)$, the \emph{ideal boundary of $C$} is $\partiali C := \partial C \setminus C$ and the \emph{non-ideal boundary of $C$} is $\partialn C := \partial C \cap C$. If $B \subset C \subset \Pb(\Rb^d)$ are properly convex sets, then we say that $B$ is \emph{properly embedded} in $C$ if $B \hookrightarrow C$ is a proper map with respect to the subspace topology. Note that $B$ is properly embedded in $C$ if and only if $\partiali B \subset \partiali C$. 
\end{definition}

An important property of properly convex domains is the existence of supporting hyperplanes. A subset $H \subset \Pb(\Rb^d)$ is called a \emph{(projective) hyperplane}  if it is the projectivization of a codimension one linear subspace of $\Rb^d$. Given a properly convex domain $\Omega \subset \Pb(\Rb^d)$ and $x \in \partial \Omega$, a hyperplane $H$ is called a \emph{supporting hyperplane of $\Omega$ at $x$} if $x \in H$ and $H \cap \Omega = \emptyset$. A boundary point $x \in \partial \Omega$ is always contained in at least one supporting hyperplane. In the case when $x \in \partial \Omega$ is contained in a unique supporting hyperplane we say that $x$ is a \emph{$\Cc^1$-smooth point} of $\partial \Omega$ and denote this unique hyperplane by $T_x \partial \Omega$.

A \emph{line segment} in $\Pb(\Rb^{d})$ is a connected subset of a projective line. Given two points $x,y \in \Pb(\Rb^{d})$ there is no canonical line segment with endpoints $x$ and $y$, but we will use the following convention: if $C \subset \Pb(\Rb^d)$ is a properly convex set and $x,y \in \overline{C}$, then (when the context is clear) we will let $[x,y]$ denote the closed line segment joining $x$ to $y$ which is contained in $\overline{C}$. In this case, we will also let $(x,y)=[x,y]\setminus\{x,y\}$, $[x,y)=[x,y]\setminus\{y\}$, and $(x,y]=[x,y]\setminus\{x\}$.

Along similar lines, given a convex set $C \subset \Pb(\Rb^d)$ and a subset $X \subset C$ we will let 
\begin{align*}
\ConvHull_C(X)
\end{align*}
 denote the smallest convex subset of $C$ which contains $X$.

We include the following examples to clarify the notations introduced above. 
  
 \begin{example} \label{eg:relint}
 Suppose  $C$ is a properly convex set.
 \begin{enumerate}
 \item If $x,y \in \overline{C}$, then $[x,y] = \ConvHull_{\overline{C}}(\{x,y\})$ and $(x,y) = \relint([x,y])$.
\item If $X \subset \overline{C}$, then $$\overline{\ConvHull_{\overline{C}}(X)}=\ConvHull_{\overline{C}}(\overline{X}).$$ Indeed, it suffices to verify that $\overline{\ConvHull_{\overline{C}}(X)} \subset \ConvHull_{\overline{C}}(\overline{X})$. But since $C$ is a properly convex set, this is a straighforward consequence of Carath\'{e}odory's convex hull theorem.
\item If $A \subset \overline{C}$ is a convex subset, then $\relint(A)=\relint(\overline{A})$. In particular, if $X \subset \overline{C}$, then part (2) of this example implies that 
$$\relint(\ConvHull_{\overline{C}}(X))=\relint(\ConvHull_{\overline{C}}(\overline{X})).$$
 \end{enumerate}
 \end{example}

 \begin{example}
Suppose $\Omega $ is a properly convex domain and $\Gamma \leq \Aut(\Omega)$ is a discrete subgroup. Recall from the introduction that $\Lc_{\Omega}(\Gamma)$ is the full orbital limit set of $\Gamma$ in $\Omega$ and $\Cc_{\Omega}(\Gamma)$ is the convex hull of $\Lc_{\Omega}(\Gamma)$ in $\Omega$. Then, in the notation introduced above, $$\Cc_{\Omega}(\Gamma)=\ConvHull_{\overline{\Omega}}(\Lc_{\Omega}(\Gamma)) \cap \Omega.$$
\end{example}

\subsection{Open faces and the Hilbert metric}

\begin{definition}\label{defn:open_faces}
If $\Omega \subset \Pb(\Rb^d)$ is a properly convex domain  and $x \in \overline{\Omega}$, let $F_\Omega(x)$ denote the \emph{open face} of $x$, that is 
\begin{equation*}
F_\Omega(x) = \{ x\} ~\cup ~\{ y \in \overline{\Omega} :  \exists \text{ an open line segment in } \overline{\Omega} \text{ containing } x \text{ and }y\}.
\end{equation*}
\end{definition}

Directly from the definitions we have the following. 

\begin{observation}\label{obs:faces} Suppose $\Omega \subset \Pb(\Rb^d)$ is a properly convex domain. 
\begin{enumerate}
\item $F_\Omega(x) = \Omega$ when $x \in \Omega$,
\item $F_\Omega(x)$ is open in its span,
\item $y \in F_\Omega(x)$ if and only if $x \in F_\Omega(y)$ if and only if $F_\Omega(x) = F_\Omega(y)$.
\end{enumerate}
\end{observation}

Next we recall the definition of the Hilbert distance. Suppose $\Omega \subset \Pb(\Rb^d)$ is a properly convex domain. If $x, y \in \Omega$, let $\overline{xy}$ be a projective line in $\overline{\Omega}$ containing them and let $a,b$ be the two points in $\overline{xy}\cap \partial \Omega$ ordered $a, x, y, b$ along $\overline{xy}$. Then the \emph{the Hilbert distance} between $x$ and $y$ is defined to be
\begin{align*}
\hil(x,y) = \frac{1}{2}\log [a, x,y, b]
\end{align*}
 where 
 \begin{align*}
 [a,x,y,b] = \frac{\abs{x-b}\abs{y-a}}{\abs{x-a}\abs{y-b}}
 \end{align*}
 is the cross ratio. Then $(\Omega, \hil)$ is a complete geodesic metric space and $\Aut(\Omega)$ acts properly and by isometries on $\Omega$ (see for instance~\cite[Section 28]{BK1953}). Further, the projective line segment $[x,y]$ is a geodesic for the Hilbert distance. 
 
The asymptotic behavior of the Hilbert distance connects naturally with the structure of open faces in the boundary.

\begin{observation}
\label{obs:dist_est_and_faces}
Suppose $\Omega \subset \Pb(\Rb^d)$ is a properly convex domain. If $(x_n)_{n \geq 1}$, $(y_n)_{n \geq 1}$ are sequences in $\Omega$, $x:=\lim_{n \to \infty} x_n \in \overline{\Omega}$, $y:=\lim_{n \to \infty} y_n \in \overline{\Omega}$, and $\sup_{n \geq 1} \hil(x_n,y_n) < + \infty$, then $y \in F_\Omega(x)$.
\end{observation}

\subsection{Cones and simplices}\label{sec:cones and simplices}

A subset $C \subset \Pb(\Rb^d)$ is a \emph{properly convex cone} if it is properly convex and there exists an affine chart $\mathbb{A}$ of $\Pb(\Rb^d)$ where $C \subset \mathbb{A}$ is a cone. 

In the case when $C \subset \Pb(\Rb^d)$ is a properly convex cone which is open in its span, there exists a direct sum decomposition $\Span C = V_1 \oplus V_2$ with $\dim V_1=1$ and there exists a properly convex domain $B \subset \Pb(V_2)$ such that 
\begin{align*}
C=\relint \left(\ConvHull_{\overline{C}}\left(\left\{v,B\right\}\right)\right)
\end{align*}
where $v := \Pb(V_1)$. Then we say that \emph{$v$ is a vertex of $C$ with base $B$} (a cone could have several decomposition of this type).

A subset $S \subset \Pb(\Rb^d)$ is called a \emph{$k$-dimensional simplex} in $\Pb(\Rb^d)$  if there exists $g \in \PGL_d(\Rb)$ such that 
\begin{align*}
g S = \left\{ [x_1:\dots:x_{k+1}:0:\dots:0] \in\Pb(\Rb^d) : x_1>0,\dots,x_{k+1} > 0 \right\}.
\end{align*}
In this case, we call the $k+1$ points 
\begin{align*}
g^{-1}[1:0:\dots:0], g^{-1}[0:1:0:\dots:0], \dots, g^{-1}[0:\dots:0:1:0:\dots:0] \in \partial S
\end{align*}
the \emph{vertices of} $S$.

\begin{remark}
Any simplex $S \subset \Rb(\Rb^d)$ of dimension at least one is a properly convex cone: if $v_1,\dots, v_{k+1}$ are the vertices of $S$, then $v_1$ is a vertex of $S$ with base $\relint \left(\ConvHull_{\overline{S}}\left(\left\{v_2,\dots,v_{k+1}\right\}\right)\right)$.
\end{remark}

We now explain some properties of simplices that we will use throughout the paper. 

\begin{example}\label{ex:basic_properties_of_simplices} 
Consider the $(d-1)$-dimensional simplex 
\begin{align*}
S = \left\{ [x_1:\dots:x_{d}] \in \Pb(\Rb^{d}) : x_1>0, \dots, x_{d}> 0\right\}.
\end{align*}
Then 
\begin{align*}
\Aut (S)=\{ [{\rm diag} (\lambda_1, \ldots,\lambda_{d})] \in \PGL_d(\Rb): \lambda_1>0, \ldots, \lambda_{d}>0]\} \rtimes S_d
\end{align*} 
where $S_d$ is the subgroup of permutation matrices in $\PGL_d(\Rb)$. In particular, $\Aut(S)$ is virtually Abelian. 
\end{example}

From this explicit description of the automorphism group we observe the following. 

\begin{observation}\label{obs:cocompact-action-on-simplices}
Suppose $S \subset \Pb(\Rb^d)$ is a simplex.  If $H \leq \PGL_d(\Rb)$ preserves $S$ and acts co-compactly on $S$, then:
\begin{enumerate}
\item  If $H_0 \leq H$ is the subgroup of elements that fix the vertices of $S$, then $H_0$ also acts co-compactly on $S$.
\item If $F \subset \partial S$ is a face of $S$, then 
\begin{align*}
\Stab_H(F):=\{ h \in H : hF=F\} 
\end{align*}
acts co-compactly on $F$. 
\end{enumerate}
\end{observation}

\subsection{Limits of linear maps} Every $T \in \Pb(\End(\Rb^d))$ induces a map 
\begin{align*}
\Pb(\Rb^d) \setminus \Pb(\ker T) \rightarrow \Pb(\Rb^d)
\end{align*}
defined by $x \mapsto T(x)$. We will frequently use the following observation.

\begin{observation}\label{obs:limits_of_maps} If $(T_n)_{n \geq 1}$ is a sequence in $\Pb(\End(\Rb^d))$ converging to $T \in \Pb(\End(\Rb^d))$, then 
\begin{align*}
T(x) = \lim_{n \rightarrow \infty} T_n(x)
\end{align*}
for all $x \in \Pb(\Rb^d) \setminus \Pb(\ker T)$. Moreover, the convergence is uniform on compact subsets of $ \Pb(\Rb^d) \setminus \Pb(\ker T)$. 
\end{observation}

We can also view $\Pb(\End(\Rb^d))$ as a compactification of $\PGL_d(\Rb)$ and then consider limits of automorphisms in this compactification. 

\begin{proposition}\cite[Proposition 5.6]{IZ2019}\label{prop:dynamics_of_automorphisms_1}
Suppose $\Omega \subset \Pb(\Rb^d)$ is a properly convex domain, $p_0 \in \Omega$, and $(g_n)_{n \geq 1}$ is a sequence in $\Aut(\Omega)$ such that 
\begin{enumerate}
\item $g_n (p_0) \rightarrow x \in \partial \Omega$, 
\item $g_n^{-1} (p_0) \rightarrow y \in \partial \Omega$, and
\item $g_n\rightarrow T \in \Pb(\End(\Rb^d))$. 
\end{enumerate}
Then ${\rm image}\, T \subset \Spanset  F_\Omega(x)$, $\Pb(\ker T) \cap \Omega = \emptyset$, and $y \in \Pb(\ker T)$. 
\end{proposition} 

In the case when the orbit stays within a uniform neighborhood of a projective line geodesic (i.e. the convergence is ``conical''), we can say more about the linear map obtained in the limit.   

\begin{proposition}\cite[Proposition 5.7]{IZ2019}\label{prop:dynamics_of_automorphisms_2}
Suppose $\Omega \subset \Pb(\Rb^d)$ is a properly convex domain, $p_0 \in \Omega$, $x \in \partial \Omega$, $(p_n)_{n \geq 1}$ is a sequence in $[p_0, x)$ converging to $x$, and $(g_n)_{n \geq 1}$ is a sequence in $\Aut(\Omega)$ such that 
\begin{align*}
\sup_{n \geq 0} \hil(g_n p_0, p_n) < + \infty.
\end{align*}
If $g_n \rightarrow T \in \Pb(\End(\Rb^d))$, then
\begin{align*}
T(\Omega) = F_\Omega(x).
\end{align*}
\end{proposition}

 Proposition 5.7 in~\cite{IZ2019} is stated differently and a proof of the statement above can be found in~\cite[Proposition 2.13]{Z2020}.

\subsection{Background on relatively hyperbolic metric spaces}\label{sec:background_on_rel_hyp}

We define relative hyperbolic spaces and groups in terms of Dru{\c t}u and Sapir's tree-graded spaces (see~\cite[Definition 2.1]{DS2005}). 

\begin{definition} \ \begin{enumerate}
\item A complete geodesic metric space $(X,\dist)$ is \emph{relatively hyperbolic with respect to a collection of subsets $\Sc$} if all its asymptotic cones, with respect to a fixed non-principal ultrafilter, are tree-graded with respect to the collection of ultralimits of the elements of $\Sc$. 
\item A finitely generated group $G$ is \emph{relatively hyperbolic with respect to a family of subgroups $\{H_1,\dots, H_k\}$} if the Cayley graph of $G$ with respect to some (hence any) finite set of generators is relatively hyperbolic with respect to the collection of left cosets $\{g H_i : g \in G, i=1,\dots,k\}$. 
\end{enumerate}
\end{definition}

\begin{remark}The above definition is one of several equivalent definitions of relatively hyperbolic spaces/groups, see~\cite{DS2005} and the references therein for more details. 

\end{remark}

We now recall some useful properties of relatively hyperbolic spaces.  Recall the notation $\Nc(A;r)$ and $\dist^{\Haus}(A,B)$ from Section~\ref{sec:notation in preliminary}.

\begin{theorem}[{Dru{\c t}u--Sapir~\cite[Theorem 4.1]{DS2005}}]\label{thm:rh_intersections_of_neighborhoods} Suppose $(X,\dist)$ is relatively hyperbolic with respect to $\Sc$. For any $r > 0$ there exists $Q(r) > 0$ such that: if $S_1, S_2 \in \Sc$ are distinct, then 
\begin{align*}
\diam \big( \Nc(S_1;r)\cap \Nc(S_2;r) \big) \leq Q(r).
\end{align*}
\end{theorem}

\begin{theorem}[{Dru{\c t}u--Sapir~\cite[Corollary 5.8]{DS2005}}]\label{thm:rh_embeddings_of_flats} Suppose $(X, \dist)$ is relatively hyperbolic with respect to $\Sc$. Then for any $A \geq 1$ and $B \geq 0$ there exists $M =M(A,B)$ such that: if $k \geq 2$ and $f : \Rb^k \rightarrow X$ is an $(A,B)$-quasi-isometric embedding, then there exists some $S \in \Sc$ such that 
\begin{align*}
f(\Rb^k) \subset \Nc(S;M).
\end{align*}
\end{theorem}

\section{The geometry of 3-manifolds}
\label{sec:three_manifold_theory}

This expository section is a compilation of facts about the geometry of 3-manifolds that will be used in the paper. For more details see \cite{S1983,Thurston1997,FB2002,MF2010,AFW2015}.  

\subsection{Geometric 3-manifolds}

A \emph{3-dimensional geometry} is a pair $(X,G)$ where $X$ is a simply connected smooth 3-manifold and $G$ is a Lie group acting smoothly, transitively, and effectively on $X$ with compact point-stabilizers. It is also assumed that $G$ is maximal among all Lie groups which satisfy the above property. Two geometries $(X,G)$ and $(X',G')$ are equivalent if $\phi:X \to X'$ is a diffeomorphism which is equivariant with respect to the actions of $G$ and $G'$ respectively. For example, the 3-dimensional Euclidean geometry corresponds to the pair $(\Rb^3, \Rb^3 \rtimes {\rm O}(3))$ while the 3-dimensional hyperbolic geometry corresponds to the pair $(\Hb^3, {\rm PO}(3,1))$.

\begin{definition}
A 3-manifold $M$ is \emph{geometric} with geometry $(X,G)$ if $M$ is homeomorphic to a quotient $\Gamma \backslash X$ where $\Gamma \leq G$ is a discrete subgroup and $\Gamma \backslash X$ has finite volume with respect to the pushforward of some (any) $G$-invariant Riemannian metric on $X$ (this follows the terminology in \cite[Lecture 1]{MF2010}). 
\end{definition}

This point onwards, for brevity, we will only write $X$ to denote the pair $(X,G)$ when discussing well-known geometries. The following are the eight \emph{3-dimensional Thurston geometries}:
\begin{align}
\label{eqn:3-dim-geom}
X= \Rb^3, \ \Sb^3, \ \Hb^3, \ \Rb \times \Sb^2, \ \Rb \times \Hb^2, \ \wt{\SL_2(\Rb)}, \ {\rm Nil}, \ {\rm Sol}.
\end{align}
Thurston showed that this is precisely the list of 3-dimensional geometries that can appear as a geometry on a 3-manifold $M$. For more details about these geometries, see \cite{S1983}. The next theorem states that a closed 3-manifold which is geometric supports a unique 3-dimensional geometry.

\begin{proposition} [{\cite[Theorem 5.2]{S1983}}]
\label{prop:closed-mfld-unique-gstr}
If $M$ is a closed 3-manifold, then $M$ admits at most one of the geometries in \eqref{eqn:3-dim-geom}.
\end{proposition}   

We also observe that with our (very restrictive) definition of geometric manifold, a non-compact 3-manifold can only possibly admit three of the geometries. 

\begin{proposition}\label{prop:finite volume implies compact, sometimes}
If $M$ is a 3-manifold which admits a $\Rb^3$, $\Sb^3$, $\Rb \times \Sb^2$, ${\rm Nil}$, or ${\rm Sol}$ geometry, then $M$ is closed.
\end{proposition} 

\begin{proof} In the case when $M$ admits a $\Rb^3$, $\Sb^3$ or $\Rb \times \Sb^2$ geometry, this follows in a straightforward way from the discussion in~\cite{S1983}. In the case when $M$ admits a ${\rm Nil}$ or ${\rm Sol}$ geometry, it is perhaps easier to use the theory of lattices in solvable Lie groups. For instance,  assume that $M$ admits a ${\rm Sol}$ geometry. In this case, $X = {\rm Sol}$ and ${\rm Sol}$ coincides with the connected component of the identity in $G$. So $M$ is homeomorphic to a quotient $\Gamma \backslash {\rm Sol}$ which has finite volume with respect to the pushforward of some (any) $G$-invariant Riemannian metric on $X$. Then $\Gamma_0 := \Gamma \cap {\rm Sol}$ has finite index in $\Gamma$ and hence is a lattice in ${\rm Sol}$. 
Then~\cite[Theorem 3.1]{R1972} implies that $\Gamma_0$ is a co-compact lattice. Hence $M \cong \Gamma \backslash {\rm Sol}$ is closed. The argument when $M$ has ${\rm Nil}$ geometry is very similar. 
\end{proof}

We next recall a result about manifolds with $\Rb^3$, ${\rm Nil}$, or ${\rm Sol}$ geometry.

\begin{proposition}
\label{prop:nil_or_sol_geom}
Suppose $M$ is a closed geometric 3-manifold. 
\begin{enumerate}
\item $M$ has $\Rb^3$ geometry if and only if $\pi_1(M)$ is virtually Abelian, but not virtually cyclic. 
\item $M$ has ${\rm Nil}$ geometry if and only if $\pi_1(M)$ is virtually nilpotent, but not virtually Abelian.
\item $M$ has ${\rm Sol}$ geometry if and only if $\pi_1(M)$ is virtually solvable, but not virtually nilpotent.
\end{enumerate}
\end{proposition}

\begin{proof} This follows from~\cite[Theorem 4.7.8 and Figure 4.22]{Thurston1997}.
\end{proof} 

\subsection{Non-geometric 3-manifolds} Not every 3-manifold is geometric and such a manifold is called \emph{non-geometric}. The geometric decomposition theorem is the key to understanding the structure of these 3-manifolds. Before stating the theorem, we need two technical definitions.

A submanifold $ N \hookrightarrow M$ is called \emph{incompressible} provided the induced map $\pi_1(N) \to \pi_1(M)$ is injective. A 3-manifold $M$ is called \emph{irreducible} provided any embedded $\Sb^2$ in $M$ is the boundary of a 3-ball  in $M$.

\begin{remark}\label{rem:irred_implies_constraint_on _geom}
If $M$ is a closed irreducible orientable 3-manifold, then $M$ cannot have $\Rb \times \Sb^2$ geometry. Indeed, there are only four closed 3-manifolds admitting a $\Rb \times \Sb^2$ geometry (see~\cite[pg.\ 457]{S1983}) and amongst these only $\Sb^1 \times \Sb^2$ is orientable. 
\end{remark}

\begin{theorem}[Geometric decomposition theorem]
\label{thm:geom-decomp}
Suppose $M$ is a closed irreducible orientable 3-manifold. Then there is a (possibly empty) collection of disjoint and incompressible tori and Klein bottles $\Tc=\{T_1,\dots, T_m\}$ such that each connected component of $(M - \Tc)$ is a geometric 3-manifold. Moreover, any such collection $\Tc$, with minimal number of components in $(M-\Tc)$, is unique up to isotopy. 
\end{theorem}

A different (but related) decomposition of a 3-manifold is given by the  JSJ decomposition theorem, see \cite[Theorem 1.6.1]{AFW2015}. It states the existence of a collection $\Tc$ (of disjoint and incompressible tori and Klein bottle) such that each component of $(M - \Tc)$ is atoroidal (i.e. any $\pi_1$ injective tori is homotopic to a boundary tori) or a Seifert fibered 3-manifold. In the next section, we discuss the structure of Seifert fibered 3-manifolds.

\subsection{Seifert fibered 3-manifolds}
The building blocks of a Seifert fibered 3-manifold are fibered solid tori and Klein bottle. Let $\Db^2=\{z \in \Cb: |z| \leq 1\}$. A fibered solid torus is the quotient space $\Db^2\times [0,1]/{\sim}$ where $$(z,1) \sim (z \cdot \exp^{2 \pi i q/p},0)$$ for co-prime integers $(q,p)$ with $p>0$. A fibered solid Klein bottle is $\Db^2\times [0,1] /{\sim}$ where $\sim$ identifies $\Db^2 \times \{1\}$ with $\Db^2 \times \{0\}$ by a reflection across a diameter of $\Db^2$. Fibered solid tori and Klein bottle are compact 3-manifolds that are finitely covered by $\Db^2 \times \Sb^1$ and admit a foliation by circles. 

\begin{definition} 
A \emph{Seifert fibered 3-manifold} $M$ is a 3-manifold that can be decomposed as a disjoint union of topological circles (called \emph{fibers}) such that each circle has tubular neighborhood in $M$ homeomorphic to a fibered solid torus or Klein bottle. A fiber is called \emph{regular} if it has a neighborhood in $M$ isomorphic to a $\Db^2 \times \Sb^1$ and \emph{critical} otherwise.
\end{definition}  

Now observe that if $M$ is a Seifert fibered 3-manifold, then $M$ is a $\Sb^1$ bundle over a 2-dimensional orbifold $\Sigma$ (called the base orbifold). Indeed, the orbifold $\Sigma$ is obtained by collapsing each fiber to a point and the orbifold points of $\Sigma$ correspond to critical fibers, see \cite[Section 3]{S1983} for details. If $M$ has boundary components, then $\Sigma$ is an orbifold with boundary. The algebraic structure of the fundamental group of such a manifold is given by the following. 

\begin{proposition} [{\cite[Lemma 3.2]{S1983}}]
\label{prop:seifert-alg}
Suppose $M$ is a Seifert fibered 3-manifold with base orbifold $\Sigma$. Then there is an exact sequence
\begin{align*}
1 \rightarrow N \rightarrow \pi_1(M) \rightarrow \pi_1(\Sigma) \rightarrow 1
\end{align*}
where $N$ is a cyclic subgroup of $\pi_1(M)$ generated by a regular fiber. Moreover, $N$ is infinite except in the case where $M$ is covered by $\Sb^3$. 
\end{proposition}

The following results discusses the structure of geometric 3-manifolds which are Seifert fibered.

\begin{proposition}
\label{prop:seifert_geometric_manifolds}
If $M$ is a geometric $3$-manifold with $\Rb^3$, $\Rb \times \Hb^2$, or $\wt{\SL_2(\Rb)}$ geometry, then $M$ is Seifert fibered. Moreover, if $M$ has $\Rb \times \Hb^2$ or $\wt{\SL_2(\Rb)}$ geometry, then the base orbifold is homeomorphic to a 2-dimensional finite area hyperbolic orbifold. 
\end{proposition}

\begin{proof} This follows from \cite[Theorems 2.7 and 3.17]{FB2002}   (also see Theorem 1.8.1 and Table 1 in \cite{AFW2015}). \end{proof} 

Finally, we observe the following about the geometric pieces in the geometric decomposition. Below, $Z_H(h)$ denotes the centralizer of $h$ in the group $H$.  

\begin{proposition}
\label{prop:seifert-piece}
Suppose $M$ is a non-geometric closed irreducible orientable 3-manifold  and $S$ is a connected component in the geometric decomposition of $M$. If $S$ does not admit a $\Hb^3$ geometry, then:
\begin{enumerate}
\item $S$ has a $\Rb \times \Hb^2$ or $\wt{\SL_2(\Rb)}$ geometry,
\item if $\ip{h}$ is the normal subgroup of $\pi_1(S)$ generated by a regular fiber (see Propositions~\ref{prop:seifert_geometric_manifolds} and~\ref{prop:seifert-alg}), then $Z_{\pi_1(S)}(h)=Z_{\pi_1(M)}(h)$. 
\end{enumerate}
\end{proposition}

\begin{proof}
Part (1) follows from Proposition~\ref{prop:finite volume implies compact, sometimes}. Part (2) is well-known, see for instance  \cite[Proof of Theorem 4.4]{F2011}. 
\end{proof}

\section{Centralizers of convex co-compact actions}

 In this section, we will prove the following expanded version of \Cref{thm:center_intro}.

\begin{theorem}\label{thm:center}Suppose $\Omega \subset \Pb(\Rb^d)$ is a properly convex domain, $\Gamma \leq \Aut(\Omega)$ is convex co-compact, $A \leq \Gamma$ is an infinite Abelian subgroup, and $Z_{\Gamma}(A)$ is the centralizer of $A$ in $\Gamma$. If 

\begin{align*}
V := \Spanset \left\{ x \in \overline{\Cc_\Omega(\Gamma)} :  a x = x \text{ for all } a \in A \right\},
\end{align*}
then 
\begin{enumerate}
\item $\Omega \cap \Pb(V)$ is a non-empty $Z_{\Gamma}(A)$-invariant properly convex domain in $\Pb(V)$,
\item $\Omega \cap \Pb(V) \subset \Cc_\Omega(\Gamma)$, 
\item the quotient $Z_{\Gamma}(A) \backslash \Omega \cap \Pb(V)$ is compact, and
\item there exist a non-trivial $Z_{\Gamma}(A)$-invariant direct sum decomposition $V = \oplus_{j=1}^m V_j$ and properly convex domains $F_j \subset \Pb(V_j)$ such that $A$ acts by scaling on each $V_j$ and 
\begin{align*}
\Omega \cap \Pb(V) = {\rm relint}\left(\ConvHull_{\overline{\Omega}} \left( \cup_{j=1}^m F_j\right) \right).
\end{align*}
\end{enumerate}
\end{theorem}

\begin{remark} Part (2) implies that $\Omega \cap \Pb(V)$ is a properly embedded closed convex subset of $\Cc_{\Omega}(\Gamma)$. In the extremal case when $\dim V_j = 1$, we have $F_j=\Pb(V_j)$. Further, if $\dim V_j = 1$ for all $1\leq j \leq m$, then $\Omega \cap \Pb(V)$ is a properly embedded simplex. 
\end{remark}

\begin{remark} In Section~\ref{sec:failure of center theorem for naive convex co-compact groups}, we will provide examples of naive convex co-compact groups with a centralizer that doesn't satisfy parts (2) and (3) in Theorem~\ref{thm:center}. \end{remark}

Before proving Theorem~\ref{thm:center} we state and prove two corollaries. 

\begin{corollary}\label{cor:fund_gp_of_manifold} Under the hypothesis of Theorem~\ref{thm:center}: $Z_{\Gamma}(A)$ is virtually the fundamental group of a closed aspherical $(\dim V-1)$-manifold. \end{corollary}

\begin{proof} By Selberg's lemma there exists a finite index torsion free subgroup $\Gamma_0 \leq \Gamma$. Since $\Gamma$ acts properly discontinuously on $\Omega$, the stabilizer of any point in $\Omega$ is finite. Hence $\Gamma_0$ acts freely on $\Omega$. Then $G : = \Gamma_0 \cap Z_{\Gamma}(A)$ has finite index in $Z_{\Gamma}(A)$ and acts properly discontinuously, freely, and co-compactly on $\Omega \cap \Pb(V)$. Since $\Omega \cap \Pb(V)$ is a properly convex domain it is diffeomorphic to $\Rb^{\dim V-1}$. Hence $G$ is isomorphic to the fundamental group of $G \backslash \Omega \cap \Pb(V)$, a closed aspherical $(\dim V-1)$-manifold.
\end{proof}

\begin{corollary} Under the hypothesis of Theorem~\ref{thm:center}: If $N$ is the normalizer of $Z_{\Gamma}(A)$ in $\Gamma$, then $Z_{\Gamma}(A)$ has finite index in $N$.
\end{corollary}

\begin{proof}Since $\Gamma$ acts properly discontinuously on $\Omega$ and  $Z_{\Gamma}(A) \backslash \Omega \cap \Pb(V)$ is compact, it is enough to show that $nV= V$ for all $n \in N$. 

Fix $n \in N$. Then $Z_{\Gamma}(A)=nZ_{\Gamma}(A)n^{-1}$ acts co-compactly on 
\begin{align*}
\Omega \cap \Pb(nV) = n \left( \Omega \cap \Pb(V) \right).
\end{align*}
So $Z_{\Gamma}(A) \leq \Aut(\Omega \cap \Pb(nV))$ is a convex co-compact subgroup with 
\begin{align*}
\Cc_{\Omega \cap \Pb(nV)}(Z_{\Gamma}(A)) = \Omega \cap \Pb(nV).
\end{align*}
So by Theorem~\ref{thm:center}, if 
\begin{align*}
V_1 := \Spanset \left\{ x \in \overline{\Omega \cap \Pb(nV)} : ax = x \text{ for all } a \in A \right\},
\end{align*}
then $\Omega \cap \Pb(V_1)$ is non-empty, $Z_{\Gamma}(A)$ preserves $\Omega \cap \Pb(V_1)$, and the quotient $Z_{\Gamma}(A) \backslash \Omega \cap \Pb(V_1)$ is compact.

By Corollary \ref{cor:fund_gp_of_manifold}, $Z_{\Gamma}(A)$ is virtually the fundamental group of closed aspherical manifolds of dimension $(\dim(V_1)-1)$ as well as $(\dim(V)-1)$. Thus $\dim(V_1)=\dim(V)$. Further, by definition, $V_1 \subset V \cap nV$. Thus $V_1=V=nV$. 
\end{proof}

\subsection{Preliminary results} The proof of Theorem~\ref{thm:center} requires some prior results about Abelian subgroups and their centralizers  in naive convex co-compact subgroups. 

\begin{definition}\label{defn:cc_naive}  Suppose $\Omega \subset \Pb(\Rb^d)$ is a properly convex domain. An infinite discrete subgroup $\Gamma \leq \Aut(\Omega)$ is called \emph{naive convex co-compact} if there exists a non-empty closed convex subset $\Cc \subset \Omega$ such that 
\begin{enumerate}
\item $\Cc$ is $\Gamma$-invariant, that is, $g\Cc = \Cc$ for all $g \in \Gamma$, and
\item $\Gamma$ acts co-compactly on $\Cc$. 
\end{enumerate}
In this case, we say that $(\Omega, \Cc, \Gamma)$ is a \emph{naive convex co-compact triple}. 
\end{definition}

Every convex co-compact subgroup is clearly naive convex co-compact, but there exist examples of naive convex co-compact subgroups which are not convex co-compact. For naive convex co-compact groups, we previously established the following results about maximal Abelian subgroups.

\begin{theorem}[I.--Z.~\cite{IZ2019}] \label{thm:max_abelian} Suppose $(\Omega, \Cc, \Gamma)$ is a naive convex co-compact triple and $A \leq \Gamma$ is a maximal Abelian subgroup. Then there exists a properly embedded $k$-dimensional simplex $S \subset \Cc$ such that:
\begin{enumerate}
\item $S$ is $A$-invariant,
\item $A$ fixes each vertex of $S$, and
\item $A$ acts co-compactly on $S$. 
\end{enumerate}
Moreover, $A$ contains a finite index subgroup isomorphic to $\Zb^k$. 
\end{theorem}

To state the next theorem we need to recall some terminology. 

\begin{definition}\label{defn:tau and Min} Suppose $\Omega \subset \Pb(\Rb^d)$ is a properly convex domain and $g \in \Aut(\Omega)$. Define the \emph{minimal translation length of $g$} to be
\begin{align*}
\tau_\Omega(g): = \inf_{x \in \Omega} \dist_\Omega(x, g x)
\end{align*}
and the \emph{minimal translation set of $g$} to be
\begin{align*}
\Min_\Omega(g) := \{ x \in\Omega: \dist_\Omega(x,gx) = \tau_\Omega(g) \}.
\end{align*}
\end{definition}

Cooper--Long--Tillmann~\cite{CLT2015} showed that the minimal translation length of an element can be determined from its eigenvalues. 

\begin{proposition}\cite[Proposition 2.1]{CLT2015}\label{prop:min_trans_compute} If $\Omega \subset \Pb(\Rb^d)$ is a properly convex domain and $g \in \Aut(\Omega)$, then 
\begin{align*}
\tau_\Omega(g) = \frac{1}{2} \log \frac{\lambda_1(g)}{\lambda_d(g)}.
\end{align*}
\end{proposition}

In previous work, we proved that the centralizer of an Abelian subgroups of a naive convex co-compact group is also naive convex co-compact.

\begin{theorem}[I.--Z.~\cite{IZ2019}]\label{thm:centralizers_ncc} Suppose $(\Omega, \Cc, \Gamma)$ is a naive convex co-compact triple and $A \leq \Gamma$ is an Abelian subgroup. Then
\begin{align*}
\Min_{\Cc}(A): = \Cc \cap  \left( \bigcap_{a \in A} \Min_\Omega(a) \right) 
\end{align*}
is non-empty and $Z_{\Gamma}(A)$ acts co-compactly on $\ConvHull_{\Omega}(\Min_{\Cc}(A))$. \end{theorem}

We will also use the following computation.

\begin{proposition}\label{prop:min_set_inv_simplex} Suppose $\Omega \subset \Pb(\Rb^d)$ is a properly convex domain and $S \subset \Omega$ is a simplex (which may not be properly emebedded). If $g \in \Aut(\Omega)$ fixes every vertex of $S$, then $S \subset {\rm Min}_{\Omega}(g)$. 
\end{proposition}

\begin{proof} Let $V := \Spanset S$. The Hilbert metric on $S$ can be explicitly computed, see for instance~\cite[Example 3.1]{IZ2019}, and from this one sees that ${\rm Min}_S(g|_V) = S$ (this is only true since $g$ fixes every vertex of $S$). 

By Proposition~\ref{prop:min_trans_compute} there exist $1 \leq i < j \leq d$ such that 
$$
\tau_S(g|_V) = \frac{1}{2} \log \frac{\lambda_1(g|_V)}{\lambda_{\dim V}(g|_V)}=\frac{1}{2} \log \frac{\lambda_i(g)}{\lambda_j(g)}
$$
and so by Proposition~\ref{prop:min_trans_compute} again $\tau_S(g|_V) \leq \tau_\Omega(g)$. Then, by the definition of the Hilbert metric, if $p \in S$ we have
$$
\dist_\Omega(g(p), p) \leq \dist_S(g(p), p) =\tau_S(g|_V) \leq \tau_\Omega(g).
$$
So $S \subset {\rm Min}_{\Omega}(g)$. 
\end{proof}

\subsection{Proof of Theorem~\ref{thm:center}}

Suppose $\Omega \subset \Pb(\Rb^d)$, $\Gamma \leq \Aut(\Omega)$, and $A \leq \Gamma$ satisfy the hypothesis of the theorem. Let $\Cc : = \Cc_\Omega(\Gamma)$. As in the statement of Theorem~\ref{thm:center} define 
\begin{align*}
V := \Spanset \left\{  x \in \overline{\Cc} : ax= x \text{ for all } a \in A \right\}.
\end{align*}

Let $\wh{A} \leq \GL_d(\Rb)$ be the preimage of $A$ under the projection $\GL_d(\Rb) \rightarrow \PGL_d(\Rb)$. For a homomorphism $\nu: \wh{A} \rightarrow \Rb^{\times}$, let
\begin{align*}
E_{\nu} := \left\{ v\in \Rb^d : av = \nu(a)v \text{ for all } a \in \wh{A} \right\}.
\end{align*}
Then let 
\begin{align*}
\{ \nu_1,\dots, \nu_m \} = \{ \nu : E_{\nu} \neq \{0\} \}
\end{align*}
and
\begin{align*}
F_j := {\rm relint} \Big( \partiali \Cc \cap \Pb(E_{\nu_j})\Big) \subset \partiali\Cc.
\end{align*}
\Cref{lem:convex hull of closed faces is non-empty} below will show that $\{ \nu : E_{\nu} \neq \{0\} \}$ is a non-empty set. Now notice that 
\begin{align}
\label{eqn:fixed points are the closure of the Fjs}
\cup_{j =1}^m \overline{F}_j = \left\{ x \in \overline{\Cc} :a x = x \text{ for all } a \in A \right\}
\end{align}
and in particular
\begin{align*}
V = \Spanset\left\{  \cup_{j =1}^m \overline{F}_j\right\}.
\end{align*}
Next let $V_j :=\Spanset\overline{F}_j$. Then
\begin{align*}
V = V_1 + \dots + V_m
\end{align*}
is a direct sum, as the $\nu_j$-s are distinct.

\begin{lemma}\label{lem:convex hull of closed faces is non-empty}  The sets $\{ \nu : E_{\nu} \neq \{0\}\}$ and $\Omega \cap \ConvHull_{\overline{\Omega}}\left( \cup_{j=1}^m \overline{F}_j \right)$ are non-empty. \end{lemma}

\begin{proof} Fix a maximal Abelian group $A^\prime \leq \Gamma$ containing $A$. Then by Theorem~\ref{thm:max_abelian}, there exists a properly embedded simplex $S \subset \Cc$ where $A^\prime$ fixes the vertices of $S$. Since $A^\prime$ is infinite, $S$ is at least one dimensional. Then the vertices of $S$ lie in $\partiali \Cc$ and are fixed by $A^\prime$. Then, if $v$ is a vertex of $S$, there exists $\nu: \wh{A} \to \Rb^{\times}$ such that $[v] \subset E_{\nu}$. This proves the first part.
By Equation~\eqref{eqn:fixed points are the closure of the Fjs}, the vertices of $S$ are contained in $\cup_{j =1}^m \overline{F}_j$ and so 
\begin{equation*}
S \subset \Omega \cap \ConvHull_{\overline{\Omega}}\left( \cup_{j=1}^m \overline{F}_j \right). \qedhere
\end{equation*}
\end{proof}

\begin{lemma} \label{lem:inclusion-in-min-c}
$\Omega \cap \ConvHull_{\overline{\Omega}}\left( \cup_{j=1}^m \overline{F}_j \right) \subset {\rm Min}_{\Cc}(A)$.  
\end{lemma}

\begin{proof} Fix $g \in A$ and $p \in \Omega \cap \ConvHull_{\overline{\Omega}}\left( \cup_{j=1}^m \overline{F}_j \right)$. Since each $\overline{F}_j$ is convex, we can write 
$$
p=[v_1+\dots+v_m] 
$$
where $v_j \in V_j$ and $[v_j] \in \overline{F}_j$. Let $J:=\{ j : v_j \neq 0\}$. Since $V = \oplus_{j=1}^m V_j$, the vectors $\{ v_j : j \in J\}$ are linearly independent and hence  the set
\begin{align*}
S_J := {\rm relint} \left(  \ConvHull_{\overline{\Omega}}\left( \{ [v_j] : j \in J\} \right)\right)
\end{align*}
is a $(\abs{J}-1)$-dimensional simplex. By construction, $p \in S_J$ and so by convexity $S_J \subset \Omega$. Further,  $g$ fixes the vertices of $S_J$. Hence $p \in {\rm Min}_{\Omega}(g)$ by Proposition~\ref{prop:min_set_inv_simplex}.
\end{proof}

\begin{lemma}\label{lem:computing the intersection}  $\Omega \cap \Pb(V)= {\rm relint}\left(\ConvHull_{\overline{\Omega}}\left( \cup_{j=1}^m F_j \right)\right)$. \end{lemma}

\begin{proof} Clearly, ${\rm relint}\left(\ConvHull_{\overline{\Omega}}\left( \cup_{j=1}^m F_j \right)\right) \subset \Omega \cap \Pb(V)$. To see the reverse inclusion, fix $p \in \Omega \cap \Pb(V)$. Then we can write
\begin{align*}
p = \left[ w_1+\dots+w_m\right]
\end{align*}
where $w_j \in V_j$. 

\medskip 
\noindent \fbox{\emph{Claim:}} $[w_1] \in F_1$. In particular, $w_1 \neq 0$. 
\medskip 

Fix $x \in F_1$. Then 
\begin{align*}
F_1 = F_\Omega(x) \cap \Pb(V_1) = F_{\Omega \cap \Pb(V)}(x).
\end{align*}
By Lemma~\ref{lem:convex hull of closed faces is non-empty}, there exists $p_0 \in \Omega \cap \ConvHull_{\overline{\Omega}}\left( \cup_{j=1}^m \overline{F}_j \right)$. Then fix a sequence $(p_n)_{n \geq 1}$ in $[p_0, x)$ converging to $x$. Since
\begin{align*}
\{p_1,p_2,\dots\} \subset \Omega \cap \ConvHull_{\overline{\Omega}}\left( \cup_{j=1}^m \overline{F}_j \right),
\end{align*}
Lemma \ref{lem:inclusion-in-min-c} implies that $\{p_1,p_2,\dots\} \subset {\rm Min}_{\Cc}(A)$. So by Theorem~\ref{thm:centralizers_ncc} there exists a sequence $(g_n)_{n \geq 1}$ in $Z_{\Gamma}(A)$ such that 
\begin{align*}
\sup_{n\geq 0} \hil(g_n p_0, p_n) <+\infty.
\end{align*}
Since $g_n(V) = V$ we can view $g_n|_V$ as an element of $\Aut(\Omega \cap \Pb(V))$. By passing to a subsequence we can suppose that $g_n|_V \rightarrow T \in \Pb(\End(V))$. By Propositions~\ref{prop:dynamics_of_automorphisms_1} and~\ref{prop:dynamics_of_automorphisms_2}
\begin{enumerate}
\item $\Pb(\ker T) \cap (\Omega\cap \Pb(V)) = \emptyset$,
\item ${\rm image} \, T = \Span F_{\Omega \cap \Pb(V)}(x)  = V_1$,
\item $T(\Omega \cap \Pb(V)) = F_{\Omega \cap \Pb(V)}(x)=F_1$.
\end{enumerate}

Next let $\overline{g}_n \in \GL(V)$ be a lift of $g_n|_V$. Since each $g_n$ commutes with the elements of $A$, each $\overline{g}_n$ preserves the direct sum $V=\oplus_{j=1}^m V_j$. Then relative to the direct sum $V=\oplus_{j=1}^m V_j$,
\begin{align*}
\overline{g}_n = \begin{pmatrix} A_{1,n} & & \\ & \ddots & \\ & & A_{m,n} \end{pmatrix}
\end{align*}
and $[A_{j,n}] \in \Aut(F_j)$. By scaling each $\overline{g}_n$ and passing to a subsequence we may assume that 
$$
\lim_{n \rightarrow \infty} \overline{g}_n = \begin{pmatrix} A_{1} & & \\ & \ddots & \\ & & A_{m} \end{pmatrix}
$$
exists in $\End(V)$ and is non-zero. Then
\begin{align*}
T = \begin{bmatrix} A_{1} & & \\ & \ddots & \\ & & A_{m} \end{bmatrix}.
\end{align*}
Since ${\rm image}\, T = V_1$, we then have 
\begin{align*}
T = \begin{bmatrix} A_{1} & & &\\ & 0 & & \\ & & \ddots & \\ & &  & 0 \end{bmatrix}
\end{align*}
and $A_1 \in \GL(V_1)$. Then $[A_{1,n}]$ converges to $[A_1]$ in $\PGL(V_1)$. Since $\Aut(F_1)$ is closed in $\PGL(V_1)$, we then have $[A_1] \in \Aut(F_1)$. 

 Since $\ker T = \oplus_{j=2}^m V_j$ and $\Pb(\ker T) \cap (\Omega \cap \Pb(V))= \emptyset$, we must have $w_1 \neq 0$. Further, 
\begin{align*}
[w_1] = [A_1^{-1}] T(p) \in [A_1^{-1}](F_1)= F_1
\end{align*}
which completes the proof of the claim. 

Then by symmetry we see that $[w_j] \in F_j$ for all $1 \leq j \leq m$. Hence 
\begin{equation*}
p \in {\rm relint}\left(\ConvHull_{\overline{\Omega}}\left( \cup_{j=1}^m F_j \right)\right). \qedhere
\end{equation*}
\end{proof}

\begin{proof}[Proof of Theorem~\ref{thm:center}] (1): By construction, $\Omega \cap \Pb(V)$ is a $Z_\Gamma(A)$-invariant properly convex domain in $\Pb(V)$. Since $\overline{F}_j \subset \Pb(V_j)$, we have
$$
\Omega \cap \ConvHull_{\overline{\Omega}}\left( \cup_{j=1}^m \overline{F}_j \right) \subset \Omega \cap \Pb(V).
$$
So Lemma~\ref{lem:convex hull of closed faces is non-empty}  implies that $\Omega \cap \Pb(V)$ is non-empty. 

(2): Since $\cup_{j=1}^m F_j \subset \partiali\Cc$, Lemma~\ref{lem:computing the intersection} implies that $\Omega \cap \Pb(V) \subset \Cc$. 

(3): Lemmas~\ref{lem:inclusion-in-min-c} and~\ref{lem:computing the intersection} imply that $\Omega \cap \Pb(V) \subset {\rm Min}_{\Cc}(A)$. By  Theorem~\ref{thm:centralizers_ncc} the quotient $Z_\Gamma(A) \backslash {\rm Min}_{\Cc}(A)$ is compact and $\Omega \cap \Pb(V)$ is closed, so $Z_\Gamma(A) \backslash \Omega \cap \Pb(V)$ is also compact.

(4): This follows from Lemma~\ref{lem:computing the intersection}.
\end{proof}

\subsection{The failure of Theorem~\ref{thm:center} in the naive convex co-compact setting}\label{sec:failure of center theorem for naive convex co-compact groups}

In this subsection we construct examples of naive convex co-compact groups with a centralizer that does not satisfy the conclusions of Theorem~\ref{thm:center}. The first example is quite simple.

\begin{example} Let $\Omega:=\{ [x_1:x_2:x_3] : x_1,x_2,x_3 >0 \}$ be a two dimensional simplex and let $\Gamma:=\langle a \rangle$ where 
$$
a=\begin{bmatrix} 1 & 0 & 0 \\0 & 2 & 0 \\ 0 & 0 & 2\end{bmatrix}.
$$  
 Then $\Cc:=\Omega \cap \Pb(\Span\{[1:0:0], [0:1:1]\})$ is an $\Gamma$-invariant projective line segment and $\Gamma \backslash \Cc \cong \Sb^1$. So $\Gamma \leq \Aut(\Omega)$  is a naive convex co-compact group. However, since $a$ fixes all three vertices of the simplex $\Omega$, with the notation in Theorem~\ref{thm:center}, $V = \Rb^3$. So $\Omega = \Omega \cap \Pb(V) \not\subset \Cc$ and $Z_{\Gamma}(\Gamma)=\Gamma$ does not act co-compactly on $\Omega = \Omega \cap \Pb(V)$. Thus this example fails to satisfy parts (2) and (3) in Theorem~\ref{thm:center}. 
\end{example}

\begin{example} Let $\Hb^2$ denote real hyperbolic 2-space and fix a convex co-compact subgroup $\Gamma_0 \leq \Isom(\Hb^2)$ which is isomorphic to a free group on two generators. Then, by definition, there exists a closed $\Gamma_0$-invariant non-empty geodesically convex subset $\Cc_0 \subsetneq \Hb^2$ where $\Gamma_0 \backslash \Cc_0$ is compact. 

Using the Klein-Beltrami model, we can identify 
$$
\Hb^2=\left\{ [x_1 : x_2 : x_3]  \in \Pb(\Rb^3) : x_1^2 + x_2^2 < x_3^2\right\}
$$
and $\Isom(\Hb^2) = \Aut(\Hb^2) = {\rm PO}(2,1)$. Since geodesics in $\Hb^2$ are projective line segments, $\Cc_0$ is a convex set in the real projective sense.

Next let 
$$
\Omega := \left\{ [1:x_1 : x_2 : x_3 ] \in \Pb(\Rb^4) : x_3 > 0, \, x_1^2 + x_2^2 < x_3^2\right\}
$$
and 
$$
\Cc := \left\{ [1: x_1 : x_2 : x_3] \in \Omega: [x_1 : x_2 : x_3] \in \Cc_0 \right\}.
$$
Then $\Omega$ is a cone with base $\Hb^2$ and $\Cc$ is a cone with base $\Cc_0$. Define 
$$
\Gamma := \Aut(\Omega) \cap \left\{ \begin{bmatrix} 2^{n} & \\ & 2^{-n} A \end{bmatrix} \in \PGL_4(\Rb) : A \in \GL_3(\Rb), \, \det(A) = \pm 1, \, [A] \in \Gamma_0, n \in \Zb \right\}.
$$
Notice that if $[A] \in \Gamma_0$, then exactly one of 
$$
\begin{bmatrix} 1 & \\ & A \end{bmatrix}, \begin{bmatrix} 1 & \\ & -A \end{bmatrix} \in \PGL_4(\Rb)
$$
is contained in $\Aut(\Omega)$. Hence the map 
$$
\begin{bmatrix} 2^{n} & \\ & 2^{-n} A \end{bmatrix}  \in \Gamma \mapsto (n, [A]) \in \Zb \times \Gamma_0
$$
is an isomorphism. This implies that $\Gamma$ acts co-compactly on $\Cc$ and so $\Gamma \leq \Aut(\Omega)$ is a naive convex co-compact group.

Consider the element
$$
g:=\begin{bmatrix} 2^{n} & \\ & 2^{-n} \id_3 \end{bmatrix} \in \Gamma. 
$$
Then the centralizer of $g$ in $\Gamma$ is $Z_\Gamma(g) = \Gamma$. Further, with the notation in Theorem~\ref{thm:center},
$$
V =\Span\{ x \in \Cc : gx = x \} = \Rb^4
$$
and so $\Omega = \Omega \cap \Pb(V)$. Thus this example fails to satisfy parts (2) and (3) in Theorem~\ref{thm:center}. 
\end{example}

\section{The classification of 3-manifolds admitting convex co-compact representations}
\label{sec:pf_main_thm}

In this section we prove Theorem~\ref{thm:main} (in Propositions~\ref{prop:geom_case} and~\ref{prop:non_geom_case}) and Proposition~\ref{prop:boring_examples} (in Proposition~\ref{prop:geom_case}).

We start by recalling the following facts about convex co-compact groups. 

\begin{proposition}\label{prop:cc_observations} Suppose $\Omega \subset \Pb(\Rb^d)$ is a properly convex domain and $\Gamma \leq \Aut(\Omega)$ is convex co-compact. 
\begin{enumerate}
\item If $x \in \partiali\Cc_\Omega(\Gamma)$, then 
\begin{align*}
F_\Omega(x) \subset \partiali\Cc_\Omega(\Gamma).
\end{align*} 
\item If $\Gamma_1 \leq \Gamma$ has finite index, then $\Gamma_1$ is convex co-compact in $\Aut(\Omega)$ and $\Cc_\Omega(\Gamma_1)= \Cc_\Omega(\Gamma)$.
\item If $S \leq \Gamma$ is virtually solvable, then $S$ is virtually Abelian. 
\item If $V \subset \Rb^d$ is a $\Gamma$-invariant linear subspace and $\Omega \cap\Pb(V) \neq \emptyset$, then $\Cc_\Omega(\Gamma)\subset \Omega \cap \Pb(V)$.
\end{enumerate}
\end{proposition}

\begin{proof} (1). Follows from the definition and Proposition~\ref{prop:dynamics_of_automorphisms_2}.

(2). Follows from the definitions. 

(3). Let $\wh{S} \leq \GL_d(\Rb)$ be the preimage of $S$ under the projection $\GL_d(\Rb) \rightarrow \PGL_d(\Rb)$. By replacing $S$ with a finite index subgroup, we can assume that the Zariski closure of $\wh{S}$ in $\GL_d(\Cb)$ is connected. By Lie's Theorem there exists $g \in \GL_d(\Cb)$ such that $g \wh{S} g^{-1}$ is a subgroup of the complex upper triangular matrices. Then 
\begin{align*}
g\left[ \wh{S},\wh{S}\right] g^{-1}=\left[g \wh{S} g^{-1},g \wh{S} g^{-1}\right]
\end{align*}
is a subgroup of the complex upper triangular matrices with ones on the diagonal. Hence every element of $\left[\wh{S},\wh{S}\right]$ is unipotent which implies by~\cite[Theorem 1.16 (C)]{DGK2017} that $\left[\wh{S},\wh{S}\right]=1$. Thus $\wh{S}$ is Abelian.

(4). Note that $\Omega \cap \Pb(V)$ is a non-empty closed $\Gamma$-invariant properly convex subset of $\Omega$. By \cite[Lemma 4.16]{DGK2017}, $\Cc_{\Omega}(\Gamma)$ is the minimal such set. Thus $\Cc_\Omega(\Gamma) \subset \Omega \cap \Pb(V)$.
\end{proof}

By \Cref{rem:irred_implies_constraint_on _geom}, an irreducible 3-manifold cannot have $\Rb \times \Sb^2$ geometry. Hence the following proposition implies Theorem~\ref{thm:main} part (1) and  Proposition~\ref{prop:boring_examples}.

\begin{proposition}\label{prop:geom_case} Suppose $M$ is a closed geometric 3-manifold, $\rho : \pi_1(M) \rightarrow \PGL_d(\Rb)$ is a convex co-compact representation, and $\Omega \subset \Pb(\Rb^d)$ is a properly convex domain where $\Gamma:=\rho(\pi_1(M)) \leq \Aut(\Omega)$ is convex co-compact. Then:
\begin{enumerate}
\item $M$ has either $\Rb \times \Sb^2$, $\Rb^3$, $\Rb \times \Hb^2$, or $\Hb^3$ geometry. \label{conc:geom_case_1}
\item If $M$ has $\Rb^3$ or $\Rb \times \Hb^2$ geometry, then there exists a four dimensional linear subspace $V \subset \Rb^d$ such that 
\begin{align*}
\Cc_\Omega(\Gamma) = \Omega \cap \Pb(V).
\end{align*}
\item If $M$ has $\Rb^3$ geometry, then $\Cc_\Omega(\Gamma)$ is a properly embedded simplex in $\Omega$,
\item If $M$ has $\Rb \times \Hb^2$ geometry, then $\Cc_\Omega(\Gamma)$ is a properly embedded cone in $\Omega$ with strictly convex base. 
\end{enumerate}
\end{proposition}

\begin{proof} We first observe that it is enough to prove the proposition for a finite cover $M^\prime$ of $M$. For part (1), any geometric structure on $M$ lifts to a geometric structure on $M^\prime$ and a closed 3-manifold can only have one type of geometric structure, see \Cref{prop:closed-mfld-unique-gstr}. So it is enough to show that $M^\prime$ has geometry either $\Rb \times \Sb^2$, $\Rb^3$, $\Rb \times \Hb^2$, or $\Hb^3$. Further, by Proposition~\ref{prop:cc_observations} the representation $\rho|_{\pi_1(M^\prime)}$ is convex co-compact and has the same convex hull as $\rho$. So it is enough to prove parts (2)-(4) for $\Gamma^\prime:= \rho(\pi_1(M^\prime))$. Thus we will freely pass to finite covers throughout the proof.

By definition $\Gamma = \rho(\pi_1(M))$ is infinite and so $M$ does not have geometry $\Sb^3$. Proposition~\ref{prop:cc_observations} part (3) implies that $M$ does not have ${\rm Sol}$ or ${\rm Nil}$ geometry: in these cases $\pi_1(M)$ is virtually solvable, but not virtually Abelian (see~\Cref{prop:nil_or_sol_geom}). If $M$ has $\Rb \times \Sb^2$ or $\Hb^3$ geometry, then there is nothing left to prove, so we can assume that $M$ has either $\Rb \times \Hb^2$, $\Rb^3$, or $\wt{\SL_2(\Rb)}$ geometry. By \Cref{prop:seifert_geometric_manifolds}, $M$ is Seifert fibered. 

Further the universal cover $\wt{M}$ of $M$ is diffeomorphic to $\Rb^3$. So $\pi_1(M)$ is torsion free and has cohomological dimension 3.  Since $\pi_1(M)$ is torsion free, $\ker \rho = 1$ by definition. 

Since $M$ is Seifert fibered, $\pi_1(M)$ contains an infinite normal cyclic subgroup $N = \ip{h}$, see \Cref{prop:seifert-alg}. 
Then, since $\Aut(N) \cong \Aut(\Zb) \cong \Zb/2\Zb$, the centralizer $Z_{\pi_1(M)}(h)$ has finite index in $\pi_1(M)$.  Thus by replacing $M$ with a finite cover we can assume that $\pi_1(M) = Z_{\pi_1(M)}(h)$. Then $\Gamma=Z_{\Gamma}(\rho(h))$

Then, by Theorem~\ref{thm:center} there exists a linear subspace $V \subset \Rb^d$ such that 
\begin{enumerate}[label={(\alph*)}]
\item $\Omega \cap \Pb(V)$ is a non-empty $\Gamma$-invariant properly convex domain in $\Pb(V)$,
\item\label{list:intersection is in convex hull} $\Omega \cap \Pb(V) \subset \Cc_\Omega(\Gamma)$, 
\item the quotient $\Gamma \backslash \Omega \cap \Pb(V)$ is compact, and
\item\label{list:e} there exists a $\Gamma$-invariant non-trivial direct sum decomposition $V = \oplus_{j=1}^m V_j$ where $\rho(h)$ acts on each $V_j$ by scaling and there exist properly convex domains $F_j \subset \Pb(V_j)$ such that 
\begin{align*}
\Omega \cap \Pb(V) = {\rm relint}\left(\ConvHull_{\overline{\Omega}}\left( \cup_{j=1}^m F_j \right) \right). 
\end{align*}
\end{enumerate}

Since $\Omega \cap \Pb(V)$ is diffeomorphic to $\Rb^{\dim V-1}$ and $\Gamma \cong \pi_1(M)$ has cohomological dimension 3, we must have $\dim V = 4$. Further, by Proposition~\ref{prop:cc_observations} part (4) and property~\ref{list:intersection is in convex hull} we have $\Cc_\Omega(\Gamma)=\Omega \cap \Pb(V)$. This proves part (2).

Let 
\begin{align*}
\Gamma_V :=\{ g|_V \in \PGL(V) : g \in \Gamma\}.
\end{align*}
 The map $g \in\Gamma \mapsto g|_V \in \PGL(V)$ is proper and injective since $\Gamma$ acts properly on $\Omega$ and $\Gamma$ is torsion free. Hence $\Gamma_V \cong \Gamma$ and $\Gamma_V$ is discrete.

Since $\dim V  = 4$, up to relabelling we have four cases for $(\dim V_1,\dots,\dim V_m)$: (1,1,1,1), (1,1,2), (2,2), and (1,3). Notice that if $\dim V_j=1$, then $F_j =\Pb(V_j)$ is a point in $\Pb(\Rb^d)$ and if $\dim V_j = 2$, then $F_j$ is an open line segment in $\Pb(\Rb^d)$. 

In the first three cases $\Omega \cap \Pb(V)$ is a simplex.  Then, by the description of $\Aut(\Omega \cap\Pb(V))$ in Example~\ref{ex:basic_properties_of_simplices}, $\pi_1(M) \cong \Gamma \cong \Gamma_V$ is virtually isomorphic to $\Zb^3$. So \Cref{prop:nil_or_sol_geom} implies that $M$ has $\Rb^3$ geometry.

In the last case, $\Omega \cap \Pb(V)$ is a cone with base $F_2$. By a result of Benoist~\cite[Proposition 4.4]{B2003}, there exists a discrete subgroup of $\Aut(F_2)$ which acts co-compactly on $F_2$ (i.e. $F_2$ is divisible). Hence by a result of Kuiper~\cite{Kuiper1954}, either $F_2$ is a simplex or $F_2$ is a strictly convex domain. If $F_2$ is a simplex, then $\Omega \cap \Pb(V)$ is a simplex and once again $\pi_1(M)$ is virtually isomorphic to $\Zb^3$. So in this case $M$ has $\Rb^3$ geometry. 

It remains to consider the case when $F_2$ is a strictly convex domain. Then we can identify $V$ with $\Rb^4$ so that 
\begin{align*}
V_1 = \Rb \times \{(0,0,0)\} \quad \text{and} \quad V_2 = \{0\} \times \Rb^3.
\end{align*}
Then by property~\ref{list:e} above
\begin{align*}
\Gamma_V \leq \left\{ \begin{bmatrix} \lambda & \\ & A \end{bmatrix} \in\PGL_4(\Rb): \lambda \in\Rb^{\times} \text{ and } A \in \GL_3(\Rb) \right\}.
\end{align*}
Passing to a finite cover of $M$ we can assume that $\frac{\det(A)}{\lambda^3}>0$ for all elements of $\Gamma_V$.

Let $\tau=(\tau_1,\tau_2): \Gamma_V \rightarrow (\Rb,+) \times \PGL_3(\Rb)$ be the homomorphism
\begin{align*}
\tau \left( \begin{bmatrix} \lambda & \\ & A \end{bmatrix} \right) = \left(\log \frac{\det(A)}{\lambda^3}, [A] \right).
\end{align*}
Then $\tau$ is injective and is proper (i.e. has discrete image). 

Let $\Gamma_2:=\tau_2(\Gamma_V)$. Then $\Gamma_2 \leq \Aut(F_2)$ and the quotient $\Gamma_2 \backslash F_2$ is compact. 
We claim that $\Gamma_2$ is discrete. Suppose $(g_n)_{n \geq1}$ is a sequence in $\Gamma_V$ and $\tau_2(g_n) \rightarrow \id$. By property~\ref{list:e} above, $\tau(\rho(h)) =(\lambda, \id)$ for some non-zero $\lambda\in\Rb$. So we can find a sequence $(m_n)_{n \geq 1}$ in $\Zb$ such that $\tau(\rho(h)^{m_n} g_n)$ is relatively compact in $(\Rb,+) \times \PGL_3(\Rb)$. Since $\tau$ is proper, the set  $\{\rho(h)^{m_n} g_n : n \geq 1\}$ is finite. Then, since $\tau_2(g_n) \rightarrow \id$, we must have $\tau_2(g_n) =\id$ for $n$ sufficiently large. Thus $\Gamma_2$ is discrete.

Applying Selberg's lemma to $\Gamma_2$, we can replace $M$ with a finite cover and assume $\Gamma_2$ is torsion free. Then $\Gamma_2$ acts properly discontinuously, freely, and co-compactly on $F_2$. Since $F_2$ is strictly convex, $\Gamma_2$ is Gromov hyperbolic by a result of Benoist~\cite{B2004}. Thus $\Sigma := \Gamma_2\backslash F_2$ is a hyperbolic surface with $\pi_1(\Sigma) \cong \Gamma_2$. So there exists a proper injective homomorphism $\Gamma_2 \hookrightarrow \Isom(\Hb^2)$.

Then, since $\Gamma_V \cong \Gamma \cong \pi_1(M)$, there exists a proper injective homomorphism 
\begin{align*}
\pi_1(M) \hookrightarrow (\Rb,+) \times \Isom(\Hb^2).
\end{align*} Let $G$ be the image of this map. Then $G$ acts properly discontinuously and freely on $\Rb \times \Hb^2$. Since $G \cong \pi_1(M)$ has cohomological dimension 3, the quotient $N := G \backslash \Rb \times \Hb^2$ is a closed 3-manifold with $\Rb \times \Hb^2$ geometry. Then, since $\pi_1(N)\cong \pi_1(M)$, the manifolds $N$ and $M$ are homeomorphic~\cite{S1983b}. So $M$ also has $\Rb \times \Hb^2$ geometry.
\end{proof}

 \begin{proposition}\label{prop:non_geom_case} Suppose $M$ is a closed irreducible orientable 3-manifold. If $\rho : \pi_1(M) \rightarrow \PGL_d(\Rb)$ is a convex co-compact representation and $M$ is non-geometric, then every component in the geometric decomposition of $M$ is hyperbolic. 
 \end{proposition}
\begin{proof} Suppose, for a contradiction, that there exists a component $S$ in the geometric decomposition that is not hyperbolic. Then, \Cref{prop:seifert-piece} implies that $S$ is Seifert fibered and has $\wt{\SL_2(\Rb)}$ or $\Rb \times \Hb^2$ geometry.

Let $N = \ip{h}$ denote the infinite cyclic subgroup in $\pi_1(S)$ generated by a regular fiber, see \Cref{prop:seifert-alg}. Then, since $\Aut(N) \cong \Zb/2\Zb$, the centralizer $Z_{\pi_1(S)}(h)$ has finite index in $\pi_1(S)$. Further, by \Cref{prop:seifert-piece} part (2),  
\begin{align*}
Z_{\pi_1(S)}(h) = Z_{\pi_1(M)}(h).
\end{align*}
Since $\pi_1(M)$ is torsion-free, see for instance~\cite[(C.1)]{AFW2015}, and $\ker \rho$ is finite by definition, we see that $\rho$ is injective. Then $Z_{\pi_1(S)}(h)$, and hence $\pi_1(S)$, is virtually the fundamental group of a closed aspherical $k$-manifold by Corollary~\ref{cor:fund_gp_of_manifold}. We will show that this is impossible.

Let $\Sigma$ be the base orbifold given by the Seifert fibration of $S$. By \Cref{prop:seifert_geometric_manifolds}, $\Sigma$ is homeomorphic to a finite area non-compact hyperbolic 2-orbifold (recall that $S$ has $\wt{\SL_2(\Rb)}$ or $\Rb \times \Hb^2$ geometry). Further, by \Cref{prop:seifert-alg}, we have a short exact sequence 
\begin{align*}
1 \rightarrow N \rightarrow \pi_1(S) \rightarrow \pi_1(\Sigma) \rightarrow 1. 
\end{align*}
Fix a finite cover $\Sigma^\prime \rightarrow \Sigma$ such that $\Sigma^\prime$ is a manifold. Next let $G$ be the preimage of $\pi_1(\Sigma^\prime)$ under the map $\pi_1(S) \rightarrow \pi_1(\Sigma)$. Then $G$ has finite index in $\pi_1(S)$ and
\begin{align*}
1 \rightarrow N \rightarrow G \rightarrow \pi_1(\Sigma^\prime) \rightarrow 1
\end{align*}
is a short exact sequence. Since $\Sigma^\prime$ is a finite volume non-compact hyperbolic surface, there exists $m \geq 2$ such that $\pi_1(\Sigma^\prime) \cong \mathsf{F}_m$ where $\mathsf{F}_m$ is the free group on $m$ generators. Hence the short exact sequence splits and $G \cong N \times \pi_1(\Sigma^\prime) \cong \Zb \times  \mathsf{F}_m$.

Since $G$ has finite index in $\pi_1(S)$, our assumption implies that $G$ is virtually the fundamental group of a closed aspherical $k$-manifold. Since $G$ has cohomological dimension two, $k=2$. Since $G$ contains a non-Abelian free subgroup, $G$ is not virtually the fundamental group of a torus. Since $G$ has infinite center, $G$ is not virtually the fundamental group of a closed hyperbolic surface.  So we have a contradiction. 
\end{proof}

\section{Proof of Theorem~\ref{thm:anosov_intro}}\label{sec:Anosov}

In this section we prove Theorem~\ref{thm:anosov_intro} using the following result of Danciger--Gu\'{e}ritaud--Kassel about projective Anosov representations.

\begin{theorem}\cite[Theorem 1.15]{DGK2017}\label{thm:DGK} Suppose $\Gamma$ is a word hyperbolic group. If $\rho : \Gamma \rightarrow \PGL_d(\Rb)$ is a representation, then the following are equivalent:
\begin{enumerate}
\item  $\rho$ is convex co-compact 
\item $\rho$ is projective Anosov and $\rho(\Gamma)$ preserves a properly convex domain in $\Pb(\Rb^d)$.
\end{enumerate}
In this case, if $\Omega \subset \Pb(\Rb^d)$ is a properly convex domain such that $\rho(\Gamma) \leq \Aut(\Omega)$ is convex co-compact, $\Cc:=\Cc_\Omega(\rho(\Gamma))$, and $\xi^{(1)}:\partial_\infty \Gamma \rightarrow \Pb(\Rb^d)$ is the Anosov boundary map, then
\begin{enumerate}[label={(\alph*)}]
\item $\xi^{(1)} \left( \partial_\infty \Gamma\right) = \partiali \Cc$,
\item $\partiali \Cc$ contains no non-trivial line segments, and
\item every point in $\partiali \Cc$ is a $\Cc^1$-smooth point  of $\partial\Omega$.
\end{enumerate}
\end{theorem} 

\begin{remark} In the case of irreducible representations, Theorem~\ref{thm:DGK} was established independently by the second author~\cite{Z2017} using different terminology. 
\end{remark}

Using Theorem~\ref{thm:DGK}, Theorem~\ref{thm:anosov_intro} reduces  to the following proposition.

\begin{proposition}\label{prop:stuff_implies_cc}
Suppose $\Gamma$ is a one-ended word hyperbolic group, $\rho : \Gamma \rightarrow \PGL_d(\Rb)$ is a projective Anosov representation, and $\xi^{(1)}:\partial_\infty \Gamma \rightarrow \Pb(\Rb^d)$ is the Anosov boundary map. 
\begin{enumerate}
\item If $\Gamma$ is not commensurable to a surface group, then $\xi^{(1)} \left( \partial_\infty \Gamma \right)$ is bounded in some affine chart of $\Pb(\Rb^d)$.   
\item If $\xi^{(1)} \left( \partial_\infty \Gamma \right)$ is bounded in some affine chart of $\Pb(\Rb^d)$, then $\rho(\Gamma)$ preserves a properly convex domain in $\Pb(\Rb^d)$. 
\end{enumerate}
\end{proposition}

\begin{remark} In the case of irreducible representations, Part (2) was established by Canary--Tsouvalas~\cite[Proposition 2.8]{CK2020}. \end{remark}

Before proving this proposition we recall the definition of Anosov representations into $\PGL_d(\Rb)$ and then show that certain subsets of $\Pb(\Rb^d)$ have a well defined convex hull. Finally, we prove Proposition~\ref{prop:stuff_implies_cc} in the last subsection of this section.

\subsection{Anosov representations} To avoid the theory of semisimple Lie groups we only define Anosov representations into the general linear group.  In this case, Anosov representations are representations with exponential growth rate of singular values and controlled asymptotic behavior. To give the precise definition we need to introduce some terminology. 

If $g \in \PGL_d(\Rb)$ let 
\begin{align*}
\mu_1(g) \geq \dots \geq \mu_{d}(g)
\end{align*}
denote the singular values of some (hence any) lift of $g$ to $\SL_d^{\pm}(\Rb):=\{ h \in \GL_d(\Rb) : \det h = \pm 1\}$. 

\begin{definition} Suppose $\Gamma$ is a word hyperbolic group and $\rho: \Gamma \rightarrow \PGL_{d}(\Rb)$ is a representation.  Two maps $\xi^{(k)}: \partial_\infty \Gamma \rightarrow \Gr_k(\Rb^d)$ and $\xi^{(d-k)}: \partial_\infty \Gamma \rightarrow \Gr_{d-k}(\Rb^d)$ are called:
\begin{enumerate}
\item \emph{$\rho$-equivariant} if $\xi^{(k)}(g x) = \rho(g)\xi^{(k)}(x)$ and $\xi^{(d-k)}(g x) = \rho(g)\xi^{(d-k)}(x)$ for all $g \in \Gamma$ and $x \in \partial_\infty \Gamma$,
\item \emph{dynamics-preserving} if for every $g \in \Gamma$ of infinite order with attracting fixed point $x_g^+ \in \partial_\infty \Gamma$ the points $\xi^{(k)}(x^+_{g}) \in \Gr_k(\Rb^d)$ and $\xi^{(d-k)}(x^+_{g}) \in \Gr_{d-k}(\Rb^d)$ are attracting fixed points of the action of $\rho(g)$ on $ \Gr_k(\Rb^d)$ and $ \Gr_{d-k}(\Rb^d)$, and
\item \emph{transverse} if $\xi^{(k)}(x) + \xi^{(d-k)}(y) = \Rb^{d}$ for every distinct pair $x, y \in \partial_\infty \Gamma$.
\end{enumerate}
\end{definition}

\begin{definition} Suppose $\Gamma$ is word hyperbolic, $S$ is a finite symmetric generating set, and $d_S$ is the induced word metric on $\Gamma$. A representation $\rho: \Gamma \rightarrow \PGL_d(\Rb)$ is \emph{$P_k$-Anosov} if there exist continuous, $\rho$-equivariant, dynamics preserving, and transverse maps $\xi^{(k)}: \partial \Gamma \rightarrow \Gr_k(\Rb^d)$, $\xi^{(d-k)}: \partial \Gamma \rightarrow \Gr_{d-k}(\Rb^d)$ and constants $C,c>0$ such that
\begin{align}
\label{eq:singular_value_est}
\log  \frac{\mu_{k}(\rho(g))}{\mu_{k+1}(\rho(g))} \geq C d_S(g, \id) -c
 \end{align}
for all $g \in \Gamma$.
 \end{definition}
 
 \begin{remark} Kapovich--Leeb--Porti~\cite{KLP2014,KLP2014b} proved that if a representation of a finitely generated group satisfies the estimate in Equation~\eqref{eq:singular_value_est}, then the group is word hyperbolic and the representation is $P_k$-Anosov (also see Bochi--Potrie--Sambarino~\cite[Proposition 4.9]{BPS2019}).
 \end{remark}
 
If $\rho$ is a $P_k$-Anosov representation, the maps $\xi^{(k)}$ and $\xi^{(d-k)}$ are called the \emph{Anosov boundary maps}. Anosov representations have the following well known asymptotic behavior. 
 
 \begin{proposition}\label{prop:asym_behavior} Suppose $\Gamma$ is word hyperbolic and $\rho: \Gamma \rightarrow \PGL_d(\Rb)$ is $P_k$-Anosov. If $(g_n)_{n \geq 1}$ is a sequence in $\Gamma$ with 
 \begin{align*}
 x^+ = \lim_{n \rightarrow \infty} g_n \in \partial \Gamma \quad \text{and}  \quad x^- = \lim_{n \rightarrow \infty} g_n^{-1} \in \partial \Gamma, 
 \end{align*}
 then 
 \begin{align*}
 \xi^{(k)}(x^+) = \lim_{n \rightarrow \infty} \rho(g_n) V
 \end{align*}
 for all $V \in \Gr_k(\Rb^d)$ transverse to $\xi^{(d-k)}(x^-)$. Moreover, the convergence is uniform on compact subsets of 
 \begin{align*}
Z= \left\{  V \in \Gr_k(\Rb^d) : V \cap \xi^{(d-k)}(x^-) = \{0\}\right\}.
 \end{align*}
 \end{proposition}
 
 \begin{proof}[Proof sketch] This is a straightforward consequence of either~\cite[Theorem 5.3]{GGKW2015} or~\cite[Lemma 4.7]{BPS2019}. The $k=1$ case is explicitly given in~\cite[Lemma 8.2]{DGK2017} and the same argument works in the $k>1$ case. Alternatively, the $k>1$ case can be reduced to the $k=1$ case using the Pl\"ucker embedding and~\cite[Proposition 4.3]{GW2012}. 
 \end{proof}

As mentioned in the introduction, $P_1$-Anosov representations are often called \emph{projective Anosov representations} due to the identification  $\Gr_1(\Rb^d) = \Pb(\Rb^d)$. 

\subsection{Convex Hulls} In this section we show how to associate a ``convex hull'' to certain subsets of $\Pb(\Rb^d)$. 

A general subset of $\Pb(\Rb^d)$ has no well defined convex hull, for instance: if $X = \{x_1,x_2\} \subset \Pb(\Rb^d)$ consists of two points, then there is no natural way to select between the two projective line segments joining $x_1,x_2$. However, we will show that a subset which is connected and contained in an affine chart does indeed have a well defined convex hull. 

First, if $X \subset \Pb(\Rb^d)$ is contained in an affine chart $\Ab$, then recall that 
\begin{align*}
\ConvHull_{\Ab}(X) \subset \Ab
\end{align*}
denotes the convex hull of $X$ in $\Ab$. 

\begin{lemma} Suppose $X \subset \Pb(\Rb^d)$ is connected. If $X$ is contained in two affine charts $\Ab_1$ and $\Ab_2$, then 
\begin{align*}
\ConvHull_{\Ab_1}(X) =\ConvHull_{\Ab_2}(X).
\end{align*}
\end{lemma}

\begin{proof} It is enough to show that $\ConvHull_{\Ab_1}(X) \subset \Ab_2$. 

By changing coordinates we can assume
\begin{align*}
\Ab_j = \{ [x_1:x_2:\dots:x_d] : x_j \neq 0\}.
\end{align*}
Then $\Ab_1 \cap \Ab_2$ has two connected components, namely
\begin{align*}
A := \left\{ [1:x_2:\dots:x_d] : x_2 > 0\right\} \quad \text{and} \quad B:= \left\{ [1:x_2:\dots:x_d] : x_2 < 0\right\}.
\end{align*}
Since $X \subset \Ab_1 \cap\Ab_2$ is connected it is contained in exactly one of these components. So by possibly changing coordinates again we may assume that $X \subset A$. Since $A$ is a convex subset of $\Ab_1$ we then have $\ConvHull_{\Ab_1}(X) \subset A \subset \Ab_2$.\end{proof}

\begin{definition}\label{defn:CH} If $X \subset \Pb(\Rb^d)$ is connected and contained in some affine chart, then let 
\begin{align*}
\ConvHull(X) \subset \Pb(\Rb^d)
\end{align*}
denote the convex hull of $X$ in some (hence any) affine chart which contains $X$. \end{definition}

As a consequence of the definition we have the following. 

\begin{observation}\label{obs:CH_invariance} Suppose $X \subset \Pb(\Rb^d)$ is connected and contained in some affine chart. If $g \in \PGL_d(\Rb)$, then 
\begin{align*}
g \ConvHull(X)=\ConvHull(gX).
\end{align*}
\end{observation}

\subsection{Proof of Proposition~\ref{prop:stuff_implies_cc}} Suppose $\rho : \Gamma \rightarrow \PGL_d(\Rb)$ satisfies the hypothesis of Proposition~\ref{prop:stuff_implies_cc}. Let $\xi^{(1)}: \partial_\infty \Gamma \rightarrow \Pb(\Rb^d)$ and $\xi^{(d-1)}:\partial_\infty\Gamma \rightarrow \Gr_{d-1}(\Rb^d)$ denote the Anosov boundary maps. For ease of notation, we will view each $\xi^{(d-1)}(x)$ as a subset of $\Pb(\Rb^d)$,  that is we will identify $\xi^{(d-1)}(x)$ with its projectivization in $\Pb(\Rb^d)$.

\subsubsection{Proof of part (1)} Since $\Gamma$ is one-ended and not commensurable to a surface group, a number of deep results about hyperbolic groups~\cite{T1988,G1992,S1996} imply that there exist $v_1,v_2 \in \partial_\infty \Gamma$ distinct such that $\partial_\infty \Gamma \setminus \{v_1,v_2\}$ is connected (see~\cite[Theorem 2.5]{Z2017} for details).

By changing coordinates we can assume that 
\begin{align*}
\xi^{(d-1)}(v_j)=  \{ [x_1:x_2:\dots:x_d] \in \Pb(\Rb^d) : x_j = 0\}.
\end{align*}
Then 
\begin{align*}
\xi^{(1)}(\partial_\infty \Gamma \setminus \{v_1,v_2\}) \subset \Pb(\Rb^d) \setminus \left(\xi^{(d-1)}(v_1) \cup \xi^{(d-1)}(v_2)\right)
\end{align*}
and since $\xi^{(1)}(\partial_\infty \Gamma \setminus \{v_1,v_2\})$ is connected by changing coordinates we can assume that 
\begin{align*}
\xi^{(1)}(\partial_\infty \Gamma \setminus \{v_1,v_2\}) \subset  \{ [x_1:\dots:x_d]  \in \Pb(\Rb^d): x_1 > 0 \text{ and } x_2 > 0\}.
\end{align*}
Then $\xi^{(1)}(\partial_\infty\Gamma)$ is bounded in the affine chart 
\begin{align*}
 \left\{ [x_1:x_2:\dots:x_d]  \in \Pb(\Rb^d): x_1+x_2 \neq 0\right\}.
 \end{align*}
 
\subsubsection{Proof of part (2)} Since $\Gamma$ is a one-ended word hyperbolic group, a result of Swarup~\cite{S1996} implies that
\begin{enumerate}
\item $\partial_\infty \Gamma$ is connected and
\item $\partial_\infty \Gamma \setminus \{x\}$ is connected for every $x \in \partial_\infty \Gamma$.
\end{enumerate}

Fix an affine chart $\Ab \subset \Pb(\Rb^d)$ which contains $\xi^{(1)}(\partial_\infty\Gamma)$. Since $\xi^{(1)}(\partial_\infty\Gamma)$ is connected, $\xi^{(1)}(\partial_\infty\Gamma)$ has a well defined convex hull $C_0$ (in the sense of Definition~\ref{defn:CH} above). Further, since $\xi^{(1)}(\partial_\infty \Gamma)$ is compact and contained in $\Ab$, the set $C_0$ is bounded in $\Ab$. 

\begin{lemma}\label{lem:Anosov_no_intersection} $\xi^{(d-1)}(x) \cap \relint(C_0) = \emptyset$ for all $x \in \partial_\infty \Gamma$. \end{lemma}

\begin{proof} Fix $x \in\partial_\infty \Gamma$. Then $\xi^{(1)}(x) \in \xi^{(d-1)}(x) \cap \Ab$ and so $ \xi^{(d-1)}(x) \neq \Pb(\Rb^d) \setminus \Ab$. So by changing coordinates we can assume that 
\begin{align*}
\Ab =  \{ [x_1:x_2:\dots:x_d] \in \Pb(\Rb^d) : x_1 \neq 0\}
\end{align*}
and
\begin{align*}
\xi^{(d-1)}(x) =  \{ [x_1:x_2:\dots:x_d]  \in \Pb(\Rb^d): x_2 = 0\}.
\end{align*}
Then $\Ab \setminus\ker \xi^{(d-1)}(x)$ has two connected components
\begin{align*}
Y_1:= \left\{ [1:x_2:\dots:x_d] :  x_2 > 0\right\} \quad \text{and} \quad Y_2:= \left\{ [1:x_2:\dots:x_d] : x_2 < 0\right\}.
\end{align*}
Since $\partial_\infty \Gamma-\{x\}$ is connected,  
\begin{align*}
\xi^{(1)}(\partial_\infty \Gamma-\{x\}) \subset Y_j
\end{align*}
for some $j \in \{1,2\}$. Since $Y_j$ is convex in $\Ab$, we then have $\relint(C_0) \subset Y_j$. So $\xi^{(d-1)}(x) \cap \relint(C_0) = \emptyset$.
\end{proof}

Now fix $p \in \relint(C_0)$ and a bounded neighborhood $N$ of $C_0$ in $\Ab$. 

\begin{lemma} There exists a connected open neighborhood $U$ of $p$ such that 
\begin{align*}
\bigcup_{g \in\Gamma} \rho(g) U \subset N.
\end{align*}
\end{lemma}

\begin{proof} This is an immediate consequence of Proposition~\ref{prop:asym_behavior}. Suppose such a neighborhood does not exist. Then there exist sequences $(p_n)_{n \geq 1}$ in $\Ab$  and $(g_n)_{n \geq1}$ in $\Gamma$ such that $p_n \rightarrow p$ and $\rho(g_n) p_n \notin N$ for all $n$. By passing to a subsequence we can suppose that $g_n \rightarrow g_\infty \in \Gamma \cup \partial_\infty \Gamma$. If $g_\infty \in \Gamma$, then 
\begin{align*}
\lim_{n \rightarrow \infty} \rho(g_n) p_n=\rho(g_\infty)p \in C_0
\end{align*}
and hence $\rho(g_n) p_n \in N$ for $n$ sufficiently large. So we must have $g_\infty \in \partial_\infty \Gamma$. In this case, Lemma~\ref{lem:Anosov_no_intersection} implies that 
\begin{align*}
p \notin \bigcup_{y \in \partial_\infty \Gamma} \xi^{(d-1)}(y)
\end{align*}
and so by Proposition~\ref{prop:asym_behavior}
\begin{align*}
\lim_{n \rightarrow \infty} \rho(g_n) p_n=\xi^{(1)}(g_\infty) \in C_0.
\end{align*}
Hence $\rho(g_n) p_n \in N$ for $n$ sufficiently large. So we have a contradiction and thus such a neighborhood $U$ exists.
\end{proof}

Let $C_1$ be the convex hull of 
\begin{align*}
\xi^{(1)}(\partial_\infty\Gamma) \cup \bigcup_{g \in\Gamma} \rho(g) U
\end{align*}
 (in the sense of Definition~\ref{defn:CH} above) and let $\Omega = \relint(C_1)$. Then Observation~\ref{obs:CH_invariance}  implies that $\rho(\Gamma)\leq \Aut(\Omega)$.  Further, since $U$ is open, $\Omega$ is a properly convex domain.

 \section{Structure of the non-geometric examples}
 \label{sec:struct-non-geom}
 
 Suppose $M$ is a closed irreducible orientable non-geometric 3-manifold and $\rho : \pi_1(M) \rightarrow \PGL_d(\Rb)$ is a convex co-compact representation. Let $\Omega \subset \Pb(\Rb^d)$ be a properly convex domain such that $\Gamma:=\rho(\pi_1(M)) \leq \Aut(\Omega)$ is a convex co-compact subgroup and let $\Cc:=\Cc_\Omega(\Gamma)$.

We first prove that $\Gamma$ is a relatively hyperbolic group. 
\begin{proposition}
\label{prop:non_geom_eg_rel_hyp_fund_gp}
The group $\Gamma=\rho(\pi_1(M))$ is relatively hyperbolic with respect to $\{ \rho(\pi_1(T)): T \in \Tc\}$, where $\Tc$ is a  collection of embedded tori  and Klein bottle in the geometric decomposition (see \Cref{thm:geom-decomp}) of $M$. Moreover, each $\rho(\pi_1(T))$ is virtually isomorphic to $\Zb^2$.
\end{proposition}

\begin{proof} By Theorem~\ref{thm:main} part (2), every component in the geometric decomposition of $M$ (i.e. every component in $M-\Tc$) is hyperbolic. Then by Dahmani's~\cite{D2003} combination theorem $\pi_1(M)$ is relatively hyperbolic with respect to $\{\pi_1(T) : T \in \Tc \}$.  Moreover, each $\pi_1(T)$ is virtually isomorphic to $\Zb^2$, since each $T \in \Tc$ is an embedded tori or a Klein bottle. Since $\ker \rho$ is finite by definition, the result follows. 
\end{proof}

 In the rest of this section, we will discuss the structure of $\Cc$ and $M$, construct an equivariant map between boundary quotients (see Section~\ref{subsec:boundary-map}), and prove the minimality of $\Gamma$ action on $\partiali \Cc$ (see Section~\ref{subsec:minimality}).

 \subsection{Structure of $\Cc$} 
\label{subsec:struct-C} In this subsection we describe some consequences of the results in~\cite{IZ2019b}. 
 
Let $\Sc$ be the collection of \textbf{all} properly embedded simplices in $\Cc$ of dimension at least two. By Theorems 1.7 and 1.8 in~\cite{IZ2019b}, $\Sc$ has the following properties:

\begin{enumerate}[label={(a.\arabic*)}]
\item\label{item:b3} $(\Cc,\hil)$ is relatively hyperbolic with respect to $\Sc$.
\item $\Sc$ is closed and discrete in the local Hausdorff topology (induced by $\hil$).
\item $\Sc$ is $\Gamma$-invariant, i.e. if $S \in \Sc$ and $g \in \Gamma$, then $g S \in \Sc$.
\item\label{item:b4} If $S_1,S_2 \in \Sc$ are distinct, then $\partial S_1 \cap \partial S_2 = \emptyset$.
\item Each  quasi-isometrically embedded Euclidean plane in $\Cc$  is  contained in the bounded neighborhood of some $S \in \Sc$.
\item\label{item:b5} Every line segment in $\partiali\Cc$ is contained in the boundary of a simplex in $\Sc$.
\item If $x \in \partiali \Cc$ is not a $\Cc^1$-smooth point of $\partial \Omega$, then there exists  $S \in \Sc$ with $x \in \partial S$.
\end{enumerate}
Further, Theorem 1.7 in~\cite{IZ2019b} implies the following correspondence between simplices in $\Sc$ and Abelian subgroups of $\Gamma$:
\begin{itemize}
\item If $S \in \Sc$, then $S$ is two dimensional, $\Stab_{\Gamma}(S)$ acts co-compactly on $S$, and $\Stab_{\Gamma}(S)$ is virtually isomorphic to $\Zb^2$. 
\item If $A \leq \Gamma$ is an Abelian subgroup with rank at least two, then $A$ is virtually isomorphic to $\Zb^2$ and there exists a unique $S \in \Sc$ such that $A \leq \Stab_{\Gamma}(S)$.  
\end{itemize}
 
Notice that Proposition~\ref{prop:cc_observations} part (1) and Properties~\ref{item:b4}, ~\ref{item:b5} imply that 
\begin{align}
\label{eq:faces_of_simplices}
\partial S = \bigcup_{x \in \partial S} F_\Omega(x) \quad \text{for all} \quad S \in \Sc.
\end{align}

\subsection{Structure of $M$} 
\label{subsec:struct-M}
Since all the geometric components of $M$ are hyperbolic, using a result of Leeb, we can assume that $M$ is a non-positively curved Riemannian manifold~\cite[Theorem 3.3]{L1995}. Let $\wt{M}$ be the universal cover of $M$ endowed with the Riemannian metric making the covering map $\wt{M} \rightarrow M$ a local isometry. Then let $\wt{M}(\infty)$ be the geodesic boundary of $\wt{M}$. Results of Hruska--Kleiner~\cite{HK2005} then imply that there exists a collection $\Fc$ of isometrically embedded Euclidean planes in $\wt{M}$ which satisfies properties similar to $\Sc$ as above.  In particular, $\wt{M}$ is a CAT(0) \emph{space with isolated flats}, in the terminology used in \cite{HK2005}.
\begin{enumerate}[label={(b.\arabic*)}]
\item\label{item:a3} $\wt{M}$ is relatively hyperbolic with respect to $\Fc$.
\item $\Fc$ is closed and discrete in the local Hausdorff topology (induced by the Riemannian metric).
\item $\Fc$ is $\pi_1(M)$-invariant, i.e. if $F \in \Fc$ and $g \in \pi_1(M)$, then $g F \in \Fc$.
\item\label{item:a4} If $F_1,F_2 \in \Fc$ are distinct, then $F_1(\infty) \cap F_2(\infty) = \emptyset$.
\item\label{item:a6}  Each quasi-isometrically embedded Euclidean plane in $\wt{M}$ is contained in the bounded neighborhood of some flat in $\Fc$.
\item\label{item:a5} Each connected component of the Tits boundary is either an isolated point or the boundary of a flat in $\Fc$.
\end{enumerate}

\begin{remark} Recall that the Tits boundary is $\wt{M}(\infty)$ endowed with the Tits metric $\dist_T$, see for instance~\cite[Chapter II, Section 4]{B1995}. We will not need the precise definition of $\dist_T$ in the arguments that follow, only property~\ref{item:a5} and the following facts \cite[Theorem 4.11]{B1995}: 
\begin{enumerate}
\item Any two points are in the same connected component in the Tits boundary if and only if the distance between them is finite.
\item If $\dist_T(\xi,\xi')> \pi $, then there exists a bi-infinite geodesic in $\wt{M}$ whose endpoints at infinity are $\xi$ and $\xi'$.
\item $\dist_T$ is a lower semi-continuous function.  
\end{enumerate}
 \end{remark}

 \subsection{Relative Fellow Traveller Property}
\begin{definition}\label{defn:fellow_traveller}
Suppose $(X,\dist)$ is a metric space and $\Sc' \subset X$ is a family of subsets of $X$. We will say that $(X,\dist)$ satisfies the \emph{relative fellow traveller property relative to $\Sc'$} if for any $\alpha \geq 1$ and $\beta \geq 0$, there exists $L=L(\alpha,\beta) > 0$ such that: if $\gamma : [a,b] \rightarrow X$ and $\sigma : [a^\prime, b^\prime] \rightarrow X$ are $(\alpha,\beta)$-quasi-geodesics with the same endpoints, then there exist partitions 
\begin{align*}
& a=t_0 < t_1 < \dots < t_{m+1} = b \\
& a^\prime=t_0^\prime < t_1^\prime < \dots < t_{m+1}^\prime = b^\prime
\end{align*}
where for all $0 \leq i \leq m$ 
\begin{align*}
\dist(\gamma(t_i), \sigma(t_i^\prime)) \leq L
\end{align*}
and either 
\begin{enumerate}
\item $\dist^{\Haus}( \gamma|_{[t_i,t_{i+1}]},  \sigma|_{[t_i^\prime,t_{i+1}^\prime]}) \leq L$ or
\item $\gamma|_{[t_i,t_{i+1}]},  \sigma|_{[t_i^\prime,t_{i+1}^\prime]} \subset \Nc(S';L)$ for some $S' \in \Sc'$.
\end{enumerate}
 \end{definition}
 
It is a well known result in CAT(0) geometry that co-compact CAT(0) spaces with isolated flats satisfy the relative fellow traveller property relative to totally geodesic flats. 
 \begin{proposition}[{\cite[Proposition 4.1.6]{HK2005}, \cite[Theorem 4.2]{EE2019}}]
 \label{prop:fellow_traveling_CAT(0)}  The space $(\wt{M},d_{\wt{M}})$ satisfies the relative fellow traveller property relative to $\Fc$.
\end{proposition}

The next lemma states that the above relative fellow traveller property carries over to a metric space quasi-isometric to $\wt{M}$. Indeed, this follows from \Cref{prop:fellow_traveling_CAT(0)} since we only need to pull back the quasi-geodesics to $\wt{M}$ using the quasi-isometry. 
 
 \begin{lemma}\label{prop:fellow_traveling} Suppose $(X,\dist)$ is a metric space, $\Psi: (\wt{M},\dist_{\wt{M}}) \to (X,\dist)$ is a quasi-isometry, and $\Sc'$ is a collection of subsets of $X$ with the following property: there exists a bijection $$\Fc \ni F \mapsto S'_F \in \Sc'$$ and $R>0$ such that $\Psi(F) \subset \Nc(S'_F;R)$ for any $F \in \Fc$. Then $(X,\dist)$ satisfies the relative fellow traveller property relative to $\Sc'$. 
\end{lemma} 
We will apply  this lemma later to show that $(\Cc,\hil)$ satisfies the relative fellow traveller property with respect to $\Sc$.

\subsection{Boundary quotients} 

 We recall the boundary quotients $\wt{M}(\infty)/{\sim}$ and $\partiali\Cc/{\sim}$ from the introduction. Let $\wt{M}(\infty) / {\sim}$ be the topological quotient induced by the equivalence relation $\sim$: $x,y \in\wt{M}(\infty)$ are equivalent if either $x=y$ or there exists an isometrically embedded Euclidean plane $F$ with $x,y \in F(\infty)$. Notice that conditions~\ref{item:a4} and~\ref{item:a6} imply that this is an equivalence relation. 
 
 Let $\partiali \Cc / {\sim}$ denote the analogous quotient of $\partiali\Cc$ using properly embedded simplices in $\Sc$ in the place of isometrically embedded flats. 

\begin{observation} 
\label{obs:bdry-quotient}
$\wt{M}(\infty) / {\sim}$ and $\partiali \Cc / {\sim}$ are compact and Hausdorff. 
\end{observation}

\begin{proof}Clearly both spaces are compact. Recall, $\wt{M}(\infty) / {\sim}$ is Hausdorff if and only if $R = \{ (x,y) \in \wt{M}(\infty)^2 : x \sim y\}$ is closed \cite[Proposition 1.4.4]{TTD2008}. So suppose that $( (x_n,y_n) )_{n \geq 1}$ is a sequence in $R$ which converges to some $(x,y) \in \wt{M}(\infty)^2$. Let $\dist_T$ be the Tits metric on $\wt{M}(\infty)$. Then Property~\ref{item:a5} implies that $\dist_T(x_n,y_n) \leq \pi$ for all $n$. Then, since the Tits metric is lower semi-continuous (see for instance~\cite[Chapter II, Theorem 4.11]{B1995})
\begin{align*}
\dist_T(x,y) \leq \liminf_{n \rightarrow \infty} \dist_T(x_n,y_n) \leq \pi.
\end{align*}
So by Property~\ref{item:a5}, either $x=y$ or there exists a flat $F \in \Fc$ such that $x,y \in F(\infty)$. Hence $(x,y) \in  R$. So $R$ is closed and $\wt{M}(\infty) / {\sim}$ is Hausdorff.

A similar argument using Property~\ref{item:b5} shows that $\partiali \Cc / {\sim}$  is Hausdorff. 
\end{proof}

\subsection{Boundary maps}\label{subsec:boundary-map} 

Using the {\v S}varc-Milnor lemma, there exists an $\rho$-equivariant quasi-isometry $\Phi : \wt{M} \rightarrow \Cc$ and using the ``connect the dot'' trick, see for instance~\cite[Appendix A]{BF1998}, we may assume that $\Phi$ is continuous.

Notice that $\wt{M} \cup (\wt{M}(\infty) / {\sim})$ and $\Cc \cup  (\partiali \Cc / {\sim})$ compactify $\wt{M}$ and $\Cc$ respectively in a natural way. In this subsection we will prove the following extension result for this compactification. 

\begin{theorem}\label{thm:bd_extensions} $\Phi : \wt{M} \rightarrow \Cc$ extends to a homeomorphism 
\begin{align*}
\wt{M}(\infty) / {\sim} \longrightarrow \partiali \Cc / {\sim}.
\end{align*}
\end{theorem}

\begin{remark}
\label{rem:generalizes_to_nonpos_curv}
A careful reading of the proof shows that the above theorem holds more generally: whenever  $M$ is a compact non-positively curved Riemannian manifold with isolated flats (i.e. $\wt{M}$ is relatively hyperbolic with respect to the set of all totally geodesic flats in $\wt{M}$) and $\rho:\pi_1(M) \to \PGL_d(\Rb)$ is a convex co-compact representation. 
\end{remark}

The proof of Theorem~\ref{thm:bd_extensions} will require a number of lemmas. 

\begin{lemma}\label{lem:defn_of_SF}
For every $F \in \Fc$, there exists a unique $S_F \in \Sc$ such that $\Phi(F)$ is contained in a bounded neighborhood of $S_F$. Further, the map 
\begin{align*}
F \in \Fc \mapsto S_F \in \Sc
\end{align*}
is a bijection. 
\end{lemma}

\begin{proof} This follows from Properties~\ref{item:a3}, \ref{item:b3}, \ref{item:a6} and Theorems~\ref{thm:rh_intersections_of_neighborhoods}, \ref{thm:rh_embeddings_of_flats}.
\end{proof}

\begin{lemma}
\label{lem:fellow_traveling_projective}
The space $(\Cc,\hil)$ satisfies the relative fellow traveller property relative to $\Sc$.
\end{lemma}
\begin{proof}
This is immediate from \Cref{lem:defn_of_SF} and \Cref{prop:fellow_traveling}.
\end{proof}

The next lemma requires the following well-known estimates for the distance functions $\hil$ and $\dist_{\wt M}$ on $\Omega$ and $\wt{M}$. For proofs, see for instance~\cite[Lemma 8.3]{C2009} and~\cite[Chapter I, Proposition 5.4]{B1995}.

 \begin{proposition}\label{prop:Crampons_dist_est}  If $\sigma_1:[0,T_1] \rightarrow \Omega$ and  $\sigma_2 : [0,T_2] \rightarrow\Omega$ are geodesics in $\Omega$ parametrizing projective line segments, then 
\begin{align*}
\hil( \sigma_1(\lambda T_1),\sigma_2(\lambda T_2)) \leq  \hil( \sigma_1(0),\sigma_2(0))+\hil(\sigma_1(T_1),\sigma_2(T_2))
\end{align*}
for all $\lambda \in [0,1]$. 
 \end{proposition}

\begin{proposition}\label{prop:CAT0_dist} If $\gamma_1: [0,T_1] \rightarrow \wt{M}$ and  $\gamma_2 : [0,T_2] \rightarrow \wt{M}$ are  geodesics in $\wt{M}$, then 
\begin{align*}
\dist_{\wt{M}}( \gamma_1(\lambda T_1),\gamma_2(\lambda T_2)) \leq (1-\lambda) \dist_{\wt{M}}( \gamma_1(0),\gamma_2(0))+\lambda\dist_{\wt{M}}( \gamma_1(T_1),\gamma_2(T_2))
\end{align*}
for all $\lambda \in [0,1]$. 
\end{proposition}

In the proofs of the next two lemmas, let $\alpha \geq 1$ and $\beta \geq 0$ denote the quasi-isometry parameters for $\Phi$.

\begin{lemma}\label{lem:limit_type} Suppose $(x_n)_{n \geq 1}$ is a sequence in $\wt{M}$ where $x_n \rightarrow \xi \in \wt{M}(\infty)$ and $\Phi(x_n)\rightarrow \eta \in \partiali\Cc$. Then for any $F \in \Fc$: 
\begin{align*}
\xi \in  F(\infty) \quad \text{if and only if} \quad \eta \in \partial S_F.
\end{align*}
\end{lemma}

\begin{proof} Fix some $x_0 \in \wt{M}$ and for each $n$ let $\gamma_n:[0,b_n] \rightarrow \wt{M}$ be the geodesic segment joining $x_0$ to $x_n$. Then $\gamma_n$ converges locally uniformly to a geodesic ray $\gamma: [0,\infty) \rightarrow \wt{M}$ with $\gamma(\infty) = \xi$. Next for each $n$ let $y_n := \Phi(x_n)$ and $\sigma_n:[0,b_n^\prime] \rightarrow \Cc$ be the geodesic parameterizing the line segment $[y_0,y_n]$. 

$(\Rightarrow)$: Suppose $\xi \in F(\infty)$ for some $F \in \Fc$. Then there exists $R_1 > 0$ such that $\gamma \subset \Nc(F;R_1)$. Then there exists $T_n \rightarrow \infty$ such that $\gamma_n|_{[0,T_n]} \subset \Nc(F;R_1)$. Next by Lemma~\ref{lem:defn_of_SF}, there exists $R_2 > 0$  such that 
\begin{align*}
\Phi(F) \subset \Nc(S_F; R_2).
\end{align*}
Then
\begin{align*}
\Phi \circ \gamma_n|_{[0,T_n]} \subset \Nc(S_F; R_3)
\end{align*}
where $R_3 := \alpha R_1+\beta+R_2$. 

By \Cref{lem:fellow_traveling_projective}, $(\Cc,\hil)$ satisfies the relative fellow traveller property relative to $\Sc$. Let  $L=L(\alpha,\beta)>0$ be the constant as in \Cref{defn:fellow_traveller} for $(\Cc,\hil)$ and $\Sc$. Then for $n \geq 0$,  let 
\begin{align*}
[s_n,t_n] \subset [0,b_n] \quad \text{and} \quad [s_n^\prime, t_n^\prime] \subset [0,b_n^\prime]
\end{align*}
be intervals in partitions satisfying \Cref{defn:fellow_traveller} with $T_n \in [s_n,t_n]$. 

By Theorem~\ref{thm:rh_intersections_of_neighborhoods} there exists $D > 0$ so that: if $S_1,S_2 \in \Sc$ and 
\begin{align*}
\diam \left(\Nc(S_1; L) \cap \Nc(S_2;R_3) \right)\geq \frac{1}{\alpha}D-\beta,
\end{align*}
then $S_1=S_2$. 

\medskip

\noindent \fbox{\emph{Claim 1:}} There exists $(T_n^\prime)_{n \geq 1}$ such that $\lim_{n \rightarrow \infty} T_n^\prime= \infty$ and 
\begin{align*}
\hil(\sigma_n(T_n^\prime), S_F) \leq L+R_3.
\end{align*} 

\medskip

We define $T_n^\prime \in [s_n^\prime,t_n^\prime]$ as follows. If 
\begin{align}
\label{eq:case_one}
\dist_\Omega^{\Haus}\left(\Phi \circ \gamma_n|_{[s_n,t_n]},\sigma_n|_{[s_n^\prime, t_n^\prime]}\right) \leq L,
\end{align}
then pick $T_n^\prime \in[s_n^\prime,t_n^\prime]$ such that 
\begin{align*}
\hil( \Phi \circ \gamma_n(T_n), \sigma_n(T_n^\prime)) \leq L.
\end{align*}
Notice that in this case
\begin{align*}
\hil(\sigma_n(T_n^\prime),S_F) \leq L+R_3.
\end{align*}
If the estimate in Equation~\eqref{eq:case_one} does not hold, then there exists some $S \in \Sc$ such that
\begin{align}
\label{eqn:both-near-S}
\Phi \circ \gamma_n|_{[s_n,t_n]}, \sigma_n|_{[s_n^\prime,t_n^\prime]} \subset \Nc(S;L).
\end{align}
In this case define
\begin{align*}
T_n^\prime := \left\{ \begin{array}{ll}
s_n^\prime & \text{ if } T_n - s_n \leq D\\
t_n^\prime & \text{otherwise}.
\end{array}
\right.
\end{align*}
If $T_n - s_n > D$, then
\begin{align*}
\Phi \circ \gamma_n|_{[s_n,T_n]} \subset  \Nc(S;L) \cap \Nc(S_F;R_3)
\end{align*}
and 
\begin{align*}
\diam \left( \Phi \circ \gamma_n|_{[s_n,T_n]} \right) \geq \frac{1}{\alpha} (T_n-s_n) - \beta > \frac{1}{\alpha}D-\beta.
\end{align*}
So $S_F = S$ and by Equation \eqref{eqn:both-near-S}
\begin{align*}
\sigma_n(T_n^\prime)=\sigma_n(t_n') \in  \Nc(S_F;L).
\end{align*}
If $T_n -s_n \leq D$, then 
\begin{align*}
\hil(\sigma_n(T_n^\prime),S_F) &= \hil(\sigma_n(s_n^\prime),S_F) \leq L+\hil( \Phi \circ \gamma_n(s_n), S_F) \\
& \leq L+R_3.
\end{align*}

Finally, by construction, $\lim_{n \rightarrow \infty} T_n^\prime= \infty$. Thus Claim 1 is established.

\medskip

\noindent \fbox{\emph{Claim 2:}} $\sigma_n|_{[0,T_n^\prime]} \subset \Nc(S_F;L+2R_3)$.

\medskip

Fix $t \in [0,T_n^\prime]$. Then by Proposition~\ref{prop:Crampons_dist_est} and Claim 1
\begin{align*}
\hil(\sigma_n(t), S_F) &\leq \hil( \sigma_n(0), S_F)+\hil( \sigma_n( T_n^\prime), S_F) \\
&\leq L+2R_3
\end{align*}
since $\sigma_n(0)=\Phi\circ\gamma_n(0)$. 

Now $\sigma_n$ converges locally uniformly to the geodesic $\sigma:[0,\infty) \rightarrow \Cc$ parametrizing the line segment $[y_0, \eta)$. By Claim 2
\begin{align*}
\sigma \subset \Nc(S_F; L+2R_3).
\end{align*}
So by Observation~\ref{obs:dist_est_and_faces} and Equation~\eqref{eq:faces_of_simplices}
\begin{align*}
\eta = \lim_{t \rightarrow \infty} \sigma(t) \in \partial S_F
\end{align*}

$(\Leftarrow)$: Suppose that $\eta \in \partial S_F$. A similar argument, where Proposition~\ref{prop:Crampons_dist_est} is replaced by Proposition~\ref{prop:CAT0_dist} (in the proof of Claim 2), shows that $\xi \in F(\infty)$. Also, we require the relative fellow traveller property for $\wt{M}$ which now follows from \Cref{prop:fellow_traveling_CAT(0)}.
\end{proof}

\begin{lemma}\label{lem:limits_exist} If  $\xi \in \wt{M}(\infty)\backslash \cup_{F \in \Fc} F(\infty)$, then
\begin{align*}
\lim_{p \in \wt{M}, ~p \rightarrow \xi} \Phi(p)
\end{align*}
exists in $\partiali\Cc / {\sim}$. 
\end{lemma}

\begin{proof} Fix $\xi \in \wt{M}(\infty) \backslash \cup_{F \in \Fc} F(\infty)$ and assume for a contradiction that
\begin{align*}
\lim_{p \in \wt{M}, ~p \rightarrow \xi} \Phi(p)
\end{align*}
does not exist in $\Cc \cup \partiali\Cc / {\sim}$. Then, since $\overline{\Cc}$ is compact, we can find sequences $(p_n)_{n \geq 1}, (q_n)_{n \geq 1}$ in $\wt{M}$ such that $p_n \rightarrow \xi$, $q_n \rightarrow \xi$, $\Phi(p_n) \rightarrow x \in\partiali\Cc$, $\Phi(q_n) \rightarrow y\in \partiali\Cc$, and $x \not\sim y$. Let $x_n = \Phi(p_n)$ and $y_n = \Phi(q_n)$. 

By the previous lemma 
\begin{align}
\label{eq:not_in_simplex}
x,y \in \partiali\Cc \backslash \cup_{S \in \Sc} \partial S.
\end{align}
Hence $(x,y) \subset \Cc$ by Property~\ref{item:b5}. Let $\gamma_n$ be the geodesic joining $p_n$ to $q_n$ in $\wt{M}$ and let $\sigma_n$ be a geodesic which parameterizes the line segment $[x_n,y_n]$. We can pick our parametrization so that $\sigma_n$ converges locally uniformly to a geodesic $\sigma: \Rb \rightarrow \Cc$ parametrizing the line segment $(x,y).$ We can further assume that $\lim_{t \to \infty} \sigma(t)=y$.

\medskip

\noindent \fbox{\emph{Claim:}} There exists some $S \in \Sc$ such that $\sigma \subset \Nc(S;L)$. 

\medskip

Let $L=L(\alpha,\beta)>0$ be the constant in \Cref{defn:fellow_traveller} for $(\Cc,\hil)$ and $\Sc$ (see \Cref{lem:fellow_traveling_projective}). Since $\Sc$ is closed and discrete in the local Hausdorff topology, there exist only finitely many $S \in \Sc$ with $\sigma(0) \in \Nc(S;L+1)$. So it is enough to show: for any $R > 0$ there exists some $S \in \Sc$ such that $\sigma|_{[-R,R]} \subset \Nc(S;L+1)$.

Fix $R > 0$. Since the sequences $(p_n)_{n \geq 1}$ and $(q_n)_{n \geq 1}$ both converge to $\xi$, for any fixed compact set $K \subset \wt{M}$ there exists $N > 0$ such that $\gamma_n \cap K = \emptyset$ for all $n \geq N$. Then there exists $N_0 > 0$ such that 
\begin{align*}
\min\left\{ \hil(\sigma_n(s), \Phi \circ \gamma_n) : s \in [-R,R] \right\} > L
\end{align*}
for all $n \geq N_0$. So by \Cref{defn:fellow_traveller}, for every $n \geq N_0$ there exists $S_n \in \Sc$ such that 
\begin{align*}
\sigma_n|_{[-R,R]} \subset \Nc(S_n;L).
\end{align*}
Then 
\begin{align*}
\sigma|_{[-R,R]} \subset \Nc(S_n;L+1)
\end{align*}
for $n$ sufficiently large. This proves the claim. 

Then by Observation~\ref{obs:dist_est_and_faces} and Equation~\eqref{eq:faces_of_simplices}
\begin{align*}
y = \lim_{t \rightarrow \infty} \sigma(t) \in \partial S
\end{align*}
which contradicts Equation~\eqref{eq:not_in_simplex}.
\end{proof}

\begin{lemma} $\Phi : \wt{M} \rightarrow \Cc$ extends to a continuous map 
\begin{align*}
\overline{\Phi}: \wt{M} \cup \left(\wt{M}(\infty) / {\sim} \right) \longrightarrow \Cc \cup \left( \partiali\Cc / {\sim} \right).
\end{align*}
\end{lemma}

\begin{proof} Define
\begin{align*}
\overline{\Phi}(\xi) = \left\{\begin{array}{ll}
\Phi(\xi) & \text{ if } \xi \in \wt{M} \\
\lim_{p \in \wt{M}, ~p \rightarrow \xi} \Phi(p) & \text{ if } \xi \in  \wt{M}(\infty) / {\sim}
\end{array}\right.
\end{align*}
Notice that Lemmas~\ref{lem:limit_type} and~\ref{lem:limits_exist} imply that $\overline{\Phi}$ is well defined. 

By definition, to show that $\overline{\Phi}$ is continuous it is enough to prove: if $(\xi_n)_{n \geq 1}$ is a sequence in $\wt{M}(\infty)$ converging to $\xi \in \wt{M}(\infty)$, then 
\begin{align*}
\overline{\Phi}([\xi]) = \lim_{n \rightarrow \infty} \overline{\Phi}([\xi_n]).
\end{align*}
For each $n$ fix a sequence $(x_{n,m})_{m \geq 1}$ in $\wt{M}$ with $\lim_{m \rightarrow \infty} x_{n,m} = \xi_n$. Using Lemmas~\ref{lem:limit_type},~\ref{lem:limits_exist} and passing to subsequences we can assume that 
\begin{align*}
\lim_{m \rightarrow \infty} \Phi(x_{n,m}) = \eta_n \in \overline{\Phi}([\xi_n]).
\end{align*}
Fix a metric $d_{\Pb}$ on $\overline{\Omega}$ that generates the standard topology. Then we can pick a sequence $(m_n)_{n \geq 1}$ such that 
\begin{align*}
\lim_{n \rightarrow \infty} x_{n,m_n} = \xi
\end{align*}
and
\begin{align*}
\lim_{n\rightarrow \infty}d_{\Pb}\left( \eta_n, \Phi(x_{n,m_n}) \right)=0.
\end{align*}
Then by Lemma~\ref{lem:limit_type}, ~\ref{lem:limits_exist} and our choice of $(m_n)_{n\geq 1}$ we have
\begin{equation*}
\overline{\Phi}([\xi]) = \lim_{n \rightarrow \infty} \Phi(x_{n,m_n}) = \lim_{n \rightarrow \infty} \eta_n = \lim_{n \rightarrow \infty} \overline{\Phi}([\xi_n]). \qedhere
\end{equation*} 

\end{proof}

\begin{lemma} $\overline{\Phi}$ induces a homeomorphism $ \wt{M}(\infty) / {\sim} \longrightarrow  \partiali\Cc / {\sim}$.\end{lemma}

\begin{proof} By definition $\overline{\Phi}$ maps $\wt{M}(\infty) / {\sim}$ into  $\partiali\Cc / {\sim}$. Since both spaces are compact and Hausdorff, it suffices to show that $\overline{\Phi}$ is onto and one-to-one. 

\emph{Onto:} Fix $[\eta] \in\partiali\Cc / {\sim}$. Then fix a sequence $(x_n)_{n \geq 1}$ in $\Cc$ with $x_n\rightarrow \eta$. Since $\Phi$ is a quasi-isometry there exists $y_n \in \wt{M}$ such that 
\begin{align*}
\sup_{n \geq 1} \hil(\Phi(y_n), x_n) <+\infty.
\end{align*}
Passing to subsequences we can suppose that $y_n\rightarrow \xi \in\wt{M}(\infty)$ and $\Phi(y_n) \rightarrow \eta^\prime \in\partiali\Cc$. Then $[\eta,\eta^\prime] \subset \partial \Omega$ by Observation~\ref{obs:dist_est_and_faces}. So Property~\ref{item:b5} implies that
\begin{align*}
\overline{\Phi}([\xi]) = [\eta^\prime]=[\eta].
\end{align*}

\emph{One-to-One:} Suppose for a contradiction that 
\begin{align*}
\overline{\Phi}([\xi_1]) =[\eta]=\overline{\Phi}([\xi_2])
\end{align*}
and $[\xi_1],[\xi_2] \in \wt{M}(\infty) /{\sim}$ are distinct. 

If $\eta \in \cup_{S \in \Sc} \partial S$, then Lemma~\ref{lem:limit_type} implies that $[\xi_1]=[\xi_2]$. So we must have $\eta \notin \cup_{S \in \Sc} \partial S$ and hence $[\eta]=\{\eta\}$.  Then \Cref{lem:limit_type} implies that $[\xi_1] \cup [\xi_2] \subset \wt{M}(\infty)\setminus \cup_{F \in \Fc} F(\infty)$.

Since $[\xi_1] \neq [\xi_2]$, Property~\ref{item:a5} implies that $d_T(\xi_1,\xi_2)=\infty > \pi$. So there exists a geodesic $\gamma :\Rb \rightarrow \wt{M}$ with 
\begin{align*}
\lim_{t \rightarrow \infty} \gamma(t) =\xi_1 \quad \text{and} \quad \lim_{t \rightarrow -\infty} \gamma(t) =\xi_2
\end{align*}
(see for instance~\cite[Chapter II, Theorem 4.11]{B1995}). Let $x_n:=\Phi(\gamma(n))$, $y_n:=\Phi(\gamma(-n))$, and $\sigma_n$ be a geodesic which parameterizes the line segment $[x_n,y_n]$. Then $\sigma_n$ leaves every compact subset of $\Cc$ since $x_n,y_n \rightarrow \eta$. Then arguing as in the proof of Lemma~\ref{lem:limits_exist} there exists a flat $F \in\Fc$ and $R >0$ such that  
\begin{align*}
\gamma \subset \Nc(F;R).
\end{align*}
Then $\xi_1,\xi_2 \in F(\infty)$ and so $[\xi_1]=[\xi_2]$. Hence we have a contradiction.
\end{proof}

\subsection{Minimality of the boundary action}
\label{subsec:minimality}

Using Theorem~\ref{thm:bd_extensions} we will prove the following. We remind the reader that we are still in the same setup as in the beginning of this \cref{sec:struct-non-geom}.

\begin{theorem}\label{thm:minimal} $\Gamma$ acts minimally on $\partiali\Cc$. \end{theorem}

We first construct certain projections to the faces of the properly embedded simplices in $\Sc$ (which are all two dimensional by  the remarks in Section \ref{subsec:struct-C}).

\begin{lemma}\label{lem:projection} Suppose $S \in \Sc$ has vertices $\{v_1,v_2,v_3\}$. Then there exists a finite index Abelian subgroup $A \leq \Stab_{\Gamma}(S)$  which fixes the vertices of $S$ and there exists a sequence $(a_n)_{n \geq 1}$ in $A$ such that 
\begin{enumerate}
\item $a_n \rightarrow T \in \Pb(\End(\Rb^d))$,
\item $T(\Omega) = (v_1,v_2)$,
\item $\Pb(\ker T) \cap \overline{\Cc} = \{ v_3\}$, 
\item $T(\partiali\Cc-\partial S) \subset (v_1,v_2)$. 
\end{enumerate}
\end{lemma}

\begin{proof} By the discussion in Section~\ref{subsec:struct-C}, $\Stab_{\Gamma}(S)$ is virtually Abelian and acts co-compactly on $S$. So there exists a finite index Abelian subgroup $A \leq \Stab_{\Gamma}(S)$ which fixes the vertices of $S$ and acts co-compactly on $S$. Then using Proposition~\ref{prop:dynamics_of_automorphisms_2} we can find a sequence $(a_n)_{n \geq 1}$ in $A$ such that $a_n \rightarrow T \in \Pb(\End(\Rb^d))$ where $T(\Omega) = (v_1,v_2)$. Fix a lift $(\overline{a}_n)_{n \geq 1}$ of $(a_n)_{n \geq 1}$ to $\GL_d(\Rb)$ and a lift $\overline{T}$ of $T$ to $\End(\Rb^d)$ such that $\overline{a}_n \rightarrow \overline{T}$ in $\End(\Rb^d)$.

Let $e_1,\dots,e_d$ be the standard basis of $\Rb^d$. By changing coordinates we can assume $[e_j] = v_j$ for $j \in\{1,2,3\}$. Then
\begin{align*}
\overline{a}_n = \begin{pmatrix} d_{n,1} & & & ^tu_{n,1} \\ & d_{n,2} & &  ^tu_{n,2} \\ & & d_{n,3} & ^tu_{n,3} \\ & & & C_n \end{pmatrix}
\end{align*}
where $d_{n,j} \in \Rb$, $u_{n,j} \in \Rb^{d-3}$, and $C_n \in\GL_{d-3}(\Rb)$. Further, since $T(\Omega)=(v_1,v_2)$ and $S \subset \Omega$,
\begin{align*}
\overline{T}=\lim_{n \rightarrow \infty} \overline{a}_n = \begin{pmatrix} d_{1} & & & ^tu_{1} \\ & d_{2} & &  ^tu_{2} \\ & & 0 & 0 \\ & & & 0 \end{pmatrix}
\end{align*}
where $d_1,d_2 \in\Rb$ are non-zero. Then $\Pb(\ker T) \cap\partial S = \{ v_3\}$. Since $\Pb(\ker T) \cap \Omega =\emptyset$, Properties~\ref{item:b4},~\ref{item:b5} imply that 
\begin{align*}
\Pb(\ker T) \cap \overline{\Cc} \subset \Pb(\ker T) \cap\partial S=\{v_3\}.
\end{align*}
So $T$ induces a continuous map $\partiali\Cc \backslash \{v_3\} \rightarrow [v_1,v_2]$ with $T(v_1)=v_1$ and $T(v_2)=v_2$. 

If $p \in T^{-1}(v_1) \cap (\partiali\Cc \backslash \{v_3\})$, then $[p,v_1] \subset T^{-1}(v_1)$. Since $T(\Omega) = (v_1,v_2)$, then $[p,v_1] \subset \partial \Omega$. So, by Properties~\ref{item:b4} and~\ref{item:b5}, $[p,v_1] \subset \partial S$ . Thus $T^{-1}(v_1)\cap (\partiali\Cc \backslash \{v_3\}) \subset \partial S$. The same argument shows that $T^{-1}(v_2) \cap (\partiali\Cc \backslash \{v_3\}) \subset \partial S$. Thus $T(\partiali\Cc-\partial S) \subset (v_1,v_2)$. 
\end{proof}

Let $\Ec \subset \partiali\Cc$ denote the set of extreme points of $\Omega$ in $\partiali\Cc$. Then by Property~\ref{item:b5} 
\begin{align*}
\Ec = \partiali\Cc - \cup_{S \in \Sc} (\partial S \backslash \Ec).
\end{align*}

\begin{lemma} If $E \subset \partiali\Cc$ is closed, non-empty, and $\Gamma$-invariant, then $\overline{\Ec} \subset E$. In particular, $\Gamma$ acts minimally on $\overline{\Ec}$.

\end{lemma}

\begin{proof} It suffices to fix $e \in \Ec$ and show that $e \in E$. Using Proposition~\ref{prop:dynamics_of_automorphisms_1} we can find a sequence $(g_n)_{n \geq 1}$ in $\Gamma$ where $g_n \rightarrow T_0 \in \Pb(\End(\Rb^d))$, ${\rm image}\, T_0 = e$, and $\Pb(\ker T_0) \cap \Omega = \emptyset$. 

Fix $x \in E$. ~If $x \notin \Pb(\ker T_0)$, then 
\begin{align*}
e = T_0(x) = \lim_{n \rightarrow \infty} g_nx \in E. 
\end{align*}
So suppose that $x \in \Pb(\ker T_0)$.

\medskip

\noindent \fbox{\emph{Claim:}} There exists a properly embedded simplex $S \in \Sc$ with $x \notin\partial S$.

\medskip

Otherwise, $x$ would be contained in every properly embedded simplex in $\Sc$. Then Property~\ref{item:b4} implies that $\Sc$ consists of a single 2-dimensional simplex $S$. Then $S$ is $\Gamma$-invariant and so Proposition~\ref{prop:cc_observations} part (4) implies that $\Cc_{\Omega}(\Gamma)=S$, which implies that $\Gamma \cong \pi_1(M)$ is virtually isomorphic to $\Zb^2$, which is impossible. This proves the claim. 

Using the claim, fix a properly embedded simplex $S \in \Sc$ with $x \notin\partial S$. Since $\Pb(\ker T_0) \cap \Omega = \emptyset$ and $x \in \Pb(\ker T_0)$, Properties~\ref{item:b4} and~\ref{item:b5} imply that $\partial S \cap \Pb(\ker T_0)=\emptyset$. Let $(a_n)_{n \geq 1}$ and $T$ be as in Lemma~\ref{lem:projection}. Then 
\begin{align*}
T(x) = \lim_{n \rightarrow \infty} a_nx \in E
\end{align*}
since $x \notin \Pb(\ker T)$. Then $T(x)  \in \partial S$ by our choice of $T$. So  $T(x)  \not \in \Pb(\ker T_0)$. Thus 
\begin{equation*}
e = T_0(T(x)) = \lim_{n \rightarrow \infty} g_nT(x)  \in E. \qedhere
\end{equation*}
\end{proof}

By the above lemma, it suffices to show that $\overline{\Ec}=\partiali \Cc$.  Since $\Sc$ is countable, the boundary of a simplex is closed, and 
\begin{align*}
\Ec = \partiali\Cc - \cup_{S \in \Sc} (\partial S \backslash \Ec) \supset \cap_{S \in \Sc} \partiali\Cc - \partial S,
\end{align*}
the Baire category theorem implies that $\Ec$ is dense in $\partiali\Cc$ if $\partiali\Cc - \partial S$ is dense in $\partiali\Cc$ for all $S \in \Sc$. The next two lemmas will prove that this is indeed the case, thus finishing the proof of Theorem \ref{thm:minimal}.

\begin{lemma} If $S \in \Sc$, then $\partiali\Cc - \partial S$ has at most two connected components. \end{lemma}

\begin{proof} Suppose not. Then there exist non-empty closed sets $K_1,K_2,K_3 \subset \partiali\Cc$ such that 
\begin{align*}
\partiali\Cc = K_1 \cup K_2 \cup K_3 \quad \text{and} \quad K_i \cap K_j =\partial S \text{ when } i \neq j. 
\end{align*}

Let $\pi_1 : \partiali\Cc \rightarrow (\partiali\Cc / {\sim})$ and $\pi_2 : \wt{M}(\infty) \rightarrow (\wt{M}(\infty) / {\sim})$ be the natural projections and let $f :(\wt{M}(\infty) / {\sim}) \rightarrow (\partiali\Cc / {\sim})$ be the homeomorphism in Theorem \ref{thm:bd_extensions}

Notice that $\pi_1(K_j)$ is closed since $K_j$ is compact, $\pi_1$ is continuous, and $\partiali\Cc / {\sim}$ is Hausdorff. So 
\begin{align*}
\wh{K}_j :=(f \circ \pi_2)^{-1}(\pi_1(K_j))
\end{align*}
 is closed. Further, if $F \in \Fc$ is the flat with $f([F(\infty)]) = [\partial S]$, then
\begin{align*}
\wt{M}(\infty) = \wh{K}_1 \cup \wh{K}_2 \cup \wh{K}_3 \quad \text{and} \quad \wh{K}_i \cap \wh{K}_j =F(\infty) \quad \text{when} \quad i \neq j. 
\end{align*}
So $\wt{M}(\infty) - F(\infty)$ has at least three connected components. However, $\wt{M}(\infty) \cong \mathbb{S}^2$ and $F(\infty)$ is an embedded simple closed curve, so the Jordan curve theorem says that $\wt{M}(\infty) - F(\infty)$ has two connected components. So we have a contradiction.
\end{proof}

\begin{lemma} If $S \in \Sc$, then $\partiali\Cc = \overline{\partiali\Cc - \partial S }$.
\end{lemma}

\begin{proof} Let $v_1,v_2,v_3 \in \partial S$ be the vertices of $S$. By symmetry, it suffices to show that $(v_1,v_2)$ is contained in $\overline{\partiali\Cc - \partial S}$. Let $A$, $(a_n)_{n \geq 1}$, and $T$ be as in Lemma~\ref{lem:projection}. Since $A \leq \Stab_{\Gamma}(S)$ and $T$ is a limit of elements in $A$, 
 \begin{equation*}
X: = T(\partiali \Cc - \partial S) \subset \overline{\partiali\Cc - \partial S}. 
 \end{equation*}
Thus the lemma reduces to showing that $X=(v_1,v_2)$. Lemma \ref{lem:projection} part (4) implies that $X \subset (v_1,v_2)$.

Since $A$ is Abelian, we have $a \circ T = T \circ a$ for every $a \in A$. Hence $A X = X$. Since $X$ is the continuous image of a space with two connected components, $X$ has at most two connected components. So there exists a finite index subgroup $A_1 \leq A$ which preserves each connected component of $X$. Since $A$ acts co-compactly on $S$,  $A_1$ also acts co-compactly on $S$ and hence also the boundary face $(v_1,v_2)$, see Observation~\ref{obs:cocompact-action-on-simplices}. Thus $X=T(\partiali \Cc - \partial S)=(v_1,v_2)$ which proves the lemma.
\end{proof}

\subsection{Examples where $\Gamma$ does not act minimally on $\partiali \Cc_{\Omega}(\Gamma)$} Notice that Theorem \ref{thm:minimal} is no longer true if $M$ is assumed to have $\Rb^3$ geometry or $\Rb \times \Hb^2$ geometry. In the former case, Proposition~\ref{prop:boring_examples} implies that $\Cc_\Omega(\Gamma)$ is a simplex and the vertices of this simplex are a closed $\Gamma$-invariant set. In the latter case, Proposition~\ref{prop:boring_examples} implies that $\Cc_\Omega(\Gamma)$ is a cone with a strictly convex base. Then the base of this cone is a closed $\Gamma$-invariant set.

\section{Rank one automorphisms in convex co-compact groups}
\label{sec:rank-one}

Following recent work of the first author~\cite{I2019} we define rank one automorphisms in convex co-compact subgroups, but first some remarks about proximal elements of $\PGL_d(\Rb)$. 

An element $g \in \PGL_d(\Rb)$ is called \emph{proximal} if $\ev_1(g)>\ev_2(g)$. In this case let $g^+ \in \Pb(\Rb^d)$ denote the eigenline corresponding to $\lambda_1(g)$ and let $H_g^- \subset \Rb^d$ be the  unique $g$-invariant linear hyperplane with $g^+ \oplus H_g^- = \Rb^d.$

An element $g \in \PGL_d(\Rb)$ is \emph{biproximal} if $g, g^{-1}$ are both proximal. In this case, define $g^- : =(g^{-1})^+$ and $H_g^+ := H_{g^{-1}}^-$. 

\begin{observation}\label{obs:iterates_of_proximal} If $g \in \PGL_d(\Rb)$ is proximal, then 
\begin{align*}
T_g := \lim_{n \rightarrow \infty} g^n
\end{align*}
exists in $\Pb(\End(\Rb^d))$. Moreover, ${ \rm image}\, T_g =g^+$ and $\ker T_g = H_g^-$. Hence 
\begin{align*}
g^+ = \lim_{n \rightarrow \infty} g^nx
\end{align*}
for all $x \in \Pb(\Rb^d) \setminus \Pb(H^-_g)$.
\end{observation}

\begin{definition}\cite{I2019}
\label{defn:rank-one-auto}
Suppose $\Omega \subset \Pb(\Rb^d)$ is a properly convex domain and $\Gamma \leq \Aut(\Omega)$ is convex co-compact. An element $g \in \Gamma$ is a \emph{rank one automorphism in $\Gamma$} if $g$ is biproximal and $(g^+,g^-) \subset \Omega$. 
\end{definition}

Rank one automorphisms were defined differently and more generally in~\cite{I2019}, but the next proposition shows that the two definitions are equivalent for convex co-compact subgroups (also see Section 6 and Appendix A of \cite{I2019}).

\begin{proposition}\label{prop:rank-one-auto} Suppose $\Omega \subset \Pb(\Rb^d)$ is a properly convex domain and $\Gamma \leq \Aut(\Omega)$ is convex co-compact. If $g$ is a rank one automorphism in $\Gamma$, then:
\begin{enumerate}
\item $g^\pm$ is a $\Cc^1$-smooth extreme point of $\Omega$ with $T_{g^\pm} \partial \Omega = \Pb(H_g^\pm)$. 
\item If $z \in \partiali\Cc_\Omega(\Gamma)$, then 
\begin{align*}
(g^+, z) \cup (z,g^-) \subset \Cc_\Omega(\Gamma).
\end{align*}
\item $\Pb(H_g^{\pm}) \cap \overline{\Cc_\Omega(\Gamma)}=\{g^{\pm}\}$ and so
\begin{align*}
g^{\pm}= \lim_{n \rightarrow \pm \infty} g^n z
\end{align*}
for all $z \in \overline{\Cc_\Omega(\Gamma)}\setminus \{g^{\mp}\}$.
\end{enumerate}
\end{proposition}

The following argument is essentially the proof of~\cite[Lemma 6.4]{I2019}.

\begin{proof} Let $T_g := \lim_{n \rightarrow \infty} g^n$. Then ${\rm image}\, T_g = g^+$ by Observation~\ref{obs:iterates_of_proximal}. If we apply Proposition~\ref{prop:dynamics_of_automorphisms_2} to $(g^n)_{n \geq 1}$ and a point $p \in (g^+,g^-)$, then $T_g(\Omega) = F_\Omega(g^+)$ and so $F_\Omega(g^+) = \{g^+\}$. So $g^+$ is an extreme point. 

Next we claim that $g^+$ is a $\Cc^1$-smooth point of $\partial \Omega$ and $T_{g^+} \partial \Omega = \Pb(H_g^+)$. Suppose not, then there exists a supporting hyperplane $H$ of $\Omega$ at $g^+$ with $H \neq \Pb(H_g^+)$. Fix $v \in H \setminus \Pb(H_g^+)$, then $g^{-n}(v) \rightarrow g^-$. By compactness, we can find a subsequence such that  $g^{-n_j}H$ converges to a projective hyperplane $H^\prime$. For every $j$, we have 
$$
g^+ = g^{-n_j}(g^+) \in g^{-n_j}H \quad \text{and} \quad g^{-n_j}(v) \in g^{-n_j} H. 
$$
Hence $g^+, g^- \in H^\prime$. Further, $H^\prime \cap \Omega = \emptyset$. So we have a contradiction. Thus $g^+$ is a $\Cc^1$-smooth point of $\partial \Omega$ with $T_{g^+} \partial \Omega = \Pb(H_g^+)$. 

By symmetry, $g^-$ is a $\Cc^1$-smooth extreme point of $\Omega$ with $T_{g^-} \partial \Omega = \Pb(H_g^-)$. Thus part (1) is true. 

For notational convenience let $\Cc :=\Cc_\Omega(\Gamma)$. 

\medskip 

\noindent \fbox{\emph{Claim:}} $\{ x \in \partiali \Cc : gx = x\} = \{ g^+, g^-\}$. 

\medskip

First notice that $\ip{g}$ has finite index in its centralizer $Z_{\Gamma}(g)$: the centralizer acts properly discontinuously on $(g^+,g^-)$ and the quotient $\ip{g} \backslash (g^+,g^-)$ is compact. Next apply Theorem~\ref{thm:center} to $A = \ip{g}$ to obtain a $Z_{\Gamma}(g)$-invariant linear subspace $V \subset \Rb^d$ where the quotient $Z_{\Gamma}(g) \backslash \Omega \cap \Pb(V)$ is compact. Since $Z_{\Gamma}(g)$ is virtually isomorphic to $\Zb$ and $\Omega \cap \Pb(V)$ is diffeomorphic to $\Rb^{\dim(V)-1}$, we must have $\dim \Pb(V) = 1$. Further, $g^+, g^- \in \Pb(V)$. So $\Omega \cap \Pb(V) = (g^+,g^-)$. Then, from the definition of $V$ in Theorem~\ref{thm:center}, we have 
\begin{align*}
\{ x \in \partiali \Cc : gx = x\} = \{ g^+, g^-\}.
\end{align*}
Thus the claim is true. 

We now prove part (2). By symmetry it is enough to assume that $[g^+,z] \subset \partiali\Cc$ for some $z \in \partiali \Cc \setminus\{ g^+\}$ and derive a contradiction. Let $e_1,\dots, e_d$ be the standard basis of $\Rb^d$. By changing coordinates, we can assume that $g^+ = [e_1]$, $H_g^+ \cap H_g^- = \Span\{e_2,\dots,e_{d-1}\}$, and $g^- = [e_d]$. By compactness, $z^\prime := \lim_{j \rightarrow \infty} g^{-n_j} z$ exists for some sequence $n_j \rightarrow \infty$. Then $[g^+,z^\prime] \subset \partiali\Cc$ and so $z^\prime \neq g^-$. Thus, by Observation~\ref{obs:iterates_of_proximal}, $z \in \Pb(H_g^+)$. Since $z \neq g^+$ and $g$ is biproximal, then 
$$
z^\prime = \lim_{j \rightarrow \infty} g^{-n_j} z \in \Pb(\Span\{e_2,\dots,e_{d}\})=\Pb(H_g^-).
$$ 
Thus $z^\prime \in \Pb(H_g^+ \cap H_g^-)$. Then $\partiali\Cc \cap \Pb(H_g^+ \cap H_g^-)$ is a non-empty closed convex $g$-invariant subset. So $g$ has a fixed point in $\partiali \Cc \cap \Pb(H_g^+ \cap H_g^-)$. But this contradicts the Claim. Thus part (2) is true. 

We now prove part (3). By Observation \ref{obs:iterates_of_proximal} and Proposition~\ref{prop:dynamics_of_automorphisms_1}, $\Pb(H_g^\pm) \cap \Omega = \emptyset$. So part (2) of this proposition implies that 
\begin{align*}
\Pb(H_g^{\pm}) \cap \overline{\Cc}=\{g^{\pm}\}.
\end{align*}
Then Observation \ref{obs:iterates_of_proximal} implies that 
\begin{align*}
g^{\pm}= \lim_{n \rightarrow \pm \infty} g^n z
\end{align*}
for all $z \in \overline{\Cc}\setminus \{g^\mp\}$. 
\end{proof}

\begin{figure}[b]
\includegraphics[scale=0.5]{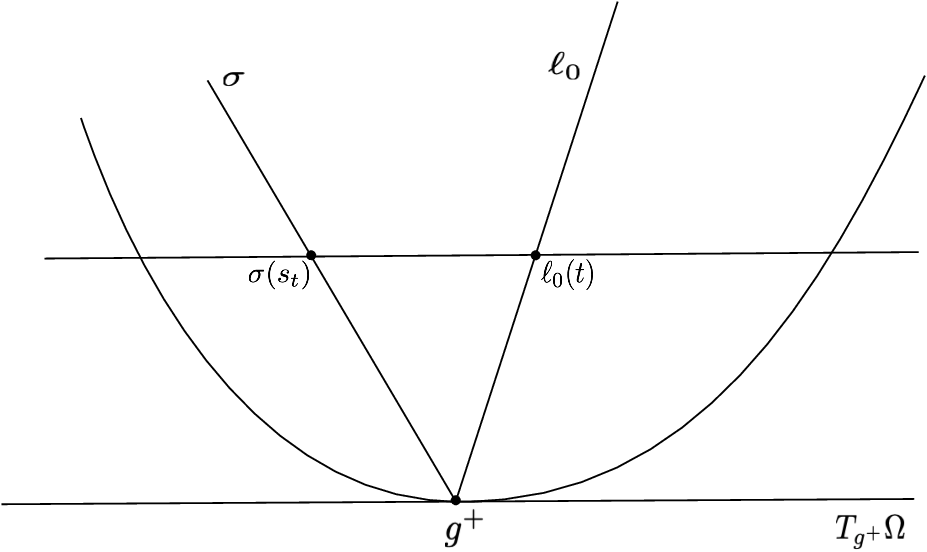}
\caption{Two asymptotic geodesic $\sigma$ and $\ell_0$ limiting to the $C^1$-smooth point $g^+$}
\centering
\label{fig:c1_point}
\end{figure}

\begin{proposition}
\label{prop:dist-estimate-rank-one-auto}
Suppose $\Omega \subset \Pb(\Rb^d)$ is a properly convex domain and $\Gamma \leq \Aut(\Omega)$ is convex co-compact. If $g$ is a rank one automorphism in $\Gamma$ and $\sigma:\Rb\to \Cc_{\Omega}(\Gamma)$ is a projective line geodesic with $\lim_{t \to \infty}\sigma(t)=g^+$, then there exists a unit speed parametrization $\ell:\Rb \to \Cc_\Omega(\Gamma)$ of $(g^-,g^+)$ with $\lim_{t \to \infty}\ell(t)=g^+$ such that 
\begin{align*}
\lim_{t \to \infty} \hil(\sigma(t), \ell(t))=0.
\end{align*}
\end{proposition}

\begin{proof} Let $\ell_0 : \Rb \rightarrow \Cc_\Omega(\Gamma)$ be any unit speed parameterization of $(g^-,g^+)$ with $\lim_{t \to \infty}\ell_0(t)=g^+$. Since $g^+$ is a $\Cc^1$-smooth point of $\Omega$, the definition of the Hilbert metric implies that for each $t \geq 0$ there exists $s_t \geq 0$ (see Figure \ref{fig:c1_point}) such that 
\begin{align*}
\lim_{t \rightarrow \infty} \dist_\Omega( \ell_0(t), \sigma(s_t)) = 0.
\end{align*}

We claim that $\lim_{t \rightarrow \infty} s_t - t$ exists. Pick sequences $(a_m)_{m \geq 1}, (b_m)_{m \geq 1}$ converging to infinity such that 
\begin{align*}
\limsup_{t \rightarrow \infty} s_t - t = \lim_{m \rightarrow \infty} s_{b_m}-b_m \quad \text{and} \quad \liminf_{t \rightarrow \infty} s_t - t = \lim_{m \rightarrow \infty} s_{a_m}-a_m.
\end{align*}
By replacing $(b_m)_{m \geq 1}$ with a subsequence we can suppose that $b_m > a_m+\dist_\Omega( \ell_0(0), \sigma(0))$ for all $m$. Then $s_{b_m} > s_{a_m}$ and 
\begin{align*}
0&= \lim_{m \rightarrow\infty} \dist_\Omega( \ell_0(b_m),\ell_0(a_m)) - (b_m-a_m) = \lim_{m \rightarrow\infty}\dist_\Omega( \sigma(s_{b_m}),\sigma(s_{a_m})) - (b_m-a_m)\\
& = \lim_{m \rightarrow\infty} (s_{b_m}-s_{a_m}) - (b_m-a_m) = \limsup_{t \rightarrow \infty} ~(s_t - t) - \liminf_{t \rightarrow \infty} ~(s_t - t).
\end{align*}
Thus the limit $\tau:=\lim_{t \rightarrow \infty} s_t - t$ exists. Then
\begin{align*}
0=\lim_{t \rightarrow \infty} \dist_\Omega( \ell_0(t), \sigma(s_t)) =\lim_{t \rightarrow \infty} \dist_\Omega( \ell_0(t), \sigma(t+\tau))= \lim_{t \rightarrow \infty} \dist_\Omega( \ell_0(t-\tau), \sigma(t))
\end{align*}
and so the geodesic $t \rightarrow \ell(t) := \ell_0(t - \tau)$ satisfies the lemma. 
\end{proof}

\section{The geodesic flow}

Suppose $\Omega \subset \Pb(\Rb^d)$ is a properly convex domain,  $\Gamma \leq \Aut(\Omega)$ is a convex co-compact subgroup, and  $\Cc : = \Cc_\Omega(\Gamma)$. The purpose of this section is to relate the dynamics of the geodesic flow to the dynamics of the boundary action. 

For the reader's convenience, we begin by recalling some notation from the introduction. Given $v \in T^1 \Omega$, we let $\gamma_v : \Rb \rightarrow \Omega$ denote the projective line geodesic with $\gamma_v^\prime(0)=v$. Then the geodesic flow on $T^1 \Omega$ is defined by $\phi_t(v) = \gamma_v^\prime(t)$. We will consider the $\Gamma$-invariant geodesic flow invariant subset of $T^1 \Omega$ defined by
\begin{align*}
\Gc_\Omega(\Gamma) :=\left\{ v \in T^1 \Omega : \pi_{\pm}(v) \in \partiali\Cc\right\}
\end{align*}
where 
$$
\pi_\pm(v) := \lim_{t \rightarrow \pm \infty} \gamma_v(t) \in \partial \Omega
$$
denotes the forward/backward limit points of $\gamma_v$. The geodesic flow descends to a flow on the quotient space $\Gamma \backslash \Gc_\Omega(\Gamma)$, which we will also denote by $\phi_t$. 

Also, recall that the geodesic flow $\phi_t$ on $\Gamma \backslash \Gc_{\Omega}(\Gamma)$ is said to be topologically transitive if for non-empty open sets $U,V \subset \Gamma \backslash \Gc_{\Omega}(\Gamma)$, there exists $t\in \Rb$ such that $\phi_t(U) \cap V \neq \emptyset.$  Our goal in this section is to prove the following theorem.

\begin{theorem}  
\label{thm:top-trans-equiv-minimal}
$\Gamma$ acts minimally on $\partiali\Cc$ if and only if the flow $\phi_t$ restricted to $\Gamma \backslash \Gc_\Omega(\Gamma)$ is topologically transitive. 
\end{theorem}

 The proof is split into the two subsequent subsections.

\subsection{Transitivity implies minimality} Suppose that the geodesic flow is topologically transitive on $\Gamma \backslash \Gc_\Omega(\Gamma)$. Then there exists $v_0 \in \Gc_\Omega(\Gamma)$ such that 
\begin{align}
\label{eqn:top_transitive_defn}
\overline{\Gamma \cdot \bigcup_{t \geq 0} \phi_t(v_0)} = \Gc_\Omega(\Gamma). 
\end{align}
Let $v_0^\pm := \pi_\pm(v_0) \in \partiali\Cc$. 

We first prove a lemma about the boundary maps $\pi_+$ and $\pi_-$ (which does not require the topological transitivity assumption). 

\begin{lemma}\label{lem: plus-minus boundary maps are onto}
The maps $\pi_+: \Gc_{\Omega}(\Gamma) \to \partiali \Cc$ and $\pi_-: \Gc_{\Omega}(\Gamma) \to \partiali \Cc$ are surjective open maps onto $\partiali \Cc$. 
\end{lemma}

\begin{proof}
By symmetry, it suffices to prove the lemma for $\pi:=\pi_+\big|_{\Gc_{\Omega}(\Gamma)}$. Observe that $\pi \circ \phi_t =\pi$ and, if $g \in \Gamma$, then $\pi \circ g = g \circ \pi. $

We first show that $\pi$ is an open map. Fix $v \in \Gc_{\Omega}(\Gamma)$. Set  $v_+:=\pi_+(v)$ and $v_-:=\pi_-(v)$. Then $(v_-,v_+) \subset \Cc$ which implies that there exists an open neighborhood $U_+$ of $v_+$ in $\partiali \Cc$ such that $(v_-,u) \subset \Cc$ for all $u \in U_+$. So for each $u \in U_+$, we can choose a vector $w_u \in \Gc_{\Omega}(\Gamma)$ such that $\pi(w_u)=u \in U_+$. Thus $\pi$ is an open map. 

We will now show that $\pi(\Gc_{\Omega}(\Gamma))=\partiali \Cc$. Suppose not. Then consider 
$$
C':=\partiali \Cc \setminus \pi(\Gc_{\Omega}(\Gamma)).
$$ 
Since $\pi$ is an open map, $C'$ is a closed $\Gamma$-invariant set. Fix $p_0 \in \Cc$ and $w \in C'$. Choose a sequence $(w_n)_{n \geq 1}$ in $[p_0,w)$ such that $\lim_{n \to \infty} w_n=w.$ Then, for each $n \geq 1$ there exists $h_n \in \Gamma$ such that  $R^\prime:=\sup_{n \in \Nb} \hil (h_np_0, w_n) <\infty. $   Passing to subsequences, we can assume that the limits 
$$
p_\infty:=\lim_{n \to \infty} h_n^{-1}p_0 \in \overline{\Cc} \quad \text{and} \quad w_\infty:=\lim_{n \to \infty} h_n^{-1}w \in \partiali\Cc
$$
exist. Then $(p_\infty, w_\infty) \subset \Cc$ since 
\begin{align*}
 \hil (h_n^{-1}w_n,p_0) \leq R^\prime.
\end{align*} 
Further, $p_\infty \in \partiali\Cc$ since
\begin{align*}
 \lim_{n \to \infty} \hil (h_n^{-1}p_0, p_0) \geq  \lim_{n \to \infty}\hil (h_n^{-1}p_0, h_n^{-1}w_n) - R^\prime = \lim_{n \to \infty} \hil (p_0,w_n)-R^\prime=\infty. 
 \end{align*}
Thus  $w_\infty \in \pi(\Gc_{\Omega}(\Gamma))$.  This is a contradiction since $w_\infty=\lim_{n \to \infty} h_n^{-1}w_n$ and $w_n \in C'$,  which is a $\Gamma$-invariant closed set.
\end{proof}

\begin{lemma}\label{lem: v0+ has dense orbit} $\overline{\Gamma \cdot v_0^+}= \partiali\Cc$. \end{lemma} 

\begin{proof} Fix $x \in \partiali\Cc$. By Lemma~\ref{lem: plus-minus boundary maps are onto} there exists  $w \in \Gc_\Omega(\Gamma)$ with $\pi_+(w) = x$. Next fix sequences $(t_n)_{n \geq 1}$ in $[0,\infty)$ and $(g_n)_{n \geq 1}$ in $\Gamma$ such that $g_n \phi_{t_n}(v_0) \rightarrow w$. Then $g_n \cdot v_0^+ \rightarrow x$. 
\end{proof} 

Let $\Ec \subset \partiali \Cc$ denote the set of extreme points of $\Omega$ in $\partiali\Cc$. 

\begin{lemma} $\overline{\Cc} = {\rm ConvHull}_{\overline{\Omega}} ( \Ec)$. In particular, $\Ec \neq \emptyset$. \end{lemma} 

\begin{proof} This follows from Proposition~\ref{prop:cc_observations} part (1). More precisely, for any $x \in \overline{\Cc}$ the proposition implies that 
$$
x \in {\rm ConvHull}_{\overline{\Omega}}\left( \partial F_\Omega(x) \right) = {\rm ConvHull}_{\overline{\Omega}}\left( \partial F_\Omega(x) \cap \partiali\Cc\right).
$$
Then, by induction on $\dim F_\Omega(x)$, we see that $x \in {\rm ConvHull}_{\overline{\Omega}} ( \Ec)$.
\end{proof}

\begin{lemma}\label{lem:forward asymp points with dense orbits are extreme} $v_0^+ \in \Ec$. \end{lemma} 

\begin{proof} Suppose not. Then there exist $x,y \in \partial F_\Omega(v_0^+)$ with $v_0^+ \in (x,y)$. Fix $e \in \Ec$. By Lemma~\ref{lem: plus-minus boundary maps are onto} there exists $w \in \Gc_\Omega(\Gamma)$ with $\pi_+(w) = e$. By Equation \eqref{eqn:top_transitive_defn}, we can find sequences $(t_n)_{n \geq 1}$ in $[0,\infty)$ and $(g_n)_{n \geq 1}$ in $\Gamma$ such that $t_n \rightarrow \infty$ and $g_n \phi_{t_n}(v_0) \rightarrow w$.

 Let $p_n : = g_n\gamma_{v_0}(t_n)$, $x_n := g_nx$, $y_n := g_n y$, and $z_n := g_n \gamma_{v_0}(0)$. Notice that $p_n \rightarrow \gamma_{w}(0) \in \Cc$ and by passing to a subsequence we can suppose that 
$$
x_n, y_n, z_n \rightarrow x_\infty, y_\infty, z_\infty \in \overline{\Cc}.
$$
By the definition of the Hilbert metric, see Observation~\ref{obs:dist_est_and_faces}, and the fact that $t_n \rightarrow \infty$ we have
$$
\lim_{n \rightarrow \infty} d_\Omega\Big( p_n, (x_n, z_n) \cup (y_n, z_n) \Big) = \lim_{n \rightarrow \infty} d_\Omega\Big(  \gamma_{v_0}(t_n), (x,  \gamma_{v_0}(0)) \cup (y, \gamma_{v_0}(0)) \Big) =\infty. 
$$ 
So 
$$
[x_\infty, y_\infty] \cup [y_\infty, z_\infty] \cup [z_\infty, x_\infty] \subset \partiali\Cc.
$$
Thus
$$
S := \Omega \cap  {\rm ConvHull}_{\overline{\Omega}}(\{x_\infty, y_\infty, z_\infty\})
$$
is a properly embedded simplex containing $\gamma_w(0)$. However, by construction,
 $$e=\pi_+(w) =\lim_{n \to\infty} \pi_+(g_n\gamma_{v_0}(t_n))=\lim_{n \to \infty} g_n v_0^+ \in (x_\infty, y_\infty)$$ 
 and so we have a contradiction. 
\end{proof}

\begin{lemma} If $x \in \partiali \Cc$, then $\overline{\Gamma \cdot x}= \partiali\Cc$. Hence, $\Gamma$ acts minimally on $\partiali\Cc$. \end{lemma} 

\begin{proof} By Lemma~\ref{lem: v0+ has dense orbit}, it is enough to consider the case when $x \neq v_0^+$ and show that $v_0^+ \in \overline{\Gamma \cdot x}$. 
\medskip

\noindent \fbox{\emph{Case 1:}} Assume $(x, v_0^+) \subset \Cc$. By Equation \eqref{eqn:top_transitive_defn}, there exist sequences $(t_n)_{n \geq 1}$ in $[0,\infty)$ and $(g_n)_{n \geq 1}$ in $\Gamma$ such that $t_n \rightarrow \infty$ and $g_n \phi_{t_n}(v_0) \rightarrow -v_0$. Notice that $g_n \cdot v_0^+ \rightarrow v_0^-$ and $g_n\gamma_{v_0}(0) \rightarrow v_0^+$. Passing to a subsequence we can suppose that $g_n \rightarrow T \in \Pb(\End(\Rb^d))$. Since $g_n\gamma_{v_0}(0)\rightarrow v_0^+$, Proposition~\ref{prop:dynamics_of_automorphisms_1} and Lemma~\ref{lem:forward asymp points with dense orbits are extreme} imply that 
$$
{\rm image} \, T \subset \Spanset F_{\Omega}(v_0^+) = v_0^+.
$$
Thus ${\rm image} \, T = v_0^+$. Since $g_n \cdot v_0^+ \rightarrow v_0^-$, Observation~\ref{obs:limits_of_maps} implies that  $v_0^+ \in \Pb(\ker T)$. Proposition~\ref{prop:dynamics_of_automorphisms_1} also implies that $\Omega \cap \Pb(\ker T) = \emptyset$. Then, since $(x, v_0^+) \subset \Cc$, we see that $x \notin \Pb(\ker T)$. Thus 
$$
\lim_{n \rightarrow \infty} g_n x = T(x) = v_0^+
$$
and hence $v_0^+ \in \overline{\Gamma \cdot x}$. 

\medskip

\noindent \fbox{\emph{Case 2:}}  Assume $[x,v_0^+] \subset \partiali\Cc$. By Lemma~\ref{lem: plus-minus boundary maps are onto} there exists $x^\prime \in \partiali\Cc \setminus\{x\}$ with $(x,x^\prime) \subset \Cc$. By Lemma~\ref{lem: v0+ has dense orbit} there exists a sequence $(h_m)_{m \geq 1}$ in $\Gamma$ with $h_m \cdot v_0^+ \rightarrow x^\prime$. Then $(x, h_m \cdot v_0^+) \subset \Cc$ for $m$ large. Fix such an $m$. Then $(h_m^{-1} \cdot x, v_0^+) \subset \Cc$. Then by Case 1 there exists a sequence $(g_n)_{n \geq 1}$ in $\Gamma$ with $\lim_{n \rightarrow \infty} g_n h_m^{-1}\cdot x = v_0^+$. Hence $v_0^+ \in \overline{\Gamma \cdot x}$. 
\end{proof}

\subsection{Minimality implies transitivity} Suppose $\Gamma$ acts minimally on $\partiali \Cc$. 

As in the previous section, let $\Ec \subset \partiali\Cc$ denote the set of extreme points of $\Omega$ in $\partiali\Cc$.  We first prove the following proposition about approximation of certain geodesics in $\Cc$ using rank one automorphisms (see \Cref{defn:rank-one-auto}).

\begin{proposition}
\label{prop:approx-by-periodic-orbits}
If $x_1, x_2 \in \Ec$ and $(x_1,x_2) \subset \Cc$, then there exists a sequence $(\psi_n)_{n \geq 1}$ of rank one automorphisms in $\Gamma$  such that $\lim_{n \to \infty} \psi_n^+  = x_1$ and $\lim_{n \to \infty}\psi_n^- = x_2$.
\end{proposition}

\begin{proof}
Fix $p \in \Cc$. By Proposition~\ref{prop:dynamics_of_automorphisms_2}, there exist sequence $(g_n)_{n \geq 1}$, $(h_n)_{n \geq 1}$ in $\Gamma$ such that $\lim_{n \to \infty} g_n p =x_1$ and $\lim_{n \to \infty} h_n p = x_2$. By passing to a subsequence, we can assume that 
\begin{align*}
T:=\lim_{n \to \infty} g_n \quad \text{and} \quad S:=\lim_{n \to \infty} h_n
\end{align*} 
exist in $\Pb(\End(\Rb^d))$. Passing to a further subsequence, we can assume that 
\begin{align*}
y_1:=\lim_{n \to \infty} g_n^{-1}p \quad \text{and} \quad y_2 := \lim_{n \to \infty} h_n^{-1}p
\end{align*}
exist in $\partiali \Cc$. Proposition \ref{prop:dynamics_of_automorphisms_1} implies that 
\begin{align*}
T(\Omega) = F_{\Omega}(x_1)=\{x_1\},  \quad \Pb(\ker T) \cap \Omega = \emptyset, \quad \text{and} \quad y_1 \in \Pb(\ker T),
\end{align*} 
while 
\begin{align*}
 S(\Omega) = F_{\Omega}(x_2)=\{x_2\},  \quad \Pb(\ker S) \cap \Omega = \emptyset, \quad \text{and} \quad y_2 \in \Pb(\ker S).
 \end{align*}

\medskip

\noindent \fbox{\emph{Claim 1:}} By possibly changing the sequences $g_n$ and $h_n$, we can assume that $(y_1,y_2) \subset \Cc$.

\medskip 

Since $\Gamma$ acts minimally on $\partiali \Cc$, there exist $\varphi_1,\varphi_2 \in \Gamma$ such that each $\varphi_j(y_j)$ is arbitrary close to $x_j$ for $j=1,2$. In particular, since $(x_1,x_2) \subset \Cc$, we may assume $(\varphi_1(y_1), \varphi_2(y_2)) \subset \Cc$. Consider the sequences 
\begin{align*}
g_n':=g_n \varphi_1^{-1} \quad \text{and} \quad h_n':=h_n \varphi_2^{-1}.
\end{align*} 
Then 
\begin{align*}
\lim_{n \to \infty} g_n'p = T(\varphi_1^{-1}(p)) = x_1
\end{align*} 
and $\lim_{n \to \infty} g_n'^{-1}p = \varphi_1(y_1)$. Likewise $\lim_{n \to \infty} h_n'p= x_2$ and $\lim_{n \to \infty} h_n'^{-1}p = \varphi_2(y_2)$.  Thus replacing $g_n$ and $h_n$ by $g_n'$ and $h_n'$ respectively establishes the claim. 

\medskip

\noindent \fbox{\emph{Claim 2:}}  After possibly passing to a subsequence, each $\psi_n:=g_nh_n^{-1}$ is a rank one automorphism in $\Gamma$, $\lim_{n \to \infty} \psi_n^+=x_1$, and $\lim_{n \to \infty} \psi_n^-=x_2$.

\medskip 

We can assume that $\psi:=\lim_{n \rightarrow \infty} \psi_n$ exists in $\Pb(\End(\Rb^d))$.   Since $(y_1,y_2) \subset \Cc$, $y_1 \in \Pb(\ker T)$, and $\Pb(\ker T) \cap \Omega = \emptyset$ we must have $y_2 \notin \Pb(\ker T)$. Likewise, since $y_2 \in \Pb(\ker S)$, we must have $y_1 \not \in \Pb(\ker S)$. Thus by Observation~\ref{obs:limits_of_maps}
\begin{equation*}
\lim_{n \to \infty} \psi_n (p)= T(y_2)=x_1 \quad \text{and} \quad \lim_{n \to \infty} \psi_n^{-1}(p)=S(y_1)=x_2.
\end{equation*}
Then by Proposition  \ref{prop:dynamics_of_automorphisms_1} 
\begin{align*}
  {\rm image}\, \psi \subset \Spanset \{ F_{\Omega}(x_1)\}=x_1,
\end{align*}    
$x_2 \in \Pb(\ker \psi)$,  and $\Pb(\ker \psi) \cap \Omega =\emptyset$. So ${\rm image}\, \psi=x_1$ and 
\begin{align*}
\ker \psi \oplus  {\rm image}\, \psi= \Rb^d
\end{align*} 
since $(x_1,x_2) \subset \Omega$. 

Let $e_1,\dots, e_d$ be the standard basis in $\Rb^d$. By changing coordinates we can assume that $[e_1]={\rm image}\, \psi$ and $\ker \psi = \Span\{e_2,\dots,e_d\}$. Let $\overline{\psi}_n \in \GL_d(\Rb)$ be the lift of $\psi_n$ with $(\overline{\psi}_n)_{11}=1$. Then
\begin{align*}
\overline{\psi}_n = \begin{pmatrix} 1 & ^tb_n \\ c_n & D_n \end{pmatrix} \in \GL_d(\Rb)
\end{align*}
where $b_n, c_n \in \Rb^{d-1}$ and $D_n \in \End(\Rb^{d-1})$. Since
\begin{align*}
\lim_{n \rightarrow \infty} \psi_n =\psi = \begin{bmatrix} 1 & 0 \\ 0 & 0 \end{bmatrix} 
\end{align*}
in $\Pb(\End(\Rb^d))$, the sequences $b_n,c_n, D_n$ all converge to zero. Then by the continuity of eigenvalues  
\begin{align*}
\lim_{n \rightarrow \infty} \frac{\lambda_2(\psi_n)}{\lambda_1(\psi_n)} = \lim_{n \rightarrow \infty} \frac{\lambda_2(\overline{\psi}_n)}{\lambda_1(\overline{\psi}_n)} =0.
\end{align*}
Thus $\psi_n$ is proximal for $n$ sufficiently large. 

Now fix a contractible neighborhood $U$ of $x_1$ such that $U \cap \Pb(\ker \psi) = \emptyset$. Then by Observation~\ref{obs:limits_of_maps}
\begin{align*}
\psi_n(\overline{U}) \subset U
\end{align*}
for $n$ sufficiently large. Thus $\psi_n^+ \in U$ when $n$ is large. Since $U$ was an arbitrary contractible neighborhood of $x_1$ we have $\lim_{n \to \infty}\psi_n^+ = x_1$. 

The same argument applied to $\psi_n^{-1}$ shows that $\psi_n^{-1}$ is proximal and $\lim_{n \to \infty} \psi_n^-=\lim_{n \to \infty} (\psi_n^{-1})^+ =x_2$. Finally, $(\psi_n^+, \psi_n^-) \subset \Omega$ for $n$ sufficiently large since $(x_1,x_2) \subset \Omega$. Thus $\psi_n$ is a rank one automorphism in $\Gamma$ for $n$ sufficiently large. 
\end{proof}

\begin{corollary}
\label{cor:approx-by-periodic-orbits}
If $g, h$ are rank one automorphisms in $\Gamma$, then there exists a sequence $(\psi_n)_{n \geq 1}$ of rank one automorphisms in $\Gamma$ such that $\lim_{n \to \infty} \psi_n^+ = g^+$ and $\lim_{n \to \infty} \psi_n^- = h^-$.
\end{corollary}
\begin{proof} Proposition \ref{prop:rank-one-auto} implies that $g^+, h^- \in \Ec$ and $(g^+,h^-) \subset \Omega$. The result then follows from Proposition \ref{prop:approx-by-periodic-orbits}. 
\end{proof}

Now we complete the proof that minimality of the boundary action  implies topological transitivity  of the geodesic flow.

\begin{lemma} The geodesic flow on $\Gamma \backslash \Gc_\Omega(\Gamma)$ is topologically transitive. \end{lemma}

\begin{proof}
It suffices to fix $\Gamma$-invariant non-empty open sets $U, V \subset  \Gc_{\Omega}(\Gamma)$ and show that there exists $T \in \Rb$ such that $\phi_T(U) \cap V \neq \emptyset$.

\medskip

\noindent  \fbox{\emph{Claim:}} There exists a rank one automorphism $g$ in  $\Gamma$ and for any such element $g$
\begin{align*}
\partiali \Cc=\overline{\Gamma \cdot g^+}.
\end{align*}

\medskip 

Since $\Gamma$ acts minimally on $\partiali\Cc$ it is enough to show that $\Gamma$ contains a rank one automorphism. Fix $x \in \partiali \Cc$ such that $\dim F_{\Omega}(x)$ is maximal. Then fix $y \in \Ec \setminus \overline{F_\Omega(x)}$ (recall that $\Ec$ is the set of extreme points of $\Omega$ in $\partiali \Cc$). If $z \in (y,x)$, then 
\begin{align*}
\ConvHull_{\overline{\Omega}}\left( \{ y\} \cup F_{\Omega}(x)\right) \subset \overline{F_\Omega(z)}.
\end{align*}
So by maximality of $\dim F_{\Omega}(x)$, we must have $z \in \Omega$. So $(y,x) \subset \Cc$.  

Since $\Gamma$ acts minimally on $\partiali \Cc$, there exists a sequence $(g_n)_{n \geq 1}$ in  $\Gamma$ such that 
$$
\lim_{n \to \infty} g_n y=x.
$$
Thus, for $n$ large enough, we have $y, g_n y \in  \Ec$ and  $(y,  g_n y) \subset \Omega$. Then $\Gamma$ contains a rank one automorphism by Proposition \ref{prop:approx-by-periodic-orbits}. This proves the claim. 

Given $x', y' \in \partiali\Cc$ with $(x',y') \subset \Cc$, let $L(x',y') \subset \Gc_\Omega(\Gamma)$ denote the set of vectors $v \in T^1 \Omega$ with $\pi_+(v)=y'$ and $\pi_-(v)=x'$. 

Fix a rank one automorphism $g \in \Gamma$. Since $\partiali \Cc=\overline{\Gamma \cdot g^+}$, there exists $h_1, h_2 \in \Gamma$ such that $L(h_1g^+, h_2 g^+) \cap U \neq \emptyset$. Then Corollary \ref{cor:approx-by-periodic-orbits} implies that there exists a rank one automorphism $\psi_1$ such that $L(\psi_1^-, \psi_1^+) \cap U \neq \emptyset$. Similarly, there exists a rank one automorphism $\psi_2$ such that $L(\psi_2^-,\psi_2^+) \cap V \neq \emptyset$. Since $\Gamma$ acts minimally on $\partiali \Cc$ and $V$ is $\Gamma$-invariant,  we may replace $\psi_2$ with a conjugate and assume that $\psi_1^- \neq \psi_2^+$. Then Proposition~\ref{prop:rank-one-auto} implies that $(\psi_1^-, \psi_2^+) \subset \Cc$. So there exists $w \in \Gc_{\Omega}(\Gamma)$ such that $\lim_{t \rightarrow -\infty} \gamma_w(t) = \psi_1^-$ and $\lim_{t \to \infty} \gamma_w(t)=\psi_2^+$. 

By Proposition \ref{prop:dist-estimate-rank-one-auto}, there exists a geodesic parametrization $\ell:\Rb \to \Cc$ of $(\psi_2^-, \psi_2^+)$  such that $\lim_{t \to \infty} \hil(\ell(t),\gamma_w(t))=0$. This, along with the fact that 
\begin{align*}
\{ \psi_2^n \ell(0): n \in \Zb\} \subset \ell(\Rb)=(\psi_2^-,\psi_2^+),
\end{align*}
implies that there exist $k_2 \in \Nb$ and $T_2 \in \Rb$ such that $\psi_2^{-k_2}\phi_{T_2}(w) \in V$. Since $V$ is $\Gamma$-invariant, $\phi_{T_2}(w) \in V$. Similarly there exists $T_1 \in \Rb$ such that $\phi_{T_1}(w) \in U$. Then $\phi_{T_2-T_1}(U) \cap V \neq \emptyset$.  Hence the geodesic flow is topologically transitive on $\Gamma \backslash \Gc_{\Omega}(\Gamma)$.
\end{proof}

\bibliographystyle{alpha}
\bibliography{geom}

\end{document}